\documentclass[11 pt]{article}
\usepackage{amssymb, amsmath, amsfonts,amsthm}
\usepackage{bbm} 
\usepackage[a4paper,vmargin={3.5cm,3.5cm},hmargin={2.5cm,2.5cm}]{geometry}
\usepackage{mathrsfs} 
\usepackage[usenames,dvipsnames,svgnames,table]{xcolor}
\usepackage{url} 
\usepackage{enumitem} 

\usepackage{lineno} 
\usepackage{mathtools} 
\mathtoolsset{showonlyrefs}
\usepackage{comment}

\usepackage[colorlinks,
            linkcolor=Blue,
            anchorcolor=red,
            citecolor=RoyalPurple,
            ]{hyperref} 

\usepackage{titlesec} 
\titleformat{\section}[block]
{\normalfont \large\bfseries}
{\thesection}{1.2em}{\bfseries}
\titleformat{\subsection}[block]
{\normalfont \normalsize \bfseries}
{\thesubsection}{0.9em}{\bfseries}
\titlespacing{\paragraph}{%
  0pt}{
  0.5\baselineskip}{
  1em}

\usepackage{setspace}
\makeatletter
\def\th@plain{%
  \thm@notefont{}
  \itshape 
}
\def\th@definition{%
  \thm@notefont{}
  \normalfont 
}
\makeatother

\newtheorem{theorem}{Theorem}[section]
\newtheorem{corollary}[theorem]{Corollary}
\newtheorem{lemma}[theorem]{Lemma}
\newtheorem{proposition}[theorem]{Proposition}

\newtheorem{definition}[theorem]{Definition}
\numberwithin{equation}{section}

\newcommand{\RR}{\mathbb{R}}
\newcommand{\UU}{\mathbb{U}} 
\newcommand{\ind}[1]{\mathbbm{1}_{\left\{ #1 \right\} } } 
\newcommand{\dd}{\mathrm{d}} 

\newcommand{\EE}{\mathbb{E}} 
\newcommand{\PP}{\mathbb{P}} 

\newcommand{\cI}{\mathcal{I}} 
 
\newcommand{\cF}{\mathcal{F}}
\newcommand{\cG}{\mathcal{G}} 
 
\newcommand{\cN}{\mathcal{N}}

\newcommand{\bP}{\mathbf{P}} 
 
\newcommand{\bE}{\mathbf{E}}

\newcommand{\bq}{\mathbf{q}}

\newcommand{\ta}{\mathtt{a}} 

\newcommand{\diag}{\mathrm{diag}} 
\newcommand{\JJ}{J} 
\newcommand{\TOne}{\mathcal{T}_1} 

\title{\textsc{From multitype branching Brownian motions\\ 
 to branching Markov additive processes}}
\author{
Yutao Liang\footnote{University of Chinese Academy of Sciences \& State Key Laboratory of Mathematical Sciences, Academy of Mathematics and Systems Science, Chinese Academy of Sciences; {liangyutao@amss.ac.cn}}, 
Yan-Xia Ren\footnote{LMAM School of Mathematical Sciences \& Center for
	Statistical Science, Peking
	University; {yxren@math.pku.edu.cn}},
Quan Shi\footnote{State Key Laboratory of Mathematical Sciences, Academy of Mathematics and Systems Science, Chinese Academy of Sciences; {quan.shi@amss.ac.cn}}, 
Fan Yang\footnote{School of Mathematical Sciences, Beijing University of Posts and Telecommunications; {fan-yang@bupt.edu.cn} } }

\date{\today}
\begin{document}
\maketitle
\begin{abstract}
We study a class of multitype branching L\'evy processes, where particles move according to type-dependent Lévy processes, switch types via an irreducible Markov chain, and branch according to type-dependent laws. This framework generalizes multitype branching Brownian motions.

Using techniques of Markov additive processes, we develop a spine decomposition. This approach further enables us to prove convergence results for the additive martingales and derivative martingales, and establish the existence and uniqueness of travelling wave solutions to the corresponding multitype FKPP equations. 
In particular, applying our results to the on-off branching Brownian motion model resolves several open problems posed by Blath et al.\ (2025).
\end{abstract}
\section{Introduction}\label{section: introduction}

\subsection{A motivating model: on-off branching Brownian motions}
The starting point of this work is an elegant two-type branching Brownian motion model introduced by Blath et al.\ \cite{BlaEtAl25}. Their model, so-called \emph{on-off branching Brownian motion} (on-off BBM), incorporates the biological concepts of \emph{dormancy} and \emph{seed banks}. In this  branching system, particles exist in either an \emph{active} or \emph{dormant} state, with each type following distinct branching rates, reproduction laws, and motion dynamics. The model also allows stochastic switching between the two states.
This framework provides a natural extension of the well-known (single-type) branching Brownian motion. As in the classical case, an on-off BBM  is intimately related to a two-type extension of Fisher–Kolmogorov–Petrovskii–Piskunov (FKPP) equation with the  dormancy feature. 
Through a careful and insightful analysis, \cite{BlaEtAl25} established two key results in the supercritical regime: the convergence of additive martingales and the existence of travelling waves. Moreover, the authors highlighted several open questions in \cite[Section 4]{BlaEtAl25}, 
including the following: 
\begin{itemize}
    \item The uniqueness of the travelling waves;
    \item A probabilistic representation of the travelling waves akin to the Lalley-Sellke \cite{LalleySellke87} construction;
    \item Convergence of martingales and the existence of travelling waves in the critical regime. 
\end{itemize}

A classical technique for addressing such questions is the spine decomposition, which is a seminal method  in branching process theory (see e.g.\ \cite{ShiBRW}). However, its application to the on-off BBMs is non-trivial.  Although the literature on single-type and multitype BBMs is vast, it usually assumes the underlying particle motion is identical for all types; this is crucially different from the on/off BBMs, where  the motions of particles are type-dependent. This seemingly minor extension introduces non-trivial technical difficulties. As Blath et al.\ note \cite{BlaEtAl25}, \emph{``the quadratic variation is truly probabilistic, making an application of the Girsanov Theorem difficult."} This inherent difficulty has thus posed a major obstacle to adapting the powerful spine decomposition technique.

One initial motivation for this work is to overcome this obstacle. We find that the essence of the problem is easier captured when we shift to a more general framework: branching particle systems with type-dependent Lévy processes. This perspective is naturally related to Markov additive processes (MAPs), which provide a clearer view of the underlying structure and the necessary mathematical tools to handle the heterogeneity. In this framework, we introduce a new spine decomposition. 
This allows us to provide a systematic approach to studying such branching systems and, in particular, resolve several open problems identified by Blath et al. \cite{BlaEtAl25} mentioned above.

\subsection{Multitype branching Lévy processes viewed as branching MAPs}
Our models are multitype branching L\'evy processes with a finite type space $\mathcal{I} = \{1, \ldots, \mathtt{d}\}$,  for some $\mathtt{d} \in \mathbb{N}$. Each particle  is assigned a type $i \in \mathcal{I}$ and moves in $\mathbb{R}$ according to a type-dependent L\'evy process. They randomly either branches, generating offspring that all inherit its current type, or  switches to a new type. All particles evolve independently of one another.
Specifically, the model is described as follows. 

\begin{itemize}
	\item(Movement) For each particle of type $i\in \cI$, its movement is governed by an $\mathbb{R}$-valued L\'evy process $(\chi_i(t), t\ge 0)$. For $\theta \ge 0$, when the exponential moment $\EE[ e^{-\theta \chi_i(1)}]$ is finite, we define the  \emph{Laplace exponent} $\phi_i$ such that 
    \begin{equation}
        \EE[ e^{-\theta \chi_i(t)}]  = \exp (\phi_i(\theta) t) <\infty, \qquad \forall t \ge 0. 
    \end{equation}
    
    The Laplace exponent $\phi_i$ is given by 
    \begin{equation}\label{eq:H-phi}
        \phi_i(\theta) = \frac{\sigma_i^2}{2} \theta^2 - \ta_i \theta +\int_{\RR} \left( e^{-\theta x} - 1 +x\theta \ind{|x|\le 1} \right) \Lambda_i (\dd x)<\infty, 
    \end{equation}
    with $\sigma_i^2\ge 0$, $\ta_i\in \RR$ and $\Lambda_i$ is a sigma-finite measure on $\RR\setminus \{0\}$. The  generator $\mathcal{A}_i f(x)$ is given by
    \begin{equation}\label{eqn:generator Ai}
        \mathcal{A}_i f(x) = \frac{\sigma_i^2}{2}  \partial_{xx} f(x) + \ta_i \partial_{x} f(x) + \int_{\mathbb{R}} \Big[  f(x+y) - f(x) - y  \ind{|y|<1}  \partial_{x} f(x) \Big] \,\Lambda_i (\dd y).
    \end{equation}
    We assume at least one of $\chi_i$ is non-trivial (non-constant).
    
	\item(Branching) At rate $\beta_i\ge 0$, a particle of type $i$ gives birth to a number of offspring according to $\mu_i:=(\mu_i (k), k\ge 0)$, and the children are all of type $i$ and located at the same place as the death point of the parent. 
    We suppose that each offspring has finite expectation
    \begin{equation}\label{eq:mu}
        m_i:= \sum_{k\ge 1} k \mu_i(k) \in [0,\infty), \qquad i\in \cI.  
    \end{equation}
    Note that, if $\beta_j=0$ for some $j\in \cI$, then a particle cannot split at type $j$ and the offspring law is meaningless; so in this case we just fix by convention $\mu_j(1)=1$. 

    \item(Switching types) 
    Each particle changes its type according to a continuous-time Markov chain $\Theta$ on $\cI$, with intensity matrix $Q = (q_{ij})_{i,j\in\cI}$, and possibly makes a jump at the time when it changes type. More precisely, at rate $q_{ij}\ge 0$,  a particle of type $i$  switches to the type $j\in \cI$, and at the same time makes a jump in space according to the law of a real-valued random variable $U_{ij}$, with convention that  $U_{ij}=0$ if $q_{ij} =0$, and that $U_{ii}=0$. 
    For $i\in \cI$, let $q_i = - q_{ii} = \sum_{j\ne i} q_{ij}$. When $\mathbb{E}[e^{-\theta U_{ij}}]$ is finite, we define 
    \begin{equation}\label{eq:H-G}
       G_{ij}(\theta) := \mathbb{E}[e^{-\theta U_{ij}}] <\infty, \mbox{ and } G(\theta):=(G_{ij}(\theta))_{i,j\in \cI}. 
    \end{equation} 
    We always suppose that
    \begin{equation}\label{eq:H-J}
        Q \text{ is irreducible}. 
    \end{equation}
    As the state space is finite, we deduce that $\Theta$ is positively recurrent and admits a unique invariant distribution denoted by $\pi = (\pi_i, i\in \cI)$.   
\end{itemize}

This model naturally  includes the multitype BBMs as prototypes. Processes combining Lévy behaviour with Markovian type-switching are known as the  Markov additive processes (MAPs), which provides a suitable framework for our analysis. Specifically, consider a c\`{a}dl\`{a}g process $(\chi(t))_{t\ge 0}$ on $\mathbb{R}$, and a right-continuous jump process $(\Theta(t))_{t\ge 0}$ on $\mathcal{I}$.  Assume that the joint process $(\chi,\Theta)$ is adapted to a filtration $(\mathcal{H}_t)_{t\ge 0}$ satisfying the usual conditions.

\begin{definition}[Markov additive process (MAP)]\label{def:MAP}
    We say that $(\chi(t), \Theta(t))_{t\ge 0}$ is a Markov additive process (MAP) if for all $s,t\ge 0$ and $i\in \mathcal{I}$, given $\{\Theta(t) = i\}$, the pair $(\chi(t+s)-\chi(t), \Theta(t+s))$ is independent of $\mathcal{H}_t$ and has the same law as $(\chi(s)-\chi(0), \Theta(s))$ given $\{\Theta(0) = i\}$.
\end{definition}

The theory of MAPs is well-established and has prominent applications in e.g.\ classical applied probabilistic
models for queues. 
We refer to \cite[Chapter XI]{Asmussen},  \cite[Chapter~2]{Iva2011} and \cite[Appendix]{DLK17} for detailed discussions.   
The following proposition is well-known, giving a standard representation for MAPs. 
\begin{proposition}\label{prop:rep of MAP}
    Let $(\chi(t), \Theta(t))_{t\ge0}$ be a MAP. Let $0=T_0<T_1<\dots$ be successive jump times of $\Theta$. For each $i,j\in \mathcal{I}$, there exist an i.i.d.\ sequence of random variables $(U_{ij}^{n}, n\ge 1)$, and an i.i.d.\ sequence of L\'{e}vy processes $(\chi_i^{n},n\ge 1)$, such that  these $\Theta,\chi_i^n, U_{ij}^n$ are independent, and that for $n\ge 0$, $t\in [T_n,T_{n+1}),$
    \[\chi(t) = \ind{n>0}\left(\chi(T_n-)+U_{ij}^{n}+\chi_j^{n}(t-T_n)\right),
    \]
    where $i=\Theta(T_n-)$ and $j = \Theta(T_n)$.
\end{proposition}
This representation has an intuitive interpretation:  the process $\Theta$ governs a time-dependent random environment.
When $\Theta$ is at state $i\in \cI$, the position $\chi$  evolves according to a copy of the L\'evy process $\chi_i$. Once $\Theta$ changes from $i$ to $j$, $\chi$ has an additional transitional jump $U_{ij}$. Then $\chi$ evolves according to a copy of $\chi_j$ until the next jump of $\Theta$. 
In this context, we say that $(\chi,\Theta)$ is a MAP \emph{associated with $((\chi_i)_{i\in \cI}, \Theta, (U_{ij})_{i,j\in \cI})$}, or equivalently, \emph{with characteristic triplet $((\phi_i)_{i\in \mathcal{I}}, Q, G)$}. For $\alpha\ge 0$, suppose that $\phi_i(\alpha)<\infty$, $G_{ij}(\alpha)<\infty$, $i,j\in \cI$.
Then it is well-known that, for $i,j\in \cI$, 
 \[
 \EE_{0,i}[e^{-\alpha \chi(t)}\ind{\Theta(t)=j}] = (e^{F(\alpha)t})_{ij},\qquad t\ge 0,
 \]
where the matrix exponent $F(\alpha)$ is given by 
\begin{equation}\label{eqn:MAP matrix exponent}
    F(\alpha) = \diag(\phi_i(\alpha))_{i\in \mathcal{I}}+Q\circ G(\alpha),
\end{equation}
where $Q\circ G(\alpha)$ denotes the entrywise matrix multiplication of $Q$ and $G(\alpha)$.

Therefore, we can  define our multitype branching L\'evy process model,   equivalently as a \emph{branching Markov additive process} on the space $\mathbb{R}\times \cI$ as follows. 
\begin{itemize}
    \item The spatial motion and type-switching are governed by a MAP $(\chi,\Theta)$ associated with $((\chi_i)_{i\in \cI}, \Theta, (U_{ij})_{i,j\in \cI})$; 
    \item When a particle is located at $(x,i)\in \mathbb{R}\times \cI$, it splits at rate $\beta_i$, giving birth to a random number of offsprings distributed  according to $\mu_i:=(\mu_i (k), k\ge 0)$, and the children are all initially located at the parent's position $(x,i)$. Each offspring particle evolves independently of the others. 
\end{itemize}

The construction of such a particle system may be carried out recursively in a genealogical manner, with the set of individuals indexed by the Ulam–Harris notation, thereby encoding the genealogy of the particles. More details are given in Section \ref{subsection_spine_decom}. 
Each particle indexed by $u$ is assigned with birth time $b_u$ (the global time) and lifetime $\eta_u$. 
For $t\ge 0$, let $\mathcal{N}_t$ be the collection of particles alive at time $t$, with $u\in \mathcal{N}_t$ if and only if $t\in [b_u, b_u +\eta_u)$), and $(X_u(t), \JJ_u(t))$ be the position and state of the particle $u\in \cN_t$. 
Denote by $\mathbf{P}_{x,i}$ the probability law of the particle systems with one initial particle of type $i\in \cI$ and position $x\in \RR$, and by $\mathbf{E}_{x,i}$ the expectation under $\mathbf{P}_{x,i}$.

\subsection{Main results}
We can now state our main results. We define
\[
dom := \{q\in \mathbb{R}\colon \phi_i(q)<\infty, G_{ij}(q)<\infty, \forall i,j\in \cI \}.
\]
Note that $0\in dom$. 
Let $\bar{\theta} := \sup dom\in [0,+\infty]$.
Then for each $\theta\in [0,\overline{\theta})$, we have $\phi_i(\theta)<\infty$ and $G_{ij}(\theta)<\infty$ for every $i,j\in\mathcal{I}$. 
The subsequent results primarily assume $\bar{\theta}>0$ and utilize positive values of $\theta$. If the domain of finiteness were confined to negative $\theta$, one could apply a sign change by considering $-X_u(t)$, thereby shifting the analysis back to the positive regime.

\begin{theorem}[Matrix exponent]\label{thm:matrix-exponent}
For $\theta \in [0, \bar{\theta})$ and $i,j\in \cI$, define 
	\begin{equation}
	M_{i,j}^{(\theta)}(t) := \bE_{0,i}\bigg[\sum_{u\in \cN_t} e^{-\theta X_u(t)} \ind{\JJ_u(t) = j} \bigg], \qquad t \ge 0.
\end{equation}
	Then for the matrix $M^{(\theta)}(t) := (M_{i,j}^{(\theta)}(t))_{i,j\in \mathcal{I}}$, we have 
	\begin{equation}
		M^{(\theta)}(t) = e^{t \mathcal{M}(\theta)}, \qquad t\ge 0,  
	\end{equation}
where the matrix exponent $\mathcal{M}(\theta)$ is given by
\begin{equation}\label{matric exponent}
    \mathcal{M}(\theta) = \diag (\phi_1(\theta), \ldots, \phi_{\mathtt{d}}(\theta)) + Q \circ G(\theta) 
	+ \diag (\beta_1 m_1-\beta_1, \ldots, \beta_{\mathtt{d}} m_{\mathtt{d}}-\beta_{\mathtt{d}}), 
\end{equation} 
where $Q \circ G(\theta)$ denotes the entry-wise multiplication of $Q$ and $G(\theta)$.
\end{theorem}

Let $\theta\in [0, \bar{\theta})$. By the Perron-Frobenius (PF) theorem (see \cite[Theorem 1.1]{seneta06} or  \cite[Theorem 8.3.4]{Horn_Johnson}), since $M^{(\theta)}(t)$ is a matrix with positive entries, the matrix $\mathcal{M}(\theta)$ has a PF eigenvalue $\lambda(\theta)$. 
This eigenvalue $\lambda(\theta)$ is real and larger than the real part of any other eigenvalues of $\mathcal{M}(\theta)$, and its associated right eigenvector $\vec{V}(\theta) = (V_i(\theta), i\in \cI)$ has strictly positive entries. Furthermore, because the entries of $\mathcal{M}(\theta)$ are infinitely differentiable on $(0, \bar{\theta})$, the PF eigenvalue $\lambda(\theta)$  and the corresponding eigenvector $\vec{V}(\theta)$ are also infinitely differentiable on this interval. The properties above of the PF eigenvalue are introduced in Lemma~\ref{lem:property of PF} and Lemma~\ref{lem:diff_of_V}.

We observe that the number of particles $\left(\sum_{u\in \cN_t} \ind{J_u(t)=1}, \ldots \sum_{u\in \cN_t} \ind{J_u(t)=\mathtt{d}}\right)_{t\ge 0}$,
forms a continuous-time multitype branching process (see e.g.\ \cite{Athreya1968}): for each particle of type $i\in \cI$, it is replaced by $k$ particles of type $i$ at rate $\beta_i\mu_i(k)$ for $k\ge 0$; whereas it is replaced by one type $j$ particle at rate $q_{ij}$ for $j\ne i$. When $\theta=0$, $\mathcal{M}(0)= Q+ \diag (\beta_1 m_1-\beta_1, \ldots, \beta_{\mathtt{d}} m_{\mathtt{d}}-\beta_{\mathtt{d}})$ serves as the generator of the first moment semigroup of $\left(\sum_{u\in \cN_t} \ind{J_u(t)=1}, \ldots \sum_{u\in \cN_t} \ind{J_u(t)=\mathtt{d}}\right)_{t\ge 0}$
and the PF eigenvalue $\lambda(0)$  determines the extinction behaviour with a phase transition. 
Denote the survival event by  \[
\mathscr{S} := \left\{ \#\mathcal{N}_t\ge 1,  \forall t\ge 0 \right\}.
\]
When $\lambda(0)>0$, the multitype branching process is called supercritical, and the survival probabilities $\mathbf{P}_{x,i}(\mathscr{S})>0$ for all $i\in \cI$. More precisely, define the extinction probabilities 
\[
\bq_i := \mathbf{P}_{x,i}(\exists t\ge 0, \#\mathcal{N}_t = 0) = 1 - \mathbf{P}_{x,i}(\mathscr{S}), \qquad i\in \cI. 
\]
By \cite[Theorem 2]{Athreya1968}, the vector $\vec{\bq} := (\bq_1, \dots, \bq_\mathtt{d})$ is the unique solution in $[0,1]^{\mathtt{d}} \setminus \{(1,\dots,1)\}$
to the equation:
\begin{equation}\label{eqn:surv_prob}
    \diag \bigg(\beta_1 \Big(\sum_{k\ge 0} \mu_i(k) s_1^k -s_1\Big), \ldots, \beta_{\mathtt{d}} \Big(\sum_{k\ge 0} \mu_i(k) s_{\mathtt{d}}^k -s_{\mathtt{d}}\Big)\bigg) + Q \, (s_1, \ldots , s_{\mathtt{d}})^{\top} = 0.
\end{equation}
The process also satisfies the  ``positive regularity'' and ``non-singularity'' in \cite[Theorem 2.7.1]{Harris63} and \cite[Corollary 1 of Theorem 2.7.2]{Harris63}. Therefore the unique solution $\vec{\bq}$ also lies in $[0,1)^{\mathtt{d}}$. Whereas when $\lambda(0)\le 0$, the extinction happens $\mathbf{P}_{x,i}$-a.s.\ $\forall x\in \mathbb{R}, i\in\cI$, i.e.\ $\mathbf{P}_{x,i}(\mathscr{S})=0$. 

In what follows, we will always work in the supercritical case 
$\lambda(0)>0$. 
Summarizing, we assume that 
\begin{enumerate}[label=(A1), ref=(A1),leftmargin=5.0em]
    \item $\bar{\theta} \in (0, +\infty]$ and  $\lambda(0)>0$. \label{A1}
\end{enumerate}

For each $\theta\in[0,\bar\theta)$, define
\begin{equation}
	W_{\theta} (t):= \sum_{u\in \cN_t} e^{-\left(\theta X_u(t)+\lambda(\theta) t\right)} V_{\JJ_u(t)}(\theta), \qquad t\ge 0.  
\end{equation}
By Theorem \ref{thm:matrix-exponent},  $(W_{\theta}(t))_{t\ge 0}$ is a $\mathbf{P}_{x,i}$-martingale. We call it the additive martingale.
Since $(W_\theta(t))_{t\ge 0}$ is a non-negative martingale, it converges a.s.\ to non-negative limit, denoted by $W_\theta(\infty)$. The following theorem gives the $\mathcal{L}^1$-convergence result for $(W_\theta(t))_{t\ge 0}$.

\begin{theorem}[Additive martingale]\label{Biggins_theorem}
    Let $\theta \in [0, \bar{\theta})$.
    On the event $\mathscr{S}$, the additive martingale $W_{\theta}(t)$ converges to $W_{\theta}(\infty)$ in $\mathcal{L}^{1}(\mathbf{P}_{x,i})$ for all $x \in \mathbb{R}$, $i \in \mathcal{I}$ if and only if $\theta \lambda'(\theta) < \lambda(\theta)$ and $\sum_{k\ge 1} (k \log k) \mu_j(k) < \infty$ for all $j \in \mathcal{I}$. 
    
    Moreover, when $W_{\theta}$ converges in $ \mathcal{L}^{1}(\mathbf{P}_{x,i})$, we have $\bP_{x,i}\left(\{W_{\theta}(\infty)>0\} \Delta \mathscr{S}\right) = 0$, where $\Delta$ represents the symmetric difference between two sets.
\end{theorem}

By the main Theorem of \cite{Kingman61}, the function $q \mapsto \lambda(q)$ is strictly convex on $(0, \bar{\theta})$. It follows that $q \mapsto \frac{\lambda(q)}{q}$ either attains its unique minimum in the interval $(0, \bar{\theta})$ or approaches the infimum at the boundaries $0$ or $\bar{\theta}$ (note that $\bar{\theta}$ could be $+\infty$). 
The next assumption specifies that the minimum is attained in the interior.
\begin{enumerate}[label=(A2), ref=(A2),leftmargin=5.0em]
    \item The function $q \mapsto \frac{\lambda(q)}{q}$ attains its unique minimum at some $\theta^*\in (0,\bar{\theta})$; or equivalently, there exists $\theta^*\in (0,\bar{\theta})$ such that $\lambda(\theta^*) = \theta^*\lambda'(\theta^*)$. \label{A2}
\end{enumerate}
We refer to $\theta = \theta^*$ as the critical regime;  $\frac{\lambda(\theta)}{\theta}< \lambda'(\theta)$, corresponding to $\theta \in (\theta^*, \bar{\theta})$, as the subcritical regime; and $\frac{\lambda(\theta)}{\theta} > \lambda'(\theta)$, corresponding to $\theta \in (0, \theta^*)$, as the supercritical regime.
Under the assumption \ref{A2}, the $\mathcal{L}^1$-convergence in Theorem~\ref{Biggins_theorem} only happens in the supercritical regime $\theta \in (0, \theta^*)$. 

As an application of Theorem \ref{Biggins_theorem}, we obtain the velocity of the leftmost particle.

\begin{corollary}
\label{thm:speed of Lt}
	Assume \ref{A1}, \ref{A2}, and $\sum_{k\ge 1} (k\log k)\mu_{j}(k) <+\infty$, for all $j \in \mathcal{I}$. 
    Then, for any $x\in \mathbb{R}$ and $i\in \mathcal{I}$, $\mathbf{P}_{x,i}$-a.s.\ on the non-extinction event $\mathscr{S}$,
	\[\lim_{t\to +\infty}\frac{\min_{u\in \mathcal{N}_t} X_u(t)}{t} =- \frac{\lambda(\theta^*)}{\theta^*}.
	\]
\end{corollary}

The main tool we use in the proof of Theorem \ref{Biggins_theorem} is the spine decomposition. The spine decomposition is a powerful technique in the study of branching processes. Typically, it contains two steps. The first step is a change of measure, and the second step is a deconstruction of the process in the new probability measure. For the details of the spine decomposition on classic branching systems, one can refer to \cite{Lyons97,ShiBRW} for applications on branching random walks (BRWs), \cite{KYPRIANOU04} for applications on branching Brownian motions (BBMs), and \cite{RS20} for applications on branching Markov processes (BMPs). In our model, we established a spine decomposition, which is introduced formally in  Section \ref{subsection_spine_decom}, and proved in Section~\ref{section:proof_spine}. 
This framework provides a powerful tool for  deeper analysis of the model, which we employ here to study the derivative martingale and the travelling waves.

For $\theta\in (0, \bar{\theta})$, define  \emph{the derivative  martingale} 
    \begin{equation}
        Z_{\theta}(t) := \sum_{u \in \mathcal{N}_t} e^{-\theta X_u(t) -\lambda(\theta) t}\left[V_{J_u(t)}(\theta)(X_u(t) +\lambda'(\theta)t)-V'_{J_u(t)}(\theta)\right], \qquad t\ge 0.
    \end{equation}
Note that we intuitively have $Z_{\theta}(t) := -\partial_{\theta}W_{\theta}(t)$. At criticality $\theta = \theta^*$,  the derivative martingale is a crucial object for understanding the fine structure of the frontier in branching systems; see, for example, \cite{Aidekon13,ShiBRW} 
for applications on branching random walks,  \cite{KYPRIANOU04,LalleySellke87} for applications on BBMs and \cite{Ren2014} for applications on multitype BBMs. 

Using the spine decomposition, we prove the convergence of the critical derivative martingale.

\begin{theorem}[Critical derivative martingale]\label{thm:asymptotTruncated}
    Assume \ref{A1} and \ref{A2}. Then 
    $Z_{\theta^{*}}(t)$ converges $\bP_{x,i}$-a.s.\@ to a non-negative 
	limit $Z_{\theta^*}(\infty)$, as $t\to \infty$. 
    If furthermore 
    \begin{equation}\label{eq:dev-ui}
 \sum_{k\ge1}k(\log k)^2 \mu_j(k)< \infty, \qquad \forall j\in\mathcal{I},
    \end{equation}
    holds, then we have $\bP_{x,i}\left(\{Z_{\theta^*} (\infty)>0\} \Delta  \mathscr{S}\right) = 0$.
\end{theorem}

The FKPP equation is a well-known reaction-diffusion equation that appears in population genetics and ecology. Its connection to branching processes was first studied via probabilistic tools in \cite{Mckean75}. Since then, it has been widely studied in probability, particularly in branching systems, see, for example, \cite{Champneys95}, \cite{KLMR12}, \cite{KYPRIANOU04}, \cite{LalleySellke87} and \cite{Ren2014}. In the study of FKPP equation, the existence, uniqueness and the asymptotic of the travelling wave solutions are particularly interested. One can find a ``common pattern'' in these works of FKPP equations. In the supercritical regime, the travelling wave solutions are often strongly connected to the limit of the additive martingale; while in the critical regime, it is related to the limit of the derivative martingale; while in the subcritical regime, there is no travelling wave solution.

For $i\in \cI$, let $g_i(s) := \sum_{k\ge 0} \mu_i(k) s^k$, $s\in [0,1]$, be the generating function of $\mu_i$, and $\mathcal{A}_i$ be the generator given in \eqref{eqn:generator Ai}. Our model is associated with the following FKPP type equation on $\RR_+\times\mathbb{R}\times \cI$: 
\begin{align}\label{FKPP}
\frac{\partial \mathbf{u}(t, x,i)}{\partial t}= & \, \mathcal{A}_i \mathbf{u}(t, x, i)+   \sum_{j\ne i} q_{ij} \left(\int_{\mathbb{R}} \mathbf{u}(t,x\!+ \!y, j)\mathbb{P}(U_{ij}\!\in \!\dd y) - \mathbf{u}(t,x,i)\right) \\
&\,+ \beta_i \big( g_i(\mathbf{u}(t,x,i)) - \mathbf{u}(t,x,i) \big).
\end{align}
Indeed, we prove in Proposition~\ref{prop:FKPP} that 
$\mathbf{u}(t,x,i) := \mathbf{E}_{x,i} \left[ \prod_{u\in \mathcal{N}_t} \mathbf{u}(0,X_u(t),J_u(t)) \right]$ gives 
a mild solution of \eqref{FKPP}, in the sense that it satisfies the following integral equation: 
\begin{align}\label{int-eq}
    \mathbf{u}(t, x,i)  = &\, \EE_{x,i} \left[\mathbf{u}(0,\xi_i(t),i) \right]  +\int_0^{t}  \mathbb{E}_{x,i}   \Bigg[\beta_i g_i \left(\mathbf{u}\left(s,\xi_i(t\!-\!s),i\right)\right)+  \\
    &  \sum_{j\ne i} q_{ij} \int_{\mathbb{R}} \mathbf{u}\left(s,y\!+\!\xi_i(t\!-\!s),j\right)\mathbb{P}(U_{ij}\in \dd y)  - (q_i +\beta_i) \mathbf{u}\left(s,\xi_i(t\!-\!s),i\right) \Bigg] \dd s.
\end{align}
If we look for constant solutions of the FKPP equation \eqref{FKPP} of the form $ \mathbf{u}(t, x,i)\equiv s_i\in [0,1]$ for $i\in \cI$, then the problem reduces to solving the equation \eqref{eqn:surv_prob}. As we have seen, it admits only two constant solutions in $[0,1]^{\mathtt{d}}$: $\vec{\bq}=(\bq_1, \ldots \bq_{\mathtt{d}})$ and $(1,\ldots, 1)$. 
We are interested in studying non-constant, travelling wave solutions that connect these two equilibria. To this end, let us introduce the following class of functions $\mathbb{R}\!\times\! \cI \to [0,1]$: 
\begin{align}\label{eq:TOne}
     &\TOne = \left\{f \;\middle|\;  \forall i\in \cI, x\mapsto f(x,i)  \text{ is non-decreasing in } x, \right.  \\
     &\hspace{6em}\left. \lim_{x\to -\infty} f(x,i)= \bq_i \text{ and }\lim_{x\to \infty} f(x,i)=1\right\}. 
\end{align}
A travelling wave solution is then defined as follows.
\begin{definition}[Travelling waves]
    Let $\rho\in \RR$ and $\Phi\in \TOne$. If $u(t,x,i):= \Phi(x - \rho t,i)$ is a solution of the FKPP equation \eqref{int-eq}, then we say that $\Phi$ is  \emph{a travelling wave solution with speed $\rho$}. 
\end{definition}

\begin{theorem}[Existence and uniqueness of travelling waves]\label{thm:TW} 
Assume \ref{A1}, \ref{A2}, and $\theta\in (0,\bar{\theta})$. 
\begin{enumerate}[label = (\roman*)]
    \item (Supercritical regime) If $\theta \lambda'(\theta) <\lambda(\theta)$ and $\sum_{k\geq 1}  (k\log k)\mu_j(k)<\infty$ for all $j\in\cI$, then the function $\Phi_{\theta}(x,i) = \bE_{x,i} \big[e^{-W_{\theta}(\infty)}\big] = \bE_{0,i}\left[ e^{-e^{-\theta x} W_{\theta}(\infty)} \right]$ is a travelling wave solution with speed $\rho_{\theta}= \frac{\lambda(\theta)}{\theta}>\frac{\lambda(\theta^*)}{\theta^*}$. 
    \item (Critical regime) If $\theta^* \lambda'(\theta^*) =\lambda(\theta^*)$ and $\sum_{k\geq 1}  k(\log k)^2 \mu_j(k) <\infty$ for all $j\in\cI$, then $\Phi_{\theta^*}(x,i) = \bE_{x,i} \big[e^{-Z_{\theta^*}(\infty)}\big] = \bE_{0,i}\left[ e^{-e^{-\theta x} Z_{\theta^*}(\infty)} \right]$ is a travelling wave solution with speed $\rho_{\theta^*} = \frac{\lambda(\theta^*)}{\theta^*}$. 
    \item (Subcritical regime) There are no travelling wave solutions with speed $\rho<\frac{\lambda(\theta^*)}{\theta^*}$. 
\end{enumerate} 
Suppose furthermore that the branching MAP is spectrally negative (that is the MAP $(\chi,\Theta)$ associated with $((\chi_i)_{i\in \cI}, \Theta, (U_{ij})_{i,j\in \cI})$  has no positive jumps), then the travelling wave solutions given in the supercritical and critical regimes are unique in $\mathcal{T}_1$.
\end{theorem}

\subsection{Related works}
Branching Brownian motion (BBM) and branching random walks (BRW) are canonical probabilistic models, with significant applications in statistical physics and population biology. These processes provide insights for understanding phenomena ranging from the extremes of log-correlated random fields and the structure of mean-field spin glasses to the dynamics described by diffusion-reaction equations. The classical case of a single particle type has been extensively studied. For a comprehensive treatment, we refer the reader to \cite{ShiBRW} and \cite{Bovier}.  

For branching L\'{e}vy process, there have been a recent increasing interests; see \cite{Profeta24,HJRS25,HRS24,RSZ25}, to list just a few. 
An extended model with infinite branching rate has been introduced by \cite{BerMal19} and the martingale convergence  has been studied in  \cite{Bertoin2018,MalleinShi2023}.

The theory of branching processes has been  generalized to the multitype case. Kesten and Stigum \cite{Kesten1966} established a key limit theorem for the discrete-time multitype branching processes, using the famous $L\log L$ moment condition for the reproduction law.  Athreya \cite{Athreya1968} proved this result for continuous-time settings. In the context of branching Brownian motion, 
Ren and Yang \cite{Ren2014} studied irreducible multitype cases, offering a probabilistic proof for the existence, uniqueness and asymptotic behaviours of the corresponding travelling wave solutions. More recently, Hou et al.\ \cite{Hou2024} proved that the extremal process of an irreducible multitype branching Brownian motion converges weakly to a cluster point process. For reducible cases, two-type branching Brownian motions have been investigated in \cite{BelloumMallein2021} and \cite{MaRen2024}. Note that, a common assumption in these works on multitype BBM has been that all particle types share the same underlying motion process. This is distinct from our setting, where the underlying processes differ among types.

A further generalization considers types in a general measurable space, which introduces significant new challenges, even without the presence of the spatial motion. There are significant new challenges. For the analysis of survival properties, we refer to \cite{AndDuc25} for an $L\log L$ condition in the model with countably many types, and to  \cite{MaiTou25} for a comprehensive treatment of the case with uncountably many types.

\subsection{Examples: two-type BBMs}
To make connections with results in \cite{BlaEtAl25, Champneys95}, let us exam in detail a specific case of our model: a two-type branching Brownian motion. Particles have two types $\{1,2\}$. For $i=1,2$, a type $i$ particle undergoes a Brownian motion with drift $\ta_i\in \mathbb{R}$ and variance $\sigma^2_i>0$; it branches at rate $\beta_i\ge 0$  into offspring of the same type at its current location, with offspring numbers distributed according to $\mu_i$.  Furthermore, at rate $q_1 > 0$ (for type $1$) or $q_2 > 0$ (for type $2$), the particle switches to the other type without displacement. All particles evolve independently. Let $m_i = \sum_{k\ge 1} \mu_i(k) k$, $i=1,2$. Assume $\beta_1 (m_1-1)+\beta_2(m_2-1)>0$ and $\sigma_1^2+\sigma_2^2>0$ to avoid degenerate cases.

The matrix exponent in Theorem \ref{thm:matrix-exponent} is given by
\begin{align}
    \mathcal{M}(\theta) =\begin{pmatrix}
        f_1(\theta) &  q_1 \\
         q_2 & f_2(\theta)\\
    \end{pmatrix}, 
\end{align}
where $f_i(\theta):= \frac{1}{2}\sigma_i^2\theta^2 -\ta_i \theta -q_i+\beta_i (m_i-1)$, $i= 1,2$. 
Then explicit calculation shows that the PF eigenvalue is 
\begin{equation*}
    \lambda(\theta) = \frac{1}{2} \left( f_1(\theta)+f_2(\theta) + \sqrt{(f_1(\theta)-f_2(\theta))^2 + 4  c_1 c_2} \right), 
\end{equation*}
with its corresponding PF eigenvector 
\begin{equation*}
V_1(\theta) =1, \quad V_2(\theta) = \frac{1}{q_1} (\lambda(\theta)  - f_1(\theta)).
\end{equation*}
Then $\mathcal{M}(\theta)$ and $\lambda(\theta)$ are finite for every $\theta\in (0,+\infty)$. 
We check by the explicit formula of $\lambda(\theta)$ that $\lambda(0+)>0$, and when $\theta \to +\infty$, $\frac{\lambda(\theta)}{\theta} \to +\infty$. 
There is a unique minimum in $(0,+\infty)$ of the function $q \mapsto \frac{\lambda(q)}{q}$, achieved at the solution $\theta^*$ of $\lambda'(\theta^*)\theta^* = \lambda(\theta^*)$.
Therefore assumptions \ref{A1} and \ref{A2} are satisfied. 

For $i\in \{1,2\}$, recall in \eqref{FKPP}, $g_i(s)=\sum_{k=0}^{\infty}\mu_i(k)s^{k}$. The corresponding FKPP equation is given by
\begin{equation}
    \left\{
    \begin{aligned}
        \frac{\partial u_1(t,x)}{\partial t} &= \frac{1}{2}\sigma_1^2\frac{\partial^2 u_1(t,x)}{\partial x^2} + \beta_1\big(g_1(u_1(t,x))-u_1(t,x)\big)-q_1u_1(t,x)+q_1u_2(t,x),\\
        \frac{\partial u_2(t,x)}{\partial t} &= \frac{1}{2}\sigma_2^2\frac{\partial^2 u_2(t,x)}{\partial x^2} + \beta_2\big(g_2(u_2(t,x))-u_2(t,x)\big)+q_2u_1(t,x)-q_2u_2(t,x).
    \end{aligned}
    \right.
\end{equation}

Having checked all the assumptions, we conclude that the statements in Theorem \ref{Biggins_theorem}, Corollary \ref{thm:speed of Lt}, Theorem~\ref{thm:asymptotTruncated} and Theorem~\ref{thm:TW} (including the uniqueness part) all hold for the two-type branching Brownian motion.

\paragraph{A connection with in \cite{BlaEtAl25}}
This formulation of two-type branching Brownian motions encompasses the \emph{on-off branching Brownian motion} models proposed by \cite{BlaEtAl25},  where  the two types are referred to as  ``active'' and ``dormant''. When $\beta_2 = \sigma_2 = 0$,  we obtain the Variant I model defined by \cite[(1.16)]{BlaEtAl25}; when $\beta_2 = \sigma_1 = 0$, we obtain the variant II model in  \cite[(1.19)]{BlaEtAl25}.

In \cite{BlaEtAl25}, although not explicitly stated, the last formula on page~10 has assumed $\sum k^2\mu_1(k)<\infty$. 
This already implies $\sum_{k} (k\log^2 k) \mu_j(k) <\infty$, for $j\in\{1,2\}$.
Compared to \cite{BlaEtAl25}, our work provides several extensions:
\begin{itemize}
    \item For the linear speed of the leftmost particle, the result matches \cite[Theorem 1.9]{BlaEtAl25}.
    \item When $\theta>\theta^*$, we deduce the same results on the additive martingale~\cite[Proposition 2.7]{BlaEtAl25} and travelling waves~\cite[Theorem 1.13]{BlaEtAl25}.
    \item When $0<\theta<\theta^*$, we have the same results on the additive martingale~\cite[Theorem 2.3 and Proposition 2.4]{BlaEtAl25} and travelling waves~\cite[Theorem 1.10]{BlaEtAl25}.
    \item When $\theta = \theta^*$,  we include new results on martingale convergence, and also the existence and uniqueness of the travelling wave solutions, answering the open questions in \cite{BlaEtAl25}.
\end{itemize}

\paragraph{A connection with \cite{Champneys95}}
The model in  \cite{Champneys95} is recovered by setting the drifts to zero ($a_i=0$) and requiring strictly positive branching rates ($\beta_i > 0$ for $i=1,2$).
In \cite{Champneys95}, Since the production law is given by $\mu_1(2) = \mu_2(2) = 1$, it also satisfies $\ind{{\beta}_{i}m_i >0}\sum_{k} (k\log^2 k) \mu_i(k) <\infty$ for $i\in \{1,2\}$. Our results are consistent with  \cite{Champneys95}:
\begin{itemize}
    \item For the linear speed of the leftmost particle, the result matches the speed in \cite[Theorem 1.41]{Champneys95}.
    \item When $\theta>\theta^*$, we include the non-existence of the travelling wave solutions.
     \item When $0<\theta<\theta^*$, the $\mathcal{L}^1$-convergence is consistent with \cite[Theorem 1.39]{Champneys95}. We also deduce the same existence and uniqueness results in \cite[Theorem 1.41]{Champneys95}.
    \item When $\theta = \theta^*$, we include new results on martingale convergence, and also the existence and uniqueness of the travelling wave solutions.
\end{itemize}

\subsection{Perspectives and further questions}
\paragraph{1. Generalization to non-local branching with infinite branching rate}
It is straightforward to extend our models to include non-local branching governed by a point process. We can further accommodate infinite branching rates as in \cite{BerMal19}, using similar approximation methods given there. This generalization would establish a connection with multitype growth-fragmentations \cite{DaSilva23,DaSilvaPardo24}. 

\paragraph{2. Necessary conditions for the non-triviality of the derivative martingale limit}
For the single-type BBM case, \eqref{eq:dev-ui} has been proven to be necessary and sufficient for the critical derivative martingale to converge to a non-trivial limit in \cite{YangRen11}; similar results are also known as the A\"id\'ekon--Chen condition for single-type BRWs \cite{Aidekon13,Chen15}, and for branching Lévy processes with infinite branching rates \cite{MalleinShi2023}. 
We conjecture that an analogous statement still holds for our multitype model, namely the condition \eqref{eq:dev-ui} is both necessary and sufficient. 
Indeed, we believe that  the methods from \cite{MalleinShi2023} should be applicable to our model; however, a full proof  would require a further study of the perpetual integral of a conditioned MAP.

\paragraph{3. Uniqueness of travelling waves}
In Theorem~\ref{thm:TW}, the assumption of spectrally negative jumps is used to establish uniqueness. We believe this assumption is redundant.  Extending the result to processes with two-sided jumps requires a more refined analysis. We expect that the techniques from \cite{AlsMal22}, which first establishes asymptotics of the travelling waves and then use these to prove  uniqueness, are applicable. 

\paragraph{4. Finer study of the leftmost position and extremal process} It would be interesting to develop a refined analysis of the leftmost position, e.g.\ by studying the convergence of $\min_{u\in\cN_t} X_u(t) + \frac{\lambda(\theta^*)}{\theta^*}t$ as well as the extremal process. 
This is closely related to the convergence of FKPP solutions to travelling waves and the precise information on the front propagation.  
The spine decomposition in the current work allows the methods in \cite{ShiBRW} to be adapted to the multitype setting.

\paragraph{5. Infinitely many types} 
A natural generalization is to consider more general type spaces, with countable or uncountable infinitely many types. 
 Comparing with finitely many types, this can lead to significantly different behaviours; for example, local extinction of each type would no longer be equivalent to global extinction of the entire population \cite{AndDuc25}.   This framework naturally connects to heterogeneous models of spatially dependent branching and movement.  Developing such an extension would require studies on general Markov additive processes, where the underlying
modulating processes are general Markov processes,  as opposed to Markov chains on a finite space; see \cite{MaiTou25} for recent development.

\subsection{Organization of the paper}
The remaining of this work is organized as follows. 
In Section~\ref{section: spine decomposition}, we first introduce preliminary tools for MAPs and establish the spine decomposition theorem. Then we apply the spine decomposition to deduce the $\mathcal{L}^1$-convergence of the additive martingales (Theorem~\ref{Biggins_theorem}) and determine the velocity of the leftmost particle (Corollary \ref{thm:speed of Lt}).
In Section~\ref{section: the derivative martingales}, we prove Theorem~\ref{thm:asymptotTruncated} the convergence of the critical derivative martingale and give a sufficient condition for the limit to be non-trivial. In Section~\ref{section: FKPP equations and travelling waves}, we use the martingale limits to give a probabilistic representation of the travelling wave solutions, and therefore prove the existence and uniqueness (Theorem \ref{thm:TW}).  
The proofs treat the supercritical and the critical regimes separately.
In Section~\ref{section:proof_spine}, we give detailed proofs of the spine decomposition theorem, including a ``Girsanov transformation'' of MAPs.

\section{The Spine decomposition}\label{section: spine decomposition}
\subsection{Preliminaries on MAPs}\label{sec:MAPs}
This section provides the necessary preliminaries on Markov additive processes (MAPs) defined as in Definition~\ref{def:MAP}. We refer to \cite[Chapter~2]{Iva2011} and \cite[Appendix]{DLK17} for detailed discussions on this topic.

Recall from Proposition~\ref{prop:rep of MAP} that a MAP  $(\chi,\Theta)$ is associated with a family of L\'evy processes $(\chi_i)_{i\in \cI}$ with respective Laplace exponents $(\phi_i)_{i\in \cI}$, a Markov chain with intensity matrix $\Theta$ and a family of random variables $(U_{ij})_{i,j\in \cI}$ with Laplace transform $G = (G_{ij})_{i,j\in \cI}$. 
The matrix exponent of the MAP $F(\alpha)$ is given by \eqref{eqn:MAP matrix exponent}.
Since $\Theta$ is irreducible, the entries of the matrix $e^{F(\alpha)t}$ are all strictly positive. From the Perron-Frobenius (PF) theory, $F(\alpha)$ admits a PF eigenvalue which is real and larger than the real part of any other eigenvalues; see for example \cite[ Theorem 1.1]{seneta06}. Based on the PF eigenvalue, we have the following law of large numbers.
\begin{proposition}[{ \cite[Propositions 2.13, 2.15 and Lemma 2.14]{Iva2011}}]\label{speed of MAP} 
    Suppose   $(\chi(t), \Theta(t))_{t\ge 0}$   is a MAP
    with matrix exponent $F(\alpha) = \diag(\phi_i(\alpha))_{i\in\mathcal{I}} + Q\circ G(\alpha)$. 
    Let $\gamma(\alpha)$ be the Perron-Frobenius eigenvalue of $F(\alpha)$. Then $\gamma(\alpha)$ is infinitely differentiable for $\alpha>0$, and one can define $\gamma'(0)$ as the right derivative. 
    Moreover, it holds that $\PP_{x,i}$-a.s.
    \[
    \lim_{t\to\infty}\frac{\chi(t)}{t} = -\gamma'(0) = \mathbb{E}_{\pi}[\chi(1)],
    \]
    for all $i\in\mathcal{I}$ and $x\in \RR$, where $\pi = (\pi_i)_{i\in\mathcal{I}}$ is the stationary distribution of Markov chain $\Theta$. Moreover, when the MAP is not degenerate (not constant) and $\mathbb{E}_\pi[\chi(1)]=0$, it holds that $\PP_{x,i}$-a.s. 
    \[\limsup_{t\to \infty}\chi(t) = +\infty,\ \liminf_{t\to \infty}\chi(t) = -\infty.\]
\end{proposition}

The following lemma for the perpetual integral of a MAP will be used to prove Theorem \ref{thm:L1_conv_Zb}.  For more studies on the perpetual integrals of L\'evy processes, see \cite{BDK,KolSav19}. 

\begin{lemma}\label{lem:integral-cond}
Let  $(\chi(t), \Theta(t))_{t\ge 0}$  be a MAP. Suppose  $\Theta$ is irreducible.
 Let $f\colon \RR_+ \times\cI \to \RR_+$ be a bounded non-negative measurable function. Suppose that $f$ is  eventually non-increasing and that
\begin{equation}\label{eq:f-int}
     \int_{\RR_+} f(x,k)  \dd x<\infty, \quad \forall k\in \cI. 
\end{equation}
Denote $\tau_0 := \inf\{t\ge 0: \chi(t) < 0\}$. Then we have, for any $x>0$ and $i\in \cI$,    
\begin{align*}
\EE_{x,i} \bigg[ \int_{\RR_+ }  f( \chi(s), \Theta(s))  \ind{\tau_0>s} \dd s \bigg] <\infty. 
\end{align*}
\end{lemma}
\begin{proof}
In this proof we denote by $C_n>0$ suitable constants for $n\ge 1$. 
We first treat the case when $\chi$ is non-lattice. 
Using \cite[Theorem~27]{DLK17}, we have
\begin{align}
     & \EE_{x,i} \bigg[ \int_{\RR_+}   f( \chi(s), \Theta(s))  \ind{\tau_0 >s} \dd s \bigg] \notag
     \\
     \le &\, C_1 \sum_{k\in \cI}  \sum_{j\in \cI}  \int_{y\ge 0} \int_{z\in [0,x]}
     R_{i,j}^+(\dd y) R_{k,j}^-(\dd z)  f(x+y-z, k) \notag\\
     = &\,  C_1\sum_{k\in \cI}  \sum_{j\in \cI}   \int_{z\in [0,x]} \int_{u\ge x-z} 
     R_{i,j}^+(\dd u) R_{k,j}^-(\dd z) f(u, k). \label{eq:lem:perp}
\end{align}
Here $(R_{i,j}^+)_{i,j\in\cI}$ are the potential measures for the ascending ladder height process associated with $(\chi, \Theta)$,  and $(R_{i,j}^-)_{i,j\in\cI}$ are those for the descending ladder height process; see \cite[Equation (27)]{DLK17} for more details.
When $\chi$ is lattice, then each L\'evy process $\chi_i$ is a compound Poisson process on the same lattice $\{ r\mathbb{Z}\}$ for some $r>0$, and the analysis is similar to discrete random walks.
In this case, the renewal measures can be expressed via the (strong or weak) ascending ladders, analogous to the renewal measure for discrete random walks. Let $(R_{i,j}^{>})_{i,j\in\cI}$ denote the renewal measures for the strongly ascending ladder height process of $\chi$, and let $({R}^{\le}_{i,j})_{i,j\in\cI}$ denote the renewal measures for the weakly descending ladder height process. 
Specifically, we define the renewal measures as follows. Let $T_0^{>}:=0$ (resp.\ $T_0^{\le} :=0$) and define recursively for $n\ge 0$ that $T_{n+1}^{>} := \inf \{t> T_n^{>}: \chi(t) > \chi(T_n^{>})\}$ (resp.\ $T_{n+1}^{\le} := \inf \{t> T_n^{\le}: \chi(t) \le \chi(T_n^{>})\}$). Let $H_n^{>} := \chi(T_n^{>})$ (resp.\ $H_n^{\le} := \chi(T_n^{\le})$) be the strongly ascending (resp.\ weakly descending) ladder heights. Then we define, for integer $m\ge 0$ and $k \le 0$,
\begin{equation}
\begin{aligned}
        R_{i,j}^{>}(\{rm\}) &= \frac{1}{q_j}\EE_{0,i}\left[\sum_{n=0}^{\infty}\ind{H_n^{>} = rm, \Theta(T_n^{>})=j}  \right],\\
        R_{i,j}^{\le}(\{rk\}) &= \frac{1}{q_j}\EE_{0,i}\left[\sum_{n=0}^{\infty}\ind{H_n^{\le} = rk, \Theta(T_n^{\le})=j}  \right].    
\end{aligned}
\end{equation}
Then by similar arguments as in the proof of \cite[Theorem~27]{DLK17}, we deduce a lattice analogue (c.f.\ \cite[Page 209 P3]{Spitzer} for the random walk case) of \eqref{eq:lem:perp}: 
\begin{align}
    &\hspace{1.5em}\EE_{x,i} \bigg[ \int_{\RR_+}   f( \chi(s), \Theta(s))  \ind{\tau_0 >s} \dd s \bigg]\\
    & \le C_1\sum_{k\in \cI}  \sum_{j\in \cI} \sum_{n=0}^{\infty} \sum_{m=0}^{\lfloor x/r\rfloor}  
    R_{i,j}^{>}(\{nr\}) {R}_{k,j}^{\le}(\{mr\}) f(x+ (n-m)r, k)\\
    &=C_1\sum_{k\in \cI}  \sum_{j\in \cI}   \int_{z\in [0,x]} \int_{u> x-z} R_{i,j}^{>}(\dd u) {R}_{k,j}^{\le}(\dd z) f(u, k). 
\end{align}
It follows from the Markov renewal theorem (see e.g.\ \cite{Lalley84} or \cite[Section~VII.4]{Asmussen}) that 
they are non-negative and there exists $A>0$ \footnote{Note that it holds in fact for any $A>0$ for the non-lattice case; for the lattice case, take $A$ larger than the span.} such that
\begin{align}
&R_{i,j}^+([nA,(n+1) A)), R_{i,j}^-([nA,(n+1) A)), \\
&R_{i,j}^{>}([nA,(n+1) A)), {R}_{i,j}^{\le}([nA,(n+1) A))\le C_2 A, \qquad \forall n\ge 0. 
\end{align}
For simplicity, we only consider non-lattice case in the rest of the proof. The proof also works for lattice case if we replace $R_{i,j}^+$ and $R_{i,j}^-$ with $R_{i,j}^{>}$ and ${R}_{i,j}^{\le}$.

Recall that $f$ is bounded and eventually non-increasing; replacing the value of $f$ on a compact set by a constant, we can find a function $\widehat{f}\ge f$, which is bounded and non-increasing such that 
\[
 \int_{\RR_+} \widehat{f}(x,k)  \dd x\le \int_{\RR_+} f(x,k)  \dd x + C_3. 
\]
As \eqref{eq:f-int} holds, we also have $\int_{\RR_+} \widehat{f}(x,k)  \dd x<\infty$. 
Since $\widehat{f}$ is bounded and non-increasing, we deduce that 
\begin{align*}
     \sum_{n\ge 0} \int_{u\in [nA, (n+1)A)} \widehat{f}(u, k) R_{i,j}^+(\dd u)   &\le  \sum_{n\ge 0} \widehat{f}(nA, k)  R_{i,j}^+([nA, (n+1)A)) \\
     &\le \sum_{n\ge 0} \widehat{f}(nA, k)  C_2 A \\
     &\le C_4+ C_2 \sum_{n\ge 1}  \int_{u\in [(n-1)A, nA)}  \widehat{f}(u, k) \dd u,
\end{align*}
It follows that 
\[
\int_{\RR_+} 
      \widehat{f}(u, k) R_{i,j}^+(\dd u) 
      \le C_5+ C_2 \int_{u\ge 0} 
      f(u, k) \dd u  . 
\]
Plugging this inequality  to \eqref{eq:lem:perp},  we have 
\begin{align*}
     & \EE_{x,i} \bigg[ \int_{\RR_+}   f( \chi(s), \Theta(s))  \ind{\tau_0 >s} \dd s \bigg] \\
     &\le C_1 \sum_{k\in \cI}  \sum_{j\in \cI}   \Big(C_5+ C_2 \int_{\RR_+} 
      f(u, k) \dd u \Big) \int_{z\in [0,x]}  R_{k,j}^-(\dd z)\\
     & \le C_6 \sum_{k\in \cI}  \Big( \int_{u\ge0} 
      f(u, k) \dd u +1 \Big) \sum_{j\in \cI} R_{k,j}^-([0,x]),
\end{align*}
which is finite under the assumption \eqref{eq:f-int}. 
\end{proof}

\subsection{The Matrix exponent and Perron-Frobenius eigenvalue}
In this section, we prove Theorem \ref{thm:matrix-exponent} and then introduce some properties of the Perron-Frobenius eigenvalue (PF eigenvalue).
\begin{proof}[Proof of Theorem \ref{thm:matrix-exponent}]
We decompose the process at the  first time when  the initial particle branches or changes its type; when the initial particle starts at type $i$, it has exponential distribution with parameter $\beta_i + q_i$, where  $q_i := \sum_{l\neq i}q_{il} = -q_{ii}$. By the branching property, we have, for every $t\ge 0$ 
    \begin{align}
        M_{i,j}^{(\theta)}(t) &= e^{-(\beta_i+q_i)t}e^{t\phi_i(\theta)} \ind{i=j}+\frac{\beta_i}{\beta_i+q_i}\int_{0}^{t}(\beta_i+q_i)e^{-(\beta_i+q_i)s}e^{s\phi_i(\theta)}m_iM_{i,j}^{(\theta)}(t-s)\dd s\\
        &+\frac{q_i}{\beta_i+q_i}\int_{0}^{t}(\beta_i+q_i)e^{-(\beta_i+q_i)s}e^{s\phi_i(\theta)} \sum_{l\ne i}\frac{q_{il}}{q_i}G_{il}(\theta) M_{l,j}^{(\theta)}(t-s)\dd s.
    \end{align}
     Let $D^{(\theta)}(t) := \diag \left(e^{t(\phi_i(\theta)-\beta_i-q_i)}\right)_{1\le i \le \mathtt{d}}$, $C:=\diag (\beta_im_i)_{1\le i \le \mathtt{d}}$, and $E:= \diag \left(q_i\right)_{1\le i \le \mathtt{d}}$. Then by the previous equation we have
    \begin{align}
        &\hspace{1.5em} M^{(\theta)}(t) \\
        &= D^{(\theta)}(t)+\int_{0}^{t}D^{(\theta)}(s)CM^{(\theta)}(t-s)\dd s+\int_{0}^{t}D^{(\theta)}(s)(Q\circ G(\theta) +E)M^{(\theta)}(t-s)\dd s\\
        &= D^{(\theta)}(t)+\int_{0}^{t}D^{(\theta)}(t-s)(C+Q\circ G(\theta)+E)M^{(\theta)}(s)\dd s.
    \end{align}
    Let $B^{(\theta)}:= \diag \left( (\phi_i(\theta)-\beta_i-q_i)\right)_{1\le i \le \mathtt{d}}$. Then
    \begin{align}
        &\hspace{1.5em} \int_0^t B^{(\theta)} M^{(\theta)}(s) \dd s\\
        &= \int_0^t B^{(\theta)} D^{(\theta)}(s) \dd s + \int_0^t B^{(\theta)} \int_{0}^{s}D^{(\theta)}(s-r)(C+Q\circ G(\theta)+E)M^{(\theta)}(r)\dd r \dd s\\
        &= \int_0^t d D^{(\theta)}(s) + \int_0^t \int_r^t B^{(\theta)} D^{(\theta)}(s-r) \dd s \; (C+Q\circ G(\theta)+E)M^{(\theta)}(r) \dd r\\
        &= D^{(\theta)}(t) - I + \int_0^t (D^{(\theta)}(t-r) - I)(C+Q\circ G(\theta)+E)M^{(\theta)}(r) \dd r\\
        &= M^{(\theta)}(t) - I - \int_0^t (C+Q\circ G(\theta)+E)M^{(\theta)}(r) \dd r.
    \end{align}
    Therefore,
    \begin{equation}
        M^{(\theta)}(t) = I + \int_0^t (B^{(\theta)}+C+Q\circ G(\theta)+E)M^{(\theta)}(s) \dd s=I + \int_0^t \mathcal{M(\theta)}M^{(\theta)}(s) \dd s,
    \end{equation}
    where we used  the definition of $\mathcal{M}(\theta)$  given by \eqref{matric exponent}.
    We deduce from this integral equation that $M^{(\theta)}(t) = e^{t\mathcal{M}(\theta)}, \forall t\ge 0$.
\end{proof}

As mentioned in Section~\ref{section: introduction}, we introduce some preliminary results on the Perron-Frobenius (PF) theory in the following Lemmas; see for example \cite[Theorem 1.1]{seneta06} or  \cite[Theorem 8.3.4]{Horn_Johnson} for proofs.

\begin{lemma}\label{lem:property of PF}
    Let $\mathcal{M}(\theta)$ be defined 
    by \eqref{matric exponent}
    and let $\lambda(\theta)$ denote the PF eigenvalue. Then we have:
    \begin{enumerate}
    \item $\lambda(\theta)$ is real and larger than the real part of any other eigenvalues of $\mathcal{M}(\theta)$. 
    \item The eigenvector w.r.t.\ $\lambda(\theta)$ is unique up to constant multiplication. 
    Let $\vec{Y}(\theta)=(Y_1(\theta), \ldots Y_{\mathtt{d}}(\theta))^{\top}$ and $\vec{V}(\theta) = (V_1(\theta), \ldots V_{\mathtt{d}}(\theta))^{\top}$ denote the corresponding left and right eigenvectors respectively. Without loss of generality,
    we normalize the two eigenvectors with $\pi^{\top} \vec{V}(\theta) = 1$ and $\vec{Y}^{\top}(\theta)\vec{V}(\theta) = 1$, where $\pi = (\pi_1, \ldots, \pi_\mathtt{d})^{\top}$ is the stationary distribution of Markov chain $J$.
    The entries of $\vec{Y}(\theta)$ and $\vec{V}(\theta)$ are all strictly positive. 
    \item If $\vec{U}(\theta)$ is a right eigenvector of $\mathcal{M}(\theta)$  with positive  entries, then we have $\vec{U}(\theta) = c\vec{V}(\theta)$ for some $c>0$.
\end{enumerate}
\end{lemma}
With the same method as in the proof of Proposition 2.13 of \cite{Iva2011}, we have the following lemma on differentiability of $\lambda(\theta)$ and $V(\theta)$.
\begin{lemma}\label{lem:diff_of_V}
    With the notation of Lemma \ref{lem:property of PF}, each of the functions $\theta\mapsto \lambda(\theta)$, $\vec{Y}(\theta)$ and $\vec{V}(\theta)$ is infinitely differentiable on  $(0, \bar{\theta})$. Moreover, we have 
    \begin{equation}\label{eqn:lambda'}
        \lambda'(\theta) = \vec{Y}^{\top}(\theta)\mathcal{M}'(\theta)\vec{V}(\theta),
    \end{equation}
    where $\mathcal{M}'(\theta)$ is the entry-wise derivative of $\mathcal{M}(\theta)$.
\end{lemma}

In the rest of the paper, we refer to $\lambda(\theta)$ as the PF eigenvalue of $\mathcal{M}(\theta)$ and $V(\theta)$ as the PF eigenvector, with the same normalization in Lemma \ref{lem:property of PF}.

\subsection{Measure change by the additive martingale}\label{subsection_spine_decom}

    Let us first give a formal construction of our branching MAP model as a marked Galton-Watson tree. Let $\mathbb{N}=\{1,2,3, \ldots\}$. 
	Denote the Ulam-Harris labels by
	\begin{equation}
	\mathbb{U}=\{\varnothing\} \cup \bigcup_{n \in \mathbb{N}}\mathbb{N}^{n}. 
	\end{equation}
	A planar tree $\tau$ is a subset  of $\mathbb{U}$ such that
	\begin{itemize}
		\item $\varnothing \in \tau$ (the ancestor);
		\item for $u, v \in \mathbb{U}$, $uv \in \tau$ implies $u \in \tau$;
		\item for all $u \in \tau,$ there exists $A_u \in\{1,2, \ldots\}$
		such that for $j \in \mathbb{N}$,  $uj \in \tau$ if and only if $1 \leq j \leq A_u$.
	\end{itemize}
    Let $\mathbb{T}$ be the space of planar trees. 
    We denote $v\preceq u$ if $v$ is an ancestor of $u$. Write $v\prec u$ if $v\preceq u$ and $v\ne u$. 

    For a branching MAP, we denote the initial particle by $\varnothing\in \UU$ and its birth time by $b_{\varnothing}= 0$. Say it is initially located at $x\in \RR$ with type $i\in \cI$. Then we define $(X_{\varnothing}, J_{\varnothing})$ to be a MAP with triplet $((\phi_i)_{i\in \cI}, Q,G)$ starting from $(x,i)$. Let us recursively construct a tree $\tau\in\mathbb{T}$ and assigned each particle $u\in \tau$ with a mark $(X_u,J_u,\eta_u,A_u)$ in the following way. 
    \begin{itemize}
    \item For each particle $u\in \tau$, given the birth time $b_u$ and $(X_u(t),J_u(t))_{t\ge b_u}$, its lifetime $\eta_u$ is specified by the first jump time of a non-homogeneous Poisson process with rate $\beta_{J_u(s -b_u)}, s\ge 0$. 
    At the death time $d_u:= b_u + \eta_u$, the particle $u$ gives birth to an offspring, with the number of children $A_u$ distributed according to the offspring law $\mu_{J_u(d_u)}$, all located at $X_u(d_u)$ with type $J_u(d_u)$. 
    \item For each child particle $v= uj$ with $j=1, \ldots A_u$, its birth time is $b_v:= d_u$. We construct a process $(X_v(t), J_v(t))_{t\ge b_v}$ such that $(X_v(s+b_v), J_v(s+b_v))_{s\ge 0}$ is a MAP with triplet $((\phi_i)_{i\in \cI}, Q,G)$ starting from $(X_v(b_v), J_v(b_v)) = (X_u(d_u), J_u(d_u))$, independent of the others. 
    \end{itemize}  
	We write $(\tau, M)$ as a shorthand for the marked Galton-Watson tree $\{(u,X_u,J_u,\eta_u,A_u):u\in\tau \}$. The state-space is $\mathcal{T} = \{(\tau,M):\tau\in\mathbb{T} \}$ and we denote its law by $\mathbf{P}_{x,i}$. For $s\leq t$ and $u\in \cN_t$, we still use $(X_u(s),J_u(s))$ to denote the position and type of particle $u$ or its ancestor at time $s$. For $t\ge 0$, define $\mathcal{F}_{t}$ to be the $\sigma$-algebra generated by 
	\begin{align*}
		\left\{\begin{array}{l}
			(u, J_u, A_u, \eta_{u},\{X_{u}(s): s \in[b_u, b_u+\eta_u]\}: u \in \tau \text { with } d_{u} \leq t) \text { and } \\
			(u, J_u, \{X_{u}(s): s \in[b_u, t]\}: u \in \tau \text { with } t \in[b_{u}, d_{u})): \tau \in \mathbb{T}
		\end{array}\right\}.
	\end{align*}
	Set  $\cF = \cup_{t\geq 0} \cF_t$.

    Assume \ref{A1} holds and let $\theta \in (0, \bar{\theta})$.  For $t\ge 0$, recall that $\cN_t$ denote the particles alive at time $t$. Then $u\in \cN_t$ if $b_u \le t <d_u$ and $(X_u(t), J_u(t))$ gives its position and type at time $t$. Recall that the additive martingale is 
    \begin{equation}
	    W_{\theta} (t):=  	\sum_{u\in \cN_t} e^{-\left(\theta X_u(t)+\lambda(\theta) t\right)} V_{\JJ_u(t)}(\theta), \qquad t\ge 0.  
    \end{equation}
    It follows from Theorem~\ref{thm:matrix-exponent} that $W_{\theta}(t)$ has constant expectation. The Markov property then implies that it is a non-negative martingale under $\mathbf{P}_{x,i}$ for all $x\in \mathbb{R}$ and $i\in\mathcal{I}$.

    For any $x\in \RR$, $i\in\cI$, define a new probability measure $\bP^{\theta}_{x,i}$ by
     \begin{equation}\label{eq:measure-change-P}
        \frac{\dd \bP^{\theta}_{x,i}}{\dd \bP_{x,i}}\bigg{|}_{\cF_t} := \frac{W_{\theta}(t)}{W_{\theta}(0)}.
    \end{equation}
    We next study the process under the new measure $\bP^{\theta}_{x,i}$ by  the spine approach. For simplicity, we first introduce the definition of the spine under the assumption that each particle has at least one child. However, all the results in this subsection also hold for the case allowing the possibility of no offspring when a particle dies. For this general case, we give the details at the end of Section \ref{section:subsec_proof_spine}.
    Specifically, a spine is a distinguished genealogical line of descendants from the ancestor. We write the spine as $\xi=\left\{\xi_t: t\geq 0\right\}$, where $\xi_t \in \tau$ is the label of the distinguished particle at time $t$. We write $u \in \xi$ if $u=\xi_t$ for some $t\geq0$. Let $O_u$ be the set of $u$'s children except the one in the spine. Now let
	\begin{equation*}
		\widetilde{\mathcal{T}}:=\{(\tau, M, \xi): \xi \subseteq \tau \in \mathbb{T}\}
	\end{equation*}
	be the space of marked trees in $\mathcal{T}$ with a distinguished spine $\xi$. 
    Recall that the filtration $(\mathcal{F}_{t})_{t\ge 0}$ contains all the information about the marked tree. Then we define for every $t\ge 0$ a new sigma-algebra by adding the information
	of the spine: 
	\begin{equation*}
    \widetilde{\mathcal{F}}_{t}:=\sigma\left(\mathcal{F}_{t},\left\{\xi_s: 0\leq s\leq t \right\}\right). 
	\end{equation*}
	Let $\widetilde{\mathcal{F}} := \cup_{t\geq 0} \widetilde{\mathcal{F}}_{t}$. 
    
    Let $J_{\xi} := (J_{\xi_t}(t), t\geq 0)$ denote the type process of the spine and $X_{\xi} := (X_{\xi_t}(t), t\geq 0)$ its spatial movement. We also use $X_{\xi}(t)$ and $J_{\xi}(t)$ as shorthand for $X_{\xi_t}(t)$ and $J_{\xi_t}(t)$, respectively.
    For $u\in\tau$, we use $|u|$ to denote the generation of $u$. Define $n_t: = |\xi_t|$, which tells us which generation the spine node is in, then $\mathbf{n}:=(n_t, t\geq 0)$ is the counting process of fission times along the spine. Define
    \begin{equation}\label{eqn:Sigma_fields_G}
        \begin{aligned}
        &\widetilde{\cG}_t := \sigma(J_{\xi}(s), X_{\xi}(s): 0\leq s\leq t),
        \quad \widehat{\cG}_t := \sigma(\widetilde{\cG}_t, \left\{\xi_s: 0\leq s\leq t \right\}, \{\eta_u, u\prec\xi_t\}),\\
        &\cG_t := \sigma(\widetilde{\cG}_t, \left\{\xi_s: 0\leq s\leq t \right\}, \{\eta_u, A_u, u\prec\xi_t\}).
    \end{aligned}
    \end{equation}
    
    We extend $\bP_{x,i}$ on $(\mathcal{T},\cF)$ to a probability measure $\widetilde{\bP}_{x,i}$ on $(\widetilde{\mathcal{T}},\widetilde{\cF})$ so that the spine is a single genealogical line of descendants chosen from the underlying tree. Since the children of a particle with type $j$ are all of type $j$, we assume that at each fission time along the spine we make a uniform choice among the offspring. Then for $u\in \tau$, we have
    \begin{equation}\label{eq:spine-unif}
        \widetilde{\bP}_{x,i}(\xi_t = u \mid \cF_t) = \ind{u\in \cN_t} \prod_{v\prec u} \frac{1}{A_v}, \qquad t\ge 0. 
    \end{equation}

    \begin{lemma}\label{lemma_martingale_zeta}
    For $t\ge 0$, define
    \begin{equation}
            \zeta_t := \sum_{u\in\cN_t} \left(\prod_{v\prec u} A_v\right) e^{-\theta X_u(t) - \lambda(\theta) t} V_{J_u(t)}(\theta) \ind{\xi_t = u}. 
    \end{equation}
        Then the process $(\zeta_t, t\geq 0)$ is a $\widetilde{\bP}_{x,i}$-martingale with respect to $\{\widetilde{\cF}_t, t\geq 0\}$. 
    \end{lemma}
    The proof of Lemma~\ref{lemma_martingale_zeta} is postponed to Section~\ref{section:proof_spine}. 
    Now define a probability measure $\widetilde{\bP}^{\theta}_{x,i}$ on $(\widetilde{\mathcal{T}}, \widetilde{\cF})$ by 
     \begin{equation}\label{def_tildeQ}
        \frac{\dd \widetilde{\bP}^\theta_{x,i}}{\dd \widetilde{\bP}_{x,i}} \bigg{|}_{\widetilde{\cF}_t} := \frac{\zeta_t}{\zeta_0}, \qquad t\ge 0. 
    \end{equation}
    By \eqref{eq:spine-unif}, we deduce that  
    \begin{equation}
    \widetilde{\mathbf{E}}_{x,i}[\zeta_t \ind{\xi_t = u} \mid \cF_t] 
    =   e^{-\theta X_u(t) - \lambda(\theta) t} V_{J_u(t)}(\theta) \ind{u\in \cN_t}. 
    \end{equation}
    It follows that  
    $\widetilde{\mathbf{E}}_{x,i}[\zeta_t \mid \cF_t] 
    = \sum_{u\in\cN_t}  e^{-\theta X_u(t) - \lambda(\theta) t} V_{J_u(t)}(\theta) = W_{\theta}(t)$ and thus the projection of $\widetilde{\bP}^{\theta}_{x,i}$ on $\cF$ is $\bP^{\theta}_{x,i}$. 
    Consequently, we also deduce that
        \begin{equation}\label{equation_choose_spine}
            \widetilde{\bP}^{\theta}_{x,i}(\xi_t = u \mid \cF_t)  
            = \frac{e^{-\theta X_u(t)- \lambda(\theta) t} V_{J_u(t)}(\theta)\ind{u\in \cN_t}}{ W_{\theta}(t)}.
        \end{equation}
    Indeed, for any $B\in \cF_t$, we have by \eqref{def_tildeQ} and \eqref{eq:spine-unif} that
        \begin{align}
            \widetilde{\bP}^{\theta}_{x,i}[ \xi_t=u; B] &= \frac{1}{e^{-\theta x}V_{i}(\theta)}\widetilde{\bE}_{x,i} \left[ \prod_{v\prec u} A_v e^{-\theta X_u(t) - \lambda(\theta) t} V_{J_u(t)}(\theta) \ind{u\in\cN_t} \mathbb{1}_{B}\right]\\
            &= \frac{1}{e^{-\theta x}V_{i}(\theta)}\bE_{x,i} \left[ e^{-\theta X_u(t) - \lambda(\theta) t} V_{J_u(t)}(\theta) \ind{u\in\cN_t} \mathbb{1}_{B}\right].
        \end{align}
        Since the projection of $\widetilde{\bP}^{\theta}_{x,i}$ on $\cF_t$ is $\bP^{\theta}_{x,i}$, we also have by \eqref{eq:measure-change-P} that
        \begin{equation}
        \widetilde{\bE}^{\theta}_{x,i} \left[\frac{e^{-\theta X_u(t)- \lambda(\theta) t} V_{J_u(t)}(\theta)\ind{u\in\cN_t} \mathbb{1}_{B}}{ W_{\theta}(t)}\right]
        =
        \frac{\bE_{x,i}\left[e^{-\theta X_u(t) - \lambda(\theta) t} V_{J_u(t)}(\theta) \ind{u\in\cN_t} \mathbb{1}_{B}\right]}{e^{-\theta x}V_{i}(\theta)}.
        \end{equation}
        This completes the proof of \eqref{equation_choose_spine}.

    To describe the particle system under the law $\widetilde{\bP}^\theta_{x,i}$, we introduce the following change of measure for MAPs, which is a variation of \cite[Proposition 5.6]{PalRol2002}; its proof is postponed to Section \ref{section:proof_spine}.
    \begin{lemma}\label{lemma_spine_MAP}
    For any $x\in \RR$ and $i\in \cI$, 
    let $(\chi,\Theta)$ be a MAP with triplet $((\phi_i)_{i\in\cI}, Q, G)$ under law $\PP_{x,i}$. For $\theta \in (0,\bar{\theta})$, define 
    \begin{equation}\label{def_Xi_theta}
        \Xi_{\theta}(t) := e^{-\theta \chi_t - \lambda(\theta)t + \int_0^t(\beta(m-1))(\Theta_s)\dd s} V_{\Theta_t}(\theta), \qquad t\ge 0.
    \end{equation}
    It is a $\PP_{x,i}$-martingale. Define a new probability measure $\PP_{x,i}^{\theta}$ by
    \begin{equation}\label{def_P^theta}
        \frac{\dd \PP_{x,i}^{\theta}}{\dd \PP_{x,i}} \bigg{|}_{\mathcal{F}_t^{(\chi,\Theta)}} := \frac{\Xi_{\theta}(t)}{\Xi_{\theta}(0)}, \qquad t\ge 0, 
    \end{equation}
    where $(\mathcal{F}_t^{(\chi,\Theta)}, t\geq 0)$ is the natural filtration of the MAP $(\chi, \Theta)$. Then under $\PP_{x,i}^{\theta}$, $((\chi, \Theta)_{t\geq 0})$ is a MAP with the following characteristics: for $k,j \in \cI$, 
    \begin{equation}\label{MAP_new_characher}
        \begin{aligned}
        &\widetilde{q}_{kj} = \frac{q_{kj} V_j(\theta) G_{kj}(\theta)}{V_k(\theta)}, \; \quad \mathbb{P} (\widetilde{U}_{kj}\in \dd x) = \frac{e^{-\theta x}}{G_{kj}(\theta)} \mathbb{P} (U_{kj}\in \dd x), \quad \forall j\neq k, \quad \widetilde{q}_{kk} = -\sum_{j\neq k}\widetilde{q}_{kj}, \\
        &\widetilde{\sigma}_k = \sigma_k, \; \widetilde{\ta}_k = \ta_k - \theta \sigma_k^2 - \int x \ind{[0,1]}(|x|) (1-e^{-\theta x}) \Lambda_k(\dd x), \; \widetilde{\Lambda}_k(\dd x) = e^{-\theta x} \Lambda_k(\dd x).
    \end{aligned}
    \end{equation}

    Therefore the corresponding MAP triplet is given by  
    \begin{equation}
        \left((\widetilde{\phi}_k)_{k\in\cI}, \widetilde{Q}:=(\widetilde{q}_{kj})_{k,j\in\cI}, \widetilde{G}:=(\widetilde{G}_{kj})_{k,j\in\cI}\right),
    \end{equation}
    where $\widetilde{\phi}_k(\alpha)=\phi_k(\alpha+\theta)-\phi_k(\theta)$  for $k\in \mathcal{I}$, and $\widetilde{G}_{kj}(\alpha)=\frac{G_{kj}(\alpha+\theta)}{G_{kj}(\theta)}$ for $ k,j\in \mathcal{I}$. 
    \end{lemma}
   
    We now give the spine decomposition for $\widetilde{\bP}^{\theta}_{x,i}$ and the proof is also postponed in Section \ref{section:proof_spine}. Since the projection of $\widetilde{\bP}^{\theta}_{x,i}$ on $\cF$ is $\bP^{\theta}_{x,i}$, the latter is also described by this spine decomposition. 
    \begin{theorem}[Spine decomposition]\label{thm_spine_decomposition}
        Let $x\in \RR$ and $i\in \cI$. Under $\widetilde{\bP}^{\theta}_{x,i}$, the branching MAP is described as follows. 
    \begin{itemize}
    \item The spine $\xi$ evolves according to a MAP $(\chi,\Theta)$ of law $\mathbb{P}^{\theta}_{x,i}$ with characteristics given by \eqref{MAP_new_characher}.

    \item Given the type process $J_{\xi}(t)$ of the spine, the branching rate of the spine at time $t\ge 0$ is given by $\widetilde{\beta}_{J_{\xi}(t)}$, where $\widetilde{\beta}_i = \beta_{i} m_{i}, i\in \cI$; when it splits, it gives birth to an offspring of the same type at the same position. The number of children is given by the size-biased law $(\widetilde{\mu}_j(k) := \frac{k\mu_j(k)}{m_j},k\ge 1)$, for $j\in \cI$ with $m_j>0$; if $m_j=0$, then $\widetilde{\beta}_j=0$, meaning that the spine can never split at state $j$, and we simply define by convention that $\widetilde{\mu}_j(1)=1$.
    
    \item Choose one child uniformly at random, which continues as the spine; for the other children, each of them leads a subpopulation of the original law $\mathbf{P}$ shifted to their point and time of creation. The spine continues in a similar way.  
    \end{itemize}
    \end{theorem}

    A direct consequence is the following many-to-one formula, which is well-known in context of branching particle systems: for every non-negative measurable function $g$ and any $(x,i)$, 
    \begin{align}\label{eq:many-to-one}
         &\frac{e^{\theta x}}{V_i(\theta)}  \bE_{x,i}\bigg[\sum_{u\in \cN_t} g(X_u(t), J_u(t)) e^{-\theta X_u(t)-\lambda(\theta) t} V_{J_u(t)}(\theta) \bigg]\\
         =&\,  \widetilde{\bE}^{\theta}_{x,i}[g(X_{\xi}(t), J_{\xi}(t))] = \EE^{\theta}_{x,i}[g(\chi_t, \Theta_t)].
    \end{align}

\subsection{Proof of Theorem \ref{Biggins_theorem}}
\begin{lemma}\label{speed of spine}
    For any $x\in \RR, i\in \cI$, the linear speed of the spine under $\widetilde{\bP}^{\theta}_{x,i}$ is 
    \begin{equation}
        \lim_{t\to \infty}\frac{X_\xi(t)}{t} = -\lambda'(\theta).
    \end{equation}
\end{lemma}
\begin{proof}
    According to Lemma \ref{lemma_spine_MAP}, under $\widetilde{\mathbf{P}}_{x,i}^{\theta} $, the spine behaves as a MAP with matrix exponent 
    \[\widetilde{F}(\alpha) = \diag(\widetilde{\phi}_i(\alpha))_{i\in\mathcal{I}}+\left(\widetilde{q}_{ij}\widetilde{G}_{ij}(\alpha)\right)_{i,j\in \mathcal{I}}.\]
    For each $i \in \mathcal{I}$, set $\widetilde{V}_{i}(\alpha) := \frac{V_i(\theta+\alpha)}{V_i(\theta)}$ and note that $G_{ii}(\theta) = G_{ii}(\theta+\alpha) = 1$, we have  
    \begin{align*}
        [\widetilde{F}(\alpha)&\widetilde{V}(\alpha)]_i = \left(\phi_i(\alpha+\theta)-\phi_i(\theta)\right)\widetilde{V}_i(\alpha) + 
        \sum_{j\neq i}\frac{q_{ij} V_j(\theta) G_{ij}(\theta+\alpha)}{V_i(\theta)}\widetilde{V}_j(\alpha) + \widetilde{q}_{ii}\widetilde{V}_i(\alpha)\\
        &= \left(\phi_i(\alpha+\theta)-\phi_i(\theta)+ 
        \sum_{j\neq i}
        \frac{q_{ij} V_j(\theta+\alpha) G_{ij}(\theta+\alpha)} {V_i(\theta+\alpha)}
              -\sum_{j\neq i}\frac{q_{ij} V_j(\theta) G_{ij}(\theta)}{V_i(\theta)}\right)\widetilde{V}_i(\alpha)\\
        &=\left(\left[\mathcal{M}(\theta+\alpha)\frac{\vec{V}(\theta+\alpha)}{V_i(\theta+\alpha)}\right]_{i}-\left[\mathcal{M}(\theta)\frac{\vec{V}(\theta)}{V_i(\theta)}\right]_{i}\right)\widetilde{V}_i(\alpha)\\
        &=\left(\lambda(\theta+\alpha)-\lambda(\theta)\right)\widetilde{V}_i(\alpha).
    \end{align*}
    Therefore, $\left({\widetilde{V}}_1(\alpha), \dots, \widetilde{V}_{\mathtt{d}}(\alpha) \right)$ is an eigenvector of $F(\alpha)$ with all entries positive. Then it is a PF eigenvector of $\widetilde{F}(\alpha)$, with the PF eigenvalue $\widetilde{\lambda}(\alpha) = \lambda(\alpha+\theta)-\lambda(\theta)$. Then, according to Lemma~\ref{speed of MAP}, the MAP has a linear speed of $-(\widetilde{\lambda})'(0) = -\lambda'(\theta)$.
\end{proof}

We are now ready to prove Theorem \ref{Biggins_theorem}.
\begin{proof}[Proof of Theorem \ref{Biggins_theorem}]
\ 
\paragraph{The degenerate phase}
Suppose that at least one of the following two conditions hold: (i) $\theta \lambda'(\theta)\ge \lambda(\theta)$; (ii) there exists some $j\in \cI$ such that $\sum_{k\ge 1}(k\log k)\mu_{j}(k)=+\infty$ (then we must have $\beta_j m_j>0$). 
We next show that in either case $\limsup_{t\to\infty} W_{\theta}(t) = +\infty$ 
$\widetilde{\mathbf{P}}_{x,i}^{\theta}$-a.s.\ and then  by \cite[Theorem 4.3.5]{PTEDurrett} we conclude that $W_{\theta}(t)$ converges to $0$, $\mathbf{P}_{x,i}$-a.s.

    (i) We first assume that $\theta \lambda'(\theta)\ge \lambda(\theta)$. 
        Under $\widetilde{\mathbf{P}}_{x,i}^{\theta}$, the branching system does not extinct. For $t\ge 0$,
        \[
        W_{\theta}(t) \ge e^{-\theta X_{\xi}(t)-\lambda(\theta)t}V_{J_{\xi}(t)}(\theta).
        \]
        By Lemma \ref{speed of spine},  we have $\lim_{t\to \infty}X_{\xi}(t)/t = -\lambda'(\theta)$, $\widetilde{\mathbf{P}}_{x,i}^{\theta}$-a.s. Thus, when $\theta \lambda'(\theta) >  \lambda(\theta)$, we have 
        \begin{equation}\label{eqn: limsup of spine}
            \widetilde{\mathbf{P}}_{x,i}^{\theta}\left(\limsup_{t\to\infty}(-\theta X_{\xi}(t)-\lambda(\theta)t)=+\infty\right) = 1.
        \end{equation}
        When $\theta \lambda'(\theta) = \lambda(\theta)$, by Proposition \ref{speed of MAP}, we also have \eqref{eqn: limsup of spine}. 
        Therefore, we have 
        \begin{equation}
        \limsup_{t\to\infty} W_{\theta}(t) = +\infty, \quad\widetilde{\mathbf{P}}_{x,i}^{\theta}\mbox{-a.s.}
        \end{equation}
        
        (ii) We next  assume that there exists some $j\in \cI$ such that $\sum_{k\ge 1}(k\log k)\mu_{j}(k)=+\infty$. 
        We also assume that $\theta \lambda'(\theta)< \lambda(\theta)$; otherwise it falls into the first case.
        
        Let $T_n$ denote the $n$-th fission (branching) time of the spine. For state $j$, let $T_m^{j}$ be the time the spine undergoes its $m$-th fission at state $j$. Therefore $\{A_{\xi}(T_{n}^{j}):= A_{\xi_{T_{n}^{j}}}\}_{n\ge 1}$ is an i.i.d. sequence with law $(\widetilde{\mu}_j(k))_{k\ge 1}$. 
        Since ${\widetilde{\mathbf{E}}_{x,i}}^{\theta}\left[\log A_{\xi}(T_{n}^j)\right] = \frac{1}{m_j}\sum_{k\ge 1}(k\log k)\mu_j(k)=\infty$, the Borel-Cantelli lemma leads to 
        \begin{equation}\label{eqn:logA/n}
           \limsup_{n\to\infty}
            \frac{\log A_{\xi}(T_{n}^{j})}{n}=+\infty,\quad \widetilde{\mathbf{P}}_{x,i}^{\theta}\text{-a.s.}
        \end{equation}

        According to Theorem~\ref{thm_spine_decomposition}, 
        the process $(T_m^{j})_{m\ge 1}$ is a Cox process on $\RR_+$ with rate $\widetilde{\beta}_{j}\ind{{J}_{\xi}(t) = j} \dd t$, directed by $({J}_{\xi}(t))_{t\ge0}$, where $\widetilde{\beta}_{j} := \beta_j m_j$ and $(J_{\xi}(t),t\geq 0)$ is an irreducible Markov chain specified by \eqref{MAP_new_characher}. Let $(\widetilde{\pi}_j)_{j\in\mathcal{I}}$ be the invariant distribution of $(J_{\xi}(t),t\geq 0)$.
        Then we have 
        \begin{align}
            \lim_{n\to\infty}\frac{T_n^{j}}{n} &= \frac{1}{\widetilde{\pi}_j\widetilde{\beta}_j}, \quad \widetilde{\mathbf{P}}_{x,i}^{\theta}\text{-a.s.} \label{eqn:lim Tnj/n}
        \end{align}
         Let us justify \eqref{eqn:lim Tnj/n}. We have by the ergodic theorem that
        \begin{equation}\label{eqn:lim int_beta}
            \lim_{t\to \infty}\frac{1}{t}\int_{0}^{t}\widetilde{\beta}_{j}\ind{{J}_{\xi}(s) = j}\dd s = \widetilde{\pi}_j\widetilde{\beta}_j, \quad \widetilde{\mathbf{P}}_{x,i}^{\theta}\text{-a.s.}
        \end{equation}
        Given $J_\xi$, the counting process $(N^j(t):= \max\{n: T_n^j\le t\})_{t\ge 0}$ is an inhomogeneous Poisson process with rate $\widetilde{\beta}_{j}\ind{{J}_{\xi}(t) = j}\dd t$. 
        Notice that as $t\to\infty$, the integral 
        \begin{equation}
            \int_0^{t}\widetilde{\beta}_{J_{\xi}(s)}\ind{{J}_{\xi}(s) = j}\dd s \to \infty, \quad \widetilde{\mathbf{P}}_{x,i}^{\theta}\mbox{-a.s.}     
        \end{equation}
        Therefore, we have
        \begin{equation*}
            \widetilde{\mathbf{P}}_{x,i}^{\theta}\left( \lim_{t\to \infty}\frac{N^j(t)}{\int_{0}^t\widetilde{\beta}_{j}\ind{{J}_{\xi}(s) = j}\dd s} = 1\Bigg| J_\xi\right)=1.
        \end{equation*}
        It follows that
        \begin{equation}\label{eqn:lim Nt/beta}
            \lim_{t\to \infty}\frac{N^j(t)}{\int_{0}^t\widetilde{\beta}_{j}\ind{{J}_{\xi}(s) = j}\dd s} = 1, \quad \widetilde{\mathbf{P}}_{x,i}^{\theta}\text{-a.s.}
        \end{equation}
        Combining \eqref{eqn:lim int_beta} and \eqref{eqn:lim Nt/beta} we have
        \begin{equation*}
            \lim_{t\to \infty}\frac{N^j(t)}{t} = \widetilde{\pi}_j\widetilde{\beta}_j>0, \quad \widetilde{\mathbf{P}}_{x,i}^{\theta}\text{-a.s.}
        \end{equation*}
        Taking $t=T_n^j$ and noticing that $N^j(T_n^j) = n$, we deduce \eqref{eqn:lim Tnj/n}. 

        At time $T_n^j$, we have a lower bound of the additive martingale,
        \begin{align}
            W_{\theta}(T_n^j)&\ge A_{\xi}(T_{n}^j)e^{-\theta X_{\xi}(T_n^j) -\lambda(\theta)T_n^j}V_{j}(\theta)\\
            \label{lower bound of Wn}
            &=\exp\left\{n\left(\frac{\log A_{\xi}(T_{n}^j)}{n}-\frac{T^j_n}{n}\left(\frac{\theta X_{\xi}(T_n^j)}{T_n^j}+\lambda(\theta)\right)\right)\right\}V_{j}(\theta).
        \end{align}
        In the case of $\theta \lambda'(\theta)< \lambda(\theta)$ and $\sum_{k\ge 1}(k\log k)\mu_{j}(k)=+\infty$ for $j\in\cI$ with $\beta_jm_j>0$,  
        by \eqref{eqn:logA/n}, \eqref{eqn:lim Tnj/n} and \eqref{lower bound of Wn}, we again have $\limsup_{t\to \infty}W_{\theta}(t)=+\infty, \widetilde{\mathbf{P}}_{x,i}^{\theta}$-a.s.

    \paragraph{The $\mathcal{L}^1$-convergence phase} On the other hand, assume that $\theta \lambda'(\theta) < \lambda(\theta)$ and $\sum_{k\ge 1} (k \log k) \mu_j(k) < \infty$ for all $j \in \mathcal{I}$.  We prove    that the additive martingale converges in $\mathcal{L}^1(\mathbf{P}_{x,i})$. 
    Recall that $\cG$ defined by \eqref{eqn:Sigma_fields_G} is the $\sigma$-field generated by the motion and type of the spine with its fission time and children. By \cite[Lemma~4.2]{ShiBRW}, it suffices to show:
    \begin{equation}\label{eqn:limsup_Wt}
    \limsup_{t\to +\infty}\widetilde{\mathbf{E}}_{x,i}^{\theta}\left[W_{\theta}(t)\mid \cG\right]<+\infty,\quad \widetilde{\mathbf{P}}_{x,i}^{\theta}\text{-a.s.}
    \end{equation}
    With notation in  Section~\ref{subsection_spine_decom}, we have the following decomposition:
    \begin{equation}\label{condition on spine}
        \widetilde{\mathbf{E}}_{x,i}^{\theta}\left[W_{\theta}(t)\mid\cG\right] = e^{-\theta X_{\xi}(t)-\lambda(\theta)t}V_{J_{\xi}(t)}(\theta)+\sum_{n=1}^{\infty}\ind{T_n\le t}A_{\xi}(T_n) e^{-\theta X_{\xi}(T_n)-\lambda(\theta)T_n}V_{J_{\xi}(T_n)}(\theta) .
    \end{equation} 
    When $\theta \lambda'(\theta) <\lambda(\theta)$, the first term above converges to $0$, $\widetilde{\mathbf{P}}_{x,i}^{\theta}$-a.s. 
    For the second term, we have 
    \begin{align*}
          &\sum_{n=1}^{\infty}\ind{T_n\le t}A_{\xi}(T_n) e^{-\theta X_{\xi}(T_n)-\lambda(\theta)T_n}V_{J_{\xi}(T_n)}(\theta) \\
        &= \sum_{j\in \mathcal{I}} \sum_{n=1}^{\infty}\ind{T_n\le t} A_{\xi}(T_{N_j(n)}) e^{-\theta X_{\xi}({T_{N_j(n)}})-\lambda(\theta)T_{N_j(n)}}V_{j}(\theta) \\
        &= \sum_{j\in \cI}\sum_{n=1}^{\infty} \ind{T_n\le t} V_{j}(\theta) 
        \exp{\left\{N_j(n)\left[\frac{\log A_{\xi}(T_{N_j(n)})}{N_j(n)}+\frac{T_{N_j(n)}}{N_j(n)}\left(-\frac{\theta X_{\xi}({T_{N_j(n)}})}{T_{N_j(n)}}-\lambda(\theta)\right)\right]\right\}}.
    \end{align*}
    Since $\sum_{k\ge 1}(k\log k)\mu_j(k)<+\infty$ for $\forall j \in \cI$, by the Borel-Cantelli lemma, we have
    \begin{equation}
    \widetilde{\mathbf{P}}_{x,i}^{\theta} \left(\limsup_{n\to+\infty}\frac{\log A_{\xi}(T_{N_j(n)})}{N_j(n)}=0\right)=1.
    \end{equation}
    Therefore, for any $j\in \mathcal{I}$, by \eqref{eqn:lim Tnj/n}, there exists $c_j>0$ such that $\widetilde{\mathbf{P}}_{x,i}^{\theta}$-a.s.
    \begin{equation*}
        \limsup_{n\to\infty}\left(\frac{\log A_{\xi}(T_{N_j(n)})}{N_j(n)}+\frac{T_{N_j(n)}}{N_j(n)}\left(-\frac{\theta X_{\xi}(T_{N_j(n)})}{T_{N_j(n)}}-\lambda(\theta)\right)\right)\le  c_{j}\left(\theta \lambda'(\theta)-\lambda(\theta)\right)<0.
    \end{equation*}
    Then we have \eqref{eqn:limsup_Wt}.

    We now show that, when $W_{\theta}$ converges in $\mathcal{L}^1(\bP_{x,i})$, we have $W_{\theta}(\infty)>0$ on the non-extinction event $\mathscr{S}$. 
    
    Define $w(i):= \bP_{x,i}(W_\theta(\infty) = 0)$, for $i\in\mathcal{I}$, and $\vec{w}:=(w(1), \ldots, w(\mathtt{d}))$. Note that $w(i)$ does not depend on $x$. By \cite[Theorem 4.3.5]{PTEDurrett}, when $W_{\theta}$ converges in $\mathcal{L}^1(\bP_{x,i})$, we have $\mathbf{E}_{x,i}\left[\frac{W_{\theta}(\infty)}{W_{\theta}(0)}\right] = 1$. This implies $\vec{w} \neq (1, \dots, 1)$.
    By strong Markov property on the first branching/type-changing time, we have
    \begin{equation}
        w(i) = \frac{\beta_i (\sum_{k\ge 0} \mu_{i}(k)w(i)^k) + \sum_{j\neq i}q_{ij} w(j)}{\beta_i+q_i},
    \end{equation}
    for $i\in \mathcal{I}$. This is equivalent to the equation:
    \begin{equation}
        \diag \bigg(\beta_1 \Big(\sum_{k\ge 0} \mu_i(k) w(1)^k -w(1)\Big), \ldots, \beta_{\mathtt{d}} \Big(\sum_{k\ge 0} \mu_i(k) w({\mathtt{d}})^k -w({\mathtt{d}})\Big)\bigg) + Q w^{\top} = 0. 
    \end{equation}
    This is the same equation as \eqref{eqn:surv_prob}. Then $\vec{w}$ is the unique solution on $[0,1]^d \setminus \{(1,\dots,1)\}$ of the equation, and therefore $\vec{w} = \vec{\bq}$, i.e. $\bP_{x,i}(W_{\theta}(\infty) = 0) = \bP_{x,i}(\mathscr{S}^c)$. Note that $\mathscr{S}^c\subseteq \{W_{\theta}(\infty) = 0\}$. 
    This implies $\bP_{x,i}\left(\{W_{\theta}(\infty)>0\} \Delta \mathscr{S}\right) = 0$.
\end{proof}

\subsection{Velocity of the leftmost particle}
\begin{proof}[Proof of Corollary \ref{thm:speed of Lt}]
    For simplicity, we write $L_t := \min_{u\in \mathcal{N}_t} X_u(t)$. 
    
    We first show that,  $\bP_{x,i}$-a.s.\ on $\mathscr{S}$, we have 
	\begin{equation}\label{eqn:liminf_Lt}
	    \liminf_{t\to +\infty}\frac{L_t}{t} \ge -\inf_{\theta\in [\theta^*, \bar{\theta})}\frac{\lambda(\theta)}{\theta} = -\frac{\lambda(\theta^*)}{\theta^*}.
	\end{equation}
    We start by deriving a simple lower bound for $W_\theta(t)$ on $\mathscr{S}$.
	\begin{align*}
		W_{\theta}(t) \ge \exp{\{-\theta L_t-\lambda(\theta)t\}}\min_{i\in\mathcal{I}}V_i(\theta).
	\end{align*}
	By Theorem \ref{Biggins_theorem}
	, when $\theta\in [\theta^*, \bar{\theta})$, the additive martingale $W_\theta(t)$ converges to 0, $\mathbf{P}_{x,i}$-a.s. Therefore by the above inequality, $\mathbf{P}_{x,i}$-a.s. on $\mathscr{S}$\ 
    \begin{equation}\label{eq:min_-infty}
     \lim_{t\to +\infty}( \theta L_t +\lambda(\theta)t) = \infty. \end{equation}
    This implies \eqref{eqn:liminf_Lt}.

     We then show that, 
     $\bP_{x,i}$-a.s.\ on $\mathscr{S}$, we have 
 	\begin{equation}\label{eqn:limsup_Lt}
 	    \limsup_{t\to +\infty}\frac{L_t}{t} \le -\sup_{\theta\in (0, \theta^*)}\lambda'(\theta) = \lambda'(\theta^*).
 	\end{equation}
    Using Proposition \ref{speed of MAP} to the spine under $\widetilde{\bP}^{\theta}_{x,i}$ yields that 
    \[
    \limsup_{t\to +\infty}\frac{L_t}{t} \le\limsup_{t\to +\infty}\frac{X_{\xi}(t)}{t} = -\lambda'(\theta) \quad \widetilde{\bP}^{\theta}_{x,i}\text{-a.s.} 
    \]
    Therefore, $\limsup_{t\to +\infty}\frac{L_t}{t} \le -\lambda'(\theta), \bP^{\theta}_{x,i}$-a.s. 
    Since Theorem~\ref{Biggins_theorem} yields that $\bP^{\theta}_{x,i}$ and $\bP_{x,i}$ are equivalent on the non-extinction event $\mathscr{S}$, we have $\limsup_{t\to +\infty}\frac{L_t}{t} \le -\lambda'(\theta)$ almost surely under $\bP_{x,i}$. 
    Optimizing in $\theta\in (0, \theta^*)$ yields \eqref{eqn:limsup_Lt}.

    Combining \eqref{eqn:liminf_Lt} and \eqref{eqn:limsup_Lt}, and using the identity $\lambda'(\theta^*) = \frac{\lambda(\theta^*)}{\theta^*}$ from \ref{A2}, we complete the proof.
\end{proof}

\section{Convergence of the derivative martingales}\label{section: the derivative martingales}
In this section, we assume \ref{A1} \ref{A2} and prove Theorem~\ref{thm:asymptotTruncated} for the derivative martingale at criticality $\theta = \theta^*$. The proof relies on a study of the truncated derivative martingales that we introduce in section~\ref{sec:truncated-dm}. 

\subsection{Truncated derivative martingales}\label{sec:truncated-dm}
Recall the spine decomposition in Section~\ref{subsection_spine_decom}. Under assumption~\ref{A2}, we perform the change of measure with the critical value $\theta = \theta^*$. Then the spine $(X_{\xi}(t), J_{\xi}(t))_{t\ge 0}$ under $\widetilde{\mathbf{P}}_{x,i}^{\theta^*}$ is a MAP with matrix exponent given by Lemma \ref{lemma_spine_MAP}. 
Define
\begin{equation}
  (\widehat{X}_{\xi}(t), J_{\xi}(t)) :=  (\theta^* {X}_{\xi}(t) + \lambda(\theta^*) t, J_{\xi}(t)), \qquad t\ge 0.  
\end{equation}
In particular, we have $\widetilde{\mathbf{E}}_{x,i}^{\theta^*}[\widehat{X}_{\xi}(t)] =0$ and $\widetilde{\mathbf{E}}_{x,i}^{\theta^*}[\widehat{X}_{\xi}^2(t)] <\infty$.   

For $x\in \RR$, let $\tau_x:= \inf \{t\ge 0\colon \widehat{X}_{\xi}(t) <x\}$. According to \cite[Theorem 29]{DLK17}, there exists finite and non-negative functions $(R_i)_{i\in \cI}$, such that the process 
\begin{equation}\label{eq:R-mart}
    \big(R_{J_{\xi}(t)}(\widehat{X}_{\xi}(t)) \ind{\tau_{0} > t},\quad  t \geq 0 \big) 
\quad \text{is a non-negative} ~\widetilde{\mathbf{P}}_{0,i}^{\theta^*}~\text{-martingale.}
\end{equation}
The function $(R_i)_{i\in\cI}$ are continuous and non-decreasing, which are referred to as the renewal functions for the MAP $(\widehat{X}_{\xi}(t), J_{\xi}(t))_{t\ge 0}$. 

\begin{lemma}\label{lem:renewal theorem}
There exists a constant $c_{\text{ren}}\in (0,\infty)$,  such that for every $i\in \mathcal{I}$, 
    \begin{equation}\label{eqn:renewal theorem}
        \lim_{u\to+\infty}\frac{R_{i}(u)}{u} = c_{\text{ren}}.
    \end{equation}
\end{lemma}
\begin{proof}
    Since $\widetilde{\mathbf{E}}_{x,i}^{\theta^*}[(-\widehat{X}_{\xi})^2(t)] <\infty$, we know from \cite[Theorem 35 and Lemma 38]{DLK17} 
    that $(-\widehat{X}_{\xi},J_{\xi})$ has tight overshoot under $\widetilde{\mathbf{P}}_{x,i}^{\theta^*}$, and  equivalently,  $\sum_{j\in \cI}\widetilde{\pi}_j \widetilde{\mathbf{E}}_{0,i}^{\theta^*} [H^{-}(1)]\in (0,\infty)$, where $H^{-}$ is the so-called descending ladder height process     of $(\widehat{X}_{\xi},J_{\xi})$  under $\widetilde{\mathbf{P}}_{0,i}^{\theta^*}$ and $\widetilde{\pi}$ is the stationary distribution of $J_\xi(t)$ under $\widetilde{\mathbf{P}}_{0,i}^{\theta^*}$.
    \footnote{Note that, as remarked by \cite[Page 1995]{DLK17}, we do not require non-lattice assumption.}   
    Then it follows from the Markov renewal theory (see  e.g.\ \cite{Lalley84} or \cite[Theorem 28 (i)]{DLK17};     note that the version we use here holds for the lattice case as well) that 
    \[   \lim_{u\to+\infty}\frac{R_{i}(u)}{u} = \frac{\sum_{j\in \cI} \widetilde{\pi}_j^2}{\sum_{j\in \cI}\widetilde{\pi}_j \widetilde{\mathbf{E}}_{0,i}^{\theta^*} [H^{-}(1)]}\in (0,\infty).\qedhere 
    \]
\end{proof}

For $b> \max(-\theta^*x, 0)$, we define \emph{the truncated derivative martingale} under $\mathbf{P}_{x,i}$: 
\begin{equation}\label{eqn:defTruncatedDerivativeMartingale}
	Z_{\theta^*}^{(b)}(t) := \sum_{u \in \mathcal{N}_t} R_{J_u(t)}\left(\widehat{X}_u(t) +b\right) \ind{\inf_{s \leq t} \widehat{X}_u(s) \ge -b} e^{-\widehat{X}_u(t)} V_{J_u(t)}(\theta^*), \qquad t\ge 0, 
	\end{equation}
    where $\widehat{X}_u(t) := \theta^* {X}_{u}(t) + \lambda(\theta^*) t$. 
Indeed, applying the many-to-one formula \eqref{eq:many-to-one} and then \eqref{eq:R-mart}, we have  
\begin{align*}
    \bE_{x,i} \left[Z_{\theta^*}^{(b)}(t)\right] &  = 
    V_i(\theta^*)e^{- \theta^*x} \widetilde{\mathbf{E}}_{x,i}^{\theta^*}\left[R_{J_{\xi}(t)}(\widehat{X}_{\xi}(t)+b) \ind{\inf_{s \leq t} \widehat{X}_{\xi}(s) \ge -b} \right]\\
    &=  V_i(\theta^*) e^{- \theta^*x} R_i( \theta^*x+b). 
\end{align*}
Then it follows from the branching property that $Z_{\theta^*}^{(b)}(t)$ is a non-negative martingale under $\bP_{x,i}$, and therefore converges a.s.\ to a limit $Z_{\theta^*}^{(b)}(\infty)\ge 0$ as $t\to\infty$. 

The convergence of the derivative martingale, stated in Theorem~\ref{thm:asymptotTruncated}, would rely on the  the following $\mathcal{L}^1$-convergence result of the truncated derivative martingale $Z_{\theta^*}^{(b)}$. 
\begin{theorem}[Uniform integrability]\label{thm:L1_conv_Zb}
Let $x\in \RR$ and $b> \max(-\theta^*x, 0)$. 
Assume \ref{A1}, \ref{A2}, and \eqref{eq:dev-ui}. 
Then $Z_{\theta^*}^{(b)}$ is a uniform integrable martingale under $\mathbf{P}_{x,i}$, and  $\bE_{x,i}\left[\frac{Z_{\theta^*}^{(b)}(\infty)}{Z_{\theta^*}^{(b)}(0)}\right] = 1$. 
\end{theorem}

\subsection{Proof of Theorem \ref{thm:L1_conv_Zb}}\label{subsec: L1_conv_Zb}
In this subsection we prove Theorem~\ref{thm:L1_conv_Zb}.  
To simplify the analysis, we perform the linear transformation 
$\widehat{X}_u(t) = \theta^* {X}_u(t) + \lambda(\theta^*) t$, $t\ge 0$.
This yields a new branching MAP $(\widehat{X}_u(t), J_u(t))$ for which assumption \ref{A2} becomes: 
\begin{enumerate}[label=(A3), ref=(A3),leftmargin=5.0em]
\item $\lambda(\theta^*) = \lambda'(\theta^*) = 0$ and $\theta^* = 1$. \label{A3}
\end{enumerate}
Since the results are preserved under the linear transformation, it suffices to prove for the transformed branching MAP, and then the corresponding statements for the original branching MAP follow immediately. 

Therefore, we assume that the 
assumptions \ref{A1} and \ref{A3} hold for the branching MAP  $(X_u(t), J_u(t))$ and prove Theorem~\ref{thm:L1_conv_Zb}.
To study the $\mathcal{L}^1$ convergence of the truncated derivative martingale $Z_{\theta^*}^{(b)}$, let us introduce a further change of measure.   
For any $x \ge -b$ and $i\in \cI$, let  $\widetilde{\bP}^{\theta^*}_{x,i}$ be defined as in \eqref{def_tildeQ} with $\theta^*=1$. 
Recall  by Lemma~\ref{lemma_spine_MAP} that the spine $(X_{\xi}(t), J_{\xi}(t))$ under $\widetilde{\bP}^{\theta_*}_{x,i}$ is a MAP with characteristics given by \eqref{MAP_new_characher}. Then define a change of measure by the martingale \eqref{eq:R-mart} associated with the spine: for every $t \geq 0$, 
\begin{equation}
\label{eqn:branching-Conditioned}
\left. \frac{\dd\widetilde{\bP}^\uparrow_{x,i}}{\dd\widetilde{\bP}^{\theta_*}_{x,i}} \right|_{\widetilde{\mathcal{F}}_t} := \frac{R_{J_{\xi}(t)}(X_{\xi}(t)+b)}{R_i(x+b)} \ind{\inf_{s\le t} X_{\xi}(s) \ge -b}. 
\end{equation}

To describe the law of the spine under $\widetilde{\bP}^\uparrow_{x,i}$, let us denote by ${\PP}^{\theta_*}_{x,i}$ the law  of a MAP $(\chi,\Theta)$ with characteristics given by \eqref{MAP_new_characher} and define a change of measure: 
    \begin{equation}\label{eq:MAP-cond}
    \left. \frac{\dd\PP_{x,i}^{\uparrow}}{\dd{\PP}^{\theta_*}_{x,i}} \right|_{\mathcal{F}_t^{(\chi,\Theta)}} := \frac{R_{\Theta(t)}(\chi(t)+b)}{R_i(x+b)} \ind{\inf_{s\le t} \chi(s) \ge -b}.  
    \end{equation}
Then $\PP_{x,i}^{\uparrow}$ is known as the law of  \emph{a MAP conditioned to stay positive}; we refer to \cite[Appendices A.7-A.8]{DLK17} for more details. Lévy processes (with single type) conditioned to stay positive have been the subject of a large literature; we refer to \cite{ChD05} and reference therein. 
Define $\bP_{x,i}^{\uparrow}$ as the projection of $\widetilde{\bP}_{x,i}^{\uparrow}$ on $(\mathcal{T}, \cF)$. By projection \eqref{equation_choose_spine} and change of measure \eqref{eq:measure-change-P}, we also deduce the connection with the original law $\bP_{x,i}$: for $t\ge 0$,   
\begin{equation}
    \left. \frac{\dd{\bP}^\uparrow_{x,i}}{\dd\bP_{x,i}} \right|_{\mathcal{F}_t}=\frac{Z_{\theta^*}^{(b)}(t)}{R_i(x+b)}. 
\end{equation}

Similar to Theorem~\ref{thm_spine_decomposition} for $\widetilde{\bP}_{x,i}^{\theta^*}$, under $\widetilde{\bP}^{\uparrow}_{x,i}$, we also have a spinal description of the dynamics, specified in the following proposition. The only difference lies in the spine's movement. Specifically, the spine under $\widetilde{\bP}_{x,i}^{\uparrow}$ is a MAP conditioned to stay positive, whereas under $\widetilde{\bP}_{x,i}^{\theta^*}$ it is unconditioned. 
For completeness we give a proof of Proposition~\ref{prop:spine-decomp-2} in Section~\ref{sec:spine-2}.

\begin{proposition}\label{prop:spine-decomp-2}
Let $x\in \RR$ and $i\in \cI$. Under $\widetilde{\bP}_{x,i}^{\uparrow}$, 
the branching MAP is described
as follows.
\begin{itemize}
     \item There is a spine starting from position $x$ with type $i$ and moves according to a MAP conditioned to stay positive defined by \eqref{eq:MAP-cond}. 
     \item Given the type process $\Theta$ of the spine, the branching rate at time $t\ge 0$ is given by $\beta_{\Theta_t} m_{\Theta_t}$;
     when it splits, it gives birth to an offspring of the same type at the same position. The number of children is given by the size-biased law $\widetilde{\mu}_j(k) := \frac{k\mu_j(k)}{m_j}$, for $k\ge 1$ and $j\in \cI$. 
     \item Choose one child uniformly at random, which continues as the spine; each of the other children leads a subpopulation of the original law $\mathbf{P}$, independent of each other. The spine continues in a similar way.  
\end{itemize}
\end{proposition}

\begin{proof}[Proof of Theorem \ref{thm:L1_conv_Zb}]
Denote by $\mathbf{M}(\dd s, \dd k)$ the counting measure (on $\RR_+\times \mathbb{Z}_{>0}$) of the spine's fission times and number of children. Let $\mathcal{G}$ be the $\sigma$-field generated by the spine $(X_{\xi}, J_{\xi})$ and $\mathbf{M}$, as given in Lemma \ref{condition on spine}. By \cite[Lemma~4.2]{ShiBRW}, to show that $(Z_{\theta^*}^{(b)}(t), t \geq 0)$ is uniformly integrable, and $\bE_{x,i}\left[\frac{Z_{\theta^*}^{(b)}(\infty)}{Z_{\theta^*}^{(b)}(0)}\right] = 1$, it suffices to prove that 
	\begin{equation}\label{eqn:aimui}
        \limsup_{t \to \infty} \widetilde{\bE}_{x,i}^{\uparrow} \left[ Z_{\theta^*}^{(b)}(t) \,\middle|\, \mathcal{G}\right] < \infty \quad \widetilde{\bP}_{x,i}^{\uparrow}\text{-a.s.}
	\end{equation}
    
	By the spinal decomposition, we have
	\begin{align*}
	\widetilde{\bE}_{x,i}^{\uparrow} \left[ Z_{\theta^*}^{(b)}(t) \,\middle|\,  \mathcal{G} \right] =&\, R(b + X_{\xi}(t)) e^{-X_{\xi}(t)}\\
	&+ \int_{[0,t]\times \mathbb{Z}_{>0} } k R_{J_{\xi}(s)}(b+ X_{\xi}(s)) V_{J_{\xi}(s)}(\theta^*)  e^{-X_{\xi}(s)}  \mathbf{M}(\dd s,  \dd k).
	\end{align*}   
	As $X_{\xi}(t)$ under $\widetilde{\bP}^{\uparrow}_{x,i}$ is a MAP conditioned to stay above $-b$,  by \cite[Proposition~33]{DLK17} we have \ $\lim_{t \to \infty} X_{\xi}(t) = +\infty$ $\widetilde{\bP}^{\uparrow}_{x,i}$-a.s. So the first term above converges to $0$, $\widetilde{\bP}^{\uparrow}_{x,i}$-a.s.
     
	To deal with the second term, we divide the above integral into two parts 
	\begin{align*}
	     A_1 &:= \int_{\RR_+ \times \mathbb{Z}_{>0} } k\ind{k \leq e^{\frac{1}{3}X_{\xi}(s)}} R_{J_{\xi}(s)}(b+ X_{\xi}(s)) V_{J_{\xi}(s)}(\theta^*)  e^{-X_{\xi}(s)}  \mathbf{M}(\dd s, \dd k), \\
	     A_2 &:=\int_{\RR_+ \times \mathbb{Z}_{>0} } k\ind{k > e^{\frac{1}{3}X_{\xi}(s)}} R_{J_{\xi}(s)}(b+ X_{\xi}(s)) V_{J_{\xi}(s)}(\theta^*)  e^{-X_{\xi}(s)}  \mathbf{M}(\dd s, \dd k).
	\end{align*}
    We now prove that $A_1$ and $A_2$ are both $\widetilde{\bP}^{\uparrow}_{x,i}$-a.s.\@ finite. 
 
    Recall that $\mathbf{M}$ is a Cox process, in the sense that  given $(X_{\xi}(t), {J}_{\xi}(t))_{t\ge 0}$, the conditional distribution of $\mathbf{M}$ is a Poisson point process with intensity $\widetilde{\beta}_{{J}_{\xi}(t)} \dd t \otimes \dd \widetilde{\mu}_{{J}_{\xi}(t)}$, with each $\widetilde{\mu}_{i}$ the size-biased offspring distribution. Let $C_{\mathcal{I}}:= \max_{i\in\mathcal{I}}\widetilde{\beta}_i > 0$. Then the intensity is bounded by
    $C_{\cI} \dd t \otimes \dd \widetilde{\mu}_{{J}_{\xi}(t)}\ind{\widetilde{\beta}_{J_{\xi}(t)}>0}$. The compensation formula leads to 
    \begin{align*}
    &\hspace{1.5em} \widetilde{\bE}^{\uparrow}_{x,i}[A_1]\\ &\le C_{\mathcal{I}}\widetilde{\bE}^{\uparrow}_{x,i}\bigg[ \int_{\RR_+} \!\!R_{J_{\xi}(s)}(b\!+\! X_{\xi}(s)) V_{J_{\xi}(s)}(\theta^*) e^{-X_{\xi}(s) } \Big( \sum_{k\ge 1} \widetilde{\mu}_{J_s}(k) k \ind{k \leq e^{\frac{1}{3}X_{\xi}(s)}}\Big)  \ind{\widetilde{\beta}_{J_{\xi}(s)}>0}\dd s   \bigg] \\
    &\le  C_{\mathcal{I}}\int_{\RR_+}\widetilde{\bE}^{\uparrow}_{x,i}\left[R_{J_{\xi}(s)}(b+ X_{\xi}(s)) V_{J_{\xi}(s)}(\theta^*) e^{-\frac{2}{3}X_{\xi}(s) }\ind{\widetilde{\beta}_{J_{\xi}(s)}>0}\right]\dd s\\
    &=   C_{\mathcal{I}} \int_{\RR_+}  \widetilde{\bE}_{x,i} \left[ \frac{R_{J_{\xi}(s)}(b+ X_{\xi}(s))}{R_{i}(b+x)} R_{J_{\xi}(s)}(b+ X_{\xi}(s))\ind{\inf_{r\le s}X_{\xi}(r)\ge -b} V_{J_{\xi}(s)}(\theta^*) \right.\\
    &\hspace{8em}\left.e^{-\frac{2}{3}X_{\xi}(s) }\ind{\widetilde{\beta}_{J_{\xi}(s)}>0}  \right]\dd s.
    \end{align*}
    Since $\int  R_{j}^2(b+ x) V_{j}(\theta^*) e^{-x}  e^{x/3}  \dd x \le C \int  x^2 e^{-2x/3} \dd x<\infty$, it follows from Lemma~\ref{lem:integral-cond} that $\widetilde{\bE}^{\uparrow}_{x,i}[A_1]<\infty$. Therefore, we have $A_1<\infty$, $\widetilde{\bP}^{\uparrow}_{x,i}$-a.s. 

    On the other hand,  to prove $A_2$ is $\widetilde{\bP}^{\uparrow}_{x,i}$-a.s.\@ finite, it suffices to prove that the following integral is finite: 
    \begin{equation}\label{eqn:total mass M}
        I := \int_{\RR_+ \times \mathbb{Z}_{>0}} \ind{k > e^{\frac{1} {3}X_{\xi}(s)}} \mathbf{M}(\dd s, \dd k)<\infty,\quad \widetilde{\bP}^{\uparrow}_{x,i} \text{-a.s.}
    \end{equation}
    Again, by the compensation formula, we have 
    \begin{align*}
    &\widetilde{\bE}^{\uparrow}_{x,i}[I] \le C_{\mathcal{I}}\widetilde{\bE}^{\uparrow}_{x,i}\bigg[ \int_{\RR_+}  \sum_{k\ge 0} \widetilde{\mu}_{J_s}(k) \ind{k > e^{\frac{1}{3}X_{\xi}(s)}} \ind{\widetilde{\beta}_{J_{\xi}(s)}>0}\dd s   \bigg] \\ 
    &\le C_{\mathcal{I}}\int_{\RR_+} \widetilde{\bE}^{\uparrow}_{x,i}\bigg[ \sum_{k\ge 0} \widetilde{\mu}_{J_s}(k) \ind{k > e^{\frac{1}{3}X_{\xi}(s)}}\ind{\widetilde{\beta}_{J_{\xi}(s)}>0}\bigg]\dd s\\
    &=  C_{\mathcal{I}} \int_{\RR_+}\widetilde{\bE}_{x,i}\bigg[ \frac{R_{J_{\xi}(s)}(b+X_{\xi}(s))}{R_i(b+x)}\ind{\inf_{r\le s}X_{\xi}(r) \ge -b}\sum_{k\ge 0} \widetilde{\mu}_{J_s}(k) \ind{k > e^{\frac{1}{3}X_{\xi}(s)}}\ind{\widetilde{\beta}_{J_{\xi}(s)}>0}\bigg]\dd s. 
    \end{align*}
    Applying Lemma~\ref{lem:integral-cond} with $f(r, j):= (b+r)\sum_{k\ge 1} \widetilde{\mu}_{j}(k)  \ind{3\log k > r}\ind{\widetilde{\beta}_j>0}$, we obtain 
    \begin{align*}
    \int_{-b}^{\infty} f(r, j) \dd r & =\int_{-b}^{\infty}(b+r)\sum_{k\ge 1} \widetilde{\mu}_{j}(k)  \ind{3\log k > r}\ind{\widetilde{\beta}_j>0} \dd r \\
    &\le  \frac{1}{2} \sum_{k\ge 1}\max_{i\in\cI} \left\{\widetilde{\mu}_i(k)\ind{\widetilde{\beta}_i>0}\right\}(3 \log k+b)^2. 
    \end{align*}
    This is finite under the assumption \eqref{eq:dev-ui}, as we have 
    \begin{align}
    \sum_{k\ge 1} \max_{i\in\cI} \left\{\widetilde{\mu}_i(k)\ind{\widetilde{\beta}_i>0}\right\} (\log k)^2 
    &\le \sum_{i\in\mathcal{I}} \sum_{k\ge 1}  \widetilde{\mu}_i(k)\ind{\widetilde{\beta}_i>0} (\log k)^2 \\
    &= \sum_{i\in\mathcal{I}} \frac{1}{m_i}\ind{\widetilde{\beta}_i>0}\sum_{k\ge 1}  {\mu}_i(k)k (\log k)^2  <\infty.  
    \end{align}
    Applying Lemma~\ref{lem:integral-cond} with $f(r, j)$ leads to $\widetilde{\bE}^{\uparrow}_{x,i}[I]<\infty$. This completes the proof of \eqref{eqn:aimui}, thereby establishing the desired result.
\end{proof}

\subsection{Proof of Theorem~\ref{thm:asymptotTruncated}}\label{subsec:truncatedMartingale}
\begin{proof}[Proof of Theorem \ref{thm:asymptotTruncated}]
    Applying \eqref{eq:min_-infty} to the critical parameter $\theta^*=1$, we deduce that 
    \begin{equation}\label{eqn: infinite minimum}
        \lim_{t\to +\infty}\inf_{u\in \mathcal{N}_t} X_u(t) = +\infty,\quad 
        \inf_{t\ge 0}\inf_{u\in \mathcal{N}_t} X_u(t) >-\infty, \qquad \mathbf{P}_{x,i}\text{-a.s. on } \mathscr{S}.
    \end{equation}
    Fix $\varepsilon >0$. By \eqref{eqn: infinite minimum}, there exists $b = b(\varepsilon)\ge 0$ such that 
    \[
    \mathbf{P}_{x,i}( \inf_{t\ge 0}\inf_{u\in \mathcal{N}_t} X_u(t) \ge -b\mid\mathscr{S})\ge 1-\varepsilon.
    \]

    We consider the truncated martingale $Z_{\theta^*}^{(b)}$. 
    Recall that, for any fixed $\delta>0$, we have  by Lemma \ref{lem:renewal theorem} that, for all $u$ sufficiently large, $(c_{\text{ren}}-\delta)u \le R_i(u)\le (c_{\text{ren}}+\delta)u$, for all $i\in \mathcal{I}$. We define 
    \begin{equation}\label{eqn:alternative derivative martingale}
        Z'_{\theta^*}(t) :=\sum_{u \in \mathcal{N}_t} X_u(t)e^{- X_u(t)}V_{J_u(t)}.
    \end{equation}
    Then on the event $\mathscr{S} \cap \left\{ \inf_{t\ge 0}\inf_{u\in \mathcal{N}_t} X_u(t) \ge -b\right\}$, when $t$ is sufficiently large, we have
    \begin{equation*}
        (c_{\text{ren}}-\delta)Z'_{\theta^*}(t) + bW_{\theta^*}(t) \le Z_{\theta^*}^{(b)}(t) \le (c_{\text{ren}}+\delta)Z'_{\theta^*}(t) + bW_{\theta^*}(t),
    \end{equation*}
    or equivalently 
    \begin{equation*}
        \frac{1}{c_{\text{ren}}+\delta}(Z_{\theta^*}^{(b)}(t) - bW_{\theta^*}(t)) \le Z'_{\theta^*}(t) \le \frac{1}{c_{\text{ren}}-\delta} (Z_{\theta^*}^{(b)}(t)- bW_{\theta^*}(t)).
    \end{equation*}
    As $t\to \infty$, we have by Theorem \ref{Biggins_theorem} (the critical case) that $W_{\theta^*}(t)$ converges to $0$, ${\mathbf{P}}_{x,i}$-a.s. and that that $Z_{\theta^*}^{(b)}(t)$ converges to a nonnegative limit $Z_{\theta^*}^{(b)}(\infty)$, ${\mathbf{P}}_{x,i}$-a.s. on $\mathscr{S}$.
   Then letting $\delta\to 0$, we conclude that 
   \[
   \lim_{t\to\infty}Z'_{\theta^*}(t) = \frac{1}{c_{\text{ren}}}Z_{\theta^*}^{(b)}(\infty), ~\text{on}~ \left\{\inf_{t\ge 0}\inf_{u\in \mathcal{N}_t} X_u(t)\ge -b\right\} \cap \mathscr{S}.
   \]
    
    We claim that $\lim_{t\to\infty}(Z_{\theta^*}(t)-Z'_{\theta^*}(t)) = 0$, on $\left\{\inf_{t\ge 0}\inf_{u\in \mathcal{N}_t} X_u(t)\ge -b\right\}\cap \mathscr{S}$. 
    By definition, we have 
    \begin{equation}\label{eqn:difference of Z}
        Z_{\theta^*}(t)-Z'_{\theta^*}(t) = -\sum_{u \in \mathcal{N}_t} e^{- X_u(t)}V'_{J_u(t)}(\theta^*).
    \end{equation}
    Since $\mathcal{I}$ is a finite set, $\left\{\frac{V'_i(\theta^*)}{V_i(\theta^*)}\right\}_{i\in\mathcal{I}}$ is a bounded set. Then \eqref{eqn:difference of Z} is dominated by the additive martingale $W(t)$, thus converges to $0$, ${\mathbf{P}}_{x,i}$-a.s.

    Therefore we have already proved, on the event $\left\{\inf_{t\ge 0}\inf_{u\in \mathcal{N}_t} X_u(t)\ge -b\right\}\cap \mathscr{S}$, $c_{\text{ren}}\cdot \lim_{t\to\infty}Z_{\theta^*}(t) = Z_{\theta^*}^{(b)}(\infty)$. Finally,  letting $b\to+\infty$, as  $\lim_{b\to \infty}\mathbf{P}_{x,i}(\inf_{t\ge 0}\inf_{u\in \mathcal{N}_t} X_u(t) \ge -b\big|\mathscr{S})= 1$, 
    we deduce the $\mathbf{P}_{x,i}$-a.s.\ convergence on $\mathscr{S}$ of the derivative martingale $Z$ and  the identity 
	\[
	c_{\text{ren}} \cdot Z_{\theta^*}(\infty) = \lim_{b\to \infty} Z^{(b)}_{\theta^*}(\infty), \quad \bP_{x,i}\text{-a.s.}
	\]

    Moreover, if \eqref{eq:dev-ui} holds, then by Theorem \ref{thm:L1_conv_Zb}, $\bE_{x,i}\left[\frac{Z_{\theta^*}^{(b)}(\infty)}{Z_{\theta^*}^{(b)}(0)}\right] = 1$. The  identity above implies that $\bP_{x,i}(Z_{\theta^*}(\infty)>0)>0$. 
    We claim that $\bP_{x,i}(Z_{\theta^*}(\infty) = 0)$ does not depend on $x$. 
    In fact, $\forall x\in \mathbb{R}$, we have
    \[Z_{\theta^*}(\infty) = \lim_{t\to \infty}Z_{\theta^*}(t) = e^{-x}\lim_{t\to \infty}\sum_{u\in \mathcal{N}_t} (X_u(t)-x)e^{-(X_u(t)-x)}V_{J_u(t)}(\theta^*).\]
    Therefore, 
    \begin{align}
        \bP_{x,i}\left(Z_{\theta^*}(\infty) = 0\right) &= \bP_{x,i}\left(\lim_{t\to \infty}\sum_{u\in \mathcal{N}_t} (X_u(t)-x)e^{-(X_u(t)-x)}V_{J_u(t)}(\theta^*) = 0\right)\\
        &= \bP_{0,i}(Z_{\theta^*}(\infty) = 0).
    \end{align}
    Similar to the analysis in the proof of Theorem~\ref{Biggins_theorem}, denote $\varpi(i):= \bP_{x,i}(Z_{\theta^*}(\infty) = 0)$. Then $(\varpi(i))_{i\in\mathcal{I}}$ will satisfy equation \eqref{eqn:surv_prob}, which yields $\vec{\varpi} = \vec{\mathbf{q}}$. On the other hand, $\mathscr{S}^c \subseteq \{Z_{\theta^*}(\infty) = 0\}$. Consequently, we have $\bP_{x,i}\left(\{Z_{\theta^*}(\infty)>0\} \Delta \mathscr{S}\right) = 0$.
    \end{proof}

\section{FKPP equations and travelling waves}\label{section: FKPP equations and travelling waves}

Recall that, for $\rho\in \RR$ and a function $\Phi\in \TOne$, with $\TOne$ given by \eqref{eq:TOne}, $\Phi$ is called a travelling wave solution of the FKPP equation \eqref{int-eq} with speed $\rho$, if $u(t,x,i):= \Phi(x - \rho t,i)$ is a solution of \eqref{int-eq}, i.e.\ $\Phi$ satisfies the equation 
\begin{align*}
   &\Phi(x,i) = \EE_{x+\rho t,i} \left[\Phi(\chi_i(t),i) \right] +\int_0^{t}  \EE_{x+\rho t,i} \bigg[\beta_i g_i (\Phi(\chi_i(t-s)-\rho s,i)) \\
   &\hspace{6em} +  \sum_{j\ne i} q_{ij}\Phi(U_{ij}+\chi_i(t-s)-\rho s,j) - (q_i +\beta_i) \Phi(\chi_i(t-s)-\rho s,i) \bigg] \dd s.
\end{align*}
By setting $\widehat{\chi}_i(t):= \chi_i(t) +\rho t$, we have equivalently
\begin{align}
    \Phi(x,i)
    =&\, \EE_{x,i} \left[\Phi(\widehat{\chi}_i(t),i) \right] +\int_0^{t}  \EE_{x,i} \bigg[\beta_i g_i (\Phi(\widehat{\chi}_i(t\!-\!s),i)) \\
    &\qquad +  \sum_{j\ne i} q_{ij}\Phi(U_{ij}+\widehat{\chi}_i(t\!-\!s),j) - (q_i +\beta_i) \Phi(\widehat{\chi}_i(t\!-\!s),i) \bigg] \dd s.\label{eq:TW-int}
\end{align}
Note that the generator of $\widehat \chi_i$ is ${\cal A}_i+\rho\frac{\partial}{\partial x}$, with $\mathcal{A}_i$ given by \eqref{eqn:generator Ai}. Then the travelling wave $\Phi$ is a mild solution of the following equations: for $i\in \cI$, 
\begin{align*}\label{eq:TW}
    0 = &\, \left({\cal A}_i+\rho\frac{\partial}{\partial x}\right) \Phi(x,i) + \int_{\RR}  \sum_{j\ne i} q_{ij} \big(\Phi(x+ y, j) - \Phi(x,i)\big) \mathbb{P}(U_{ij}\in \dd y) \\
    &\,+ \beta_i \big( g_i(\Phi(x,i)) - \Phi(x,i) \big) .
\end{align*} 
In this section we prove Theorem~\ref{thm:TW}. 
We first present in Section~\ref{sec:TW-mart} the multiplicative martingale that is related to a travelling wave, and then prove the existence and uniqueness in Sections~\ref{sec:TW-e} and \ref{sec:TW-u} respectively. 

\subsection{Martingale problem}\label{sec:TW-mart}
Let us first build a connection between a branching MAP and a FKPP equation \eqref{int-eq}. 

\begin{proposition}\label{prop:FKPP}
Let $\mathbf{u}_0 \colon \RR \times \mathcal{I} \to [0,1]$ be a measurable function. 
For $t\ge 0$ $x\in \RR$ and $i\in \cI$, we define 
\begin{equation}
    \mathbf{u}(t,x,i) = \mathbf{E}_{x,i} \bigg[ \prod_{u\in \mathcal{N}_t} \mathbf{u}_{0}(X_u(t),J_u(t)) \bigg],
\end{equation}
with the usual convention $\prod_{i\in \emptyset} c_i = 1$.
Then the function $\mathbf{u}$ is a solution of the FKPP equation \eqref{int-eq}. 
\end{proposition}
\begin{proof}
Recall that, the initial particle of type $i\in \cI$ moves according to a L\'evy process $\chi_i$. 
By decomposition at the first time when it branches or switches to a different type, we have  
\begin{align*}
    \mathbf{u}(t,x,i)
    =  e^{-(q_i+\beta_i)t}\EE_{x,i} \left[\mathbf{u}_0(\xi_i(t),i) \right]
    & +\int_0^t e^{-(q_i+\beta_i)r} \sum_{j\ne i} q_{ij}\EE_{x,i} \left[\mathbf{u}(t-r,U_{ij}+\chi_i(r),j)\right] \dd r \\
     & + \int_0^t e^{-(q_i+\beta_i)r} \beta_i \EE_{x,i} \left[ g_i (\mathbf{u}(t-r,\chi_i(r),i))\right] \dd r . 
\end{align*}
Applying this expression to $\mathbf{u}(t-s,\chi_i(s), i)$ leads to 
\begin{align}\label{eq:FKPP-p1}
    &\mathbf{u}(t-s,\chi_i(s),i)  =   e^{-(q_i+\beta_i)(t-s)}\EE_{\chi_i(s),i} \left[\mathbf{u}_0(\chi_i(t-s),i) \right] \\
    &  + \int_0^{t-s} e^{-(q_i+\beta_i)r}  \EE_{\chi_i(s),i} \bigg[\beta_i g_i (\mathbf{u}(t-s-r,\chi_i(r),i))+  \sum_{j\ne i} q_{ij} \mathbf{u}(t-s-r,U_{ij}+\chi_i(r),j)\bigg] \dd r. \nonumber
\end{align}
Note that, by the Markov property of a L\'evy process,  we have
\begin{equation}
\EE_{x,i} \left[\mathbf{u}_0(\chi_i(t),i) \right]
  = \EE_{x,i}  \left[ \EE_{\chi_i(s),i} \left[ \mathbf{u}_{0}(\chi_i(t-s),i)  \right] \right], \qquad \forall s\in [0,t]. 
\end{equation}
Taking expectation to \eqref{eq:FKPP-p1} and changing variables with $w= s+r$, with an application of the Markov property, we deduce that 
\begin{align}\label{eq:FKPP-p2}
    &\EE_{x,i} \left[\mathbf{u}(t-s,\chi_i(s),i) \right] =  e^{-(q_i+\beta_i)(t-s)}\EE_{x,i} \left[\mathbf{u}_0(\chi_i(t),i) \right] \\
    &\qquad+ \int_s^{t} \!\!\! e^{-(q_i+\beta_i)(w-s)}  \EE_{x,i} \bigg[\beta_i g_i (\mathbf{u}(t-w,\chi_i(w),i))+  \sum_{j\ne i} q_{ij} \mathbf{u}(t-w,U_{ij}+\chi_i(w),j)\bigg] \dd w. \nonumber
\end{align}
Integrating \eqref{eq:FKPP-p2} over $s$, then adding to $\mathbf{u}(t,x,i)$ and using Fubini's theorem, we conclude that 
\begin{align*}
 & \mathbf{u}(t,x,i)+ (q_i +\beta_i) \int_0^t \EE_{x,i} \left[\mathbf{u}(t-s,\chi_i(s),i) \right]\\
 =& \, \EE_{x,i} \left[\mathbf{u}_0(\chi_i(t),i) \right]
      + \int_0^{t}  \EE_{x,i} \bigg[\beta_i g_i (\mathbf{u}(t-w,\chi_i(w),i))+  \sum_{j\ne i} q_{ij} \mathbf{u}(t-w, U_{ij}+\chi_i(w),j)\bigg] \dd w.
\end{align*}
A change of variables yields the equation. 
\end{proof}

\begin{proposition}[Martingale problem]\label{prop_martingale_problem}
     For any $\theta\in (0,\bar{\theta})$, a function $\Phi_{\theta}\in \TOne$ is a travelling wave with speed $\rho_{\theta}=\frac{\lambda(\theta)}{\theta}$, if and only if 
    \begin{equation}
        \prod_{u\in \mathcal{N}_t} \Phi_{\theta}(X_u(t)+\rho_{\theta} t, J_u(t)), \qquad t\ge 0,
    \end{equation}
    is a martingale under $\mathbf{E}_{x,i}$, for any $i\in \cI$. 
\end{proposition}
\begin{proof}
    Let $\Phi_{\theta}\in\TOne$ be a solution of the martingale problem with speed $\rho_{\theta}$. 
    By the martingale property, for every $x\in \RR$ and $t\ge 0$ we have the identity
    \[
    \Phi_{\theta}(x - \rho_{\theta} t, i) 
    =\mathbf{E}_{x - \rho_{\theta} t,i}  \Big[ \prod_{u\in \mathcal{N}_t} \Phi_{\theta}(X_u(t)+\rho_{\theta} t, J_u(t))\Big]
    =\mathbf{E}_{x,i}  \Big[ \prod_{u\in \mathcal{N}_t} \Phi_{\theta}(X_u(t), J_u(t))\Big].  
    \] 
	It follows from Proposition~\ref{prop:FKPP} that $u(t,x,i):= \Phi_{\theta}(x - \rho_{\theta} t,i) $ is a solution of the FKPP with initial condition $u(0,x,i)=\Phi_{\theta}(x,i)$, so by definition $\Phi_{\theta}$ is a travelling wave. 

    Conversely, let $\Phi_{\theta}$ be a  travelling wave  with speed $\rho_{\theta}$.  Our goal is to prove that $\prod_{u\in \mathcal{N}_t} \Phi_{\theta} ( X_u(t) +\rho_{\theta} t ,J_u(t))$ is a martingale. Applying Proposition~\ref{prop:FKPP} to the branching MAP $(\widehat{X}_u(t) = X_u(t) +\rho_{\theta} t, t\ge 0)$, we have the identity
    \begin{align*}
         \mathbf{E}_{x,i} \bigg[\prod_{u\in \mathcal{N}_t} \Phi_{\theta}( x+\widehat{X}_u(t), J_u(t)) \bigg]  =  \EE_{x,i} \left[\Phi_{\theta}(\widehat{\chi}_i(t),i) \right]  + \int_0^{t}  \EE_{x,i} \bigg[\beta_i g_i (\Phi_{\theta}(\widehat\chi_i(t-s),i))& \\
         +  \sum_{j\ne i} q_{ij} \Phi_{\theta}(U_{ij}+\widehat\chi_i(t-s),j) - (q_i +\beta_i) \Phi_{\theta}(\widehat\chi_i(t-s),i) \bigg] \dd s.&
    \end{align*}
    Due to \eqref{eq:TW-int}, the latter is equal to $\Phi_{\theta}( x,i)$. Therefore, 
    \begin{equation}
        \mathbf{E}_{x,i} \Big[\prod_{u\in \cN_t} \Phi_{\theta}( x+\widehat{X}_u(t), J_u(t)) \Big]=\Phi(x,i)
    \end{equation} 
    is a constant for every $t$. It follows from the Markov property that $     \prod_{u\in \cN_t} \Phi_{\theta}( \widehat{X}_u(t), J_u(t))$ is a martingale under $\mathbf{P}_{x,i}$.  
\end{proof}

\subsection{Proof of existence of travelling waves}\label{sec:TW-e}
We assume \ref{A1} \ref{A2} and show the existence of travelling wave solution with speed $\rho>\frac{\lambda(\theta^*)}{\theta^*}$. 
Recall that, by the convexity of $\lambda$, 
 $\theta \mapsto \frac{\lambda(\theta)}{\theta}$ strictly decreases from $+\infty$ to $\frac{\lambda(\theta^*)}{\theta^*}$ on $(0,\theta^*]$. Therefore, for any $\rho>\frac{\lambda(\theta^*)}{\theta^*}$, there exists a unique $\theta\in (0,\theta^*)$ such that $\rho = \frac{\lambda(\theta)}{\theta}$. 
Recall that $W_{\theta}(t) = e^{-\lambda(\theta)t} \sum_{u\in \mathcal{N}_t} e^{-\theta X_u(t)} V_{J_u(t)}(\theta)$ and $W_{\theta}(\infty) = \lim_{t\rightarrow\infty} W_{\theta}(t)$ almost surely under $\bP_{x,i}$. 

\begin{lemma}\label{lemma_travelling_exist_super}
Suppose $\theta\in (0,\bar{\theta})$ with $\theta \lambda'(\theta) <\lambda(\theta)$ and $\ind{\beta_jm_j>0}\sum_{k\ge 1}(k \log k)\mu_j(k)<\infty$ for all $j\in \cI$. Define a function $\Phi_{\theta}\colon \mathbb{R}\times \cI \to [0,1]$ by $ (x,i)\mapsto \Phi_{\theta}(x,i) = \bE_{x,i} \big[e^{-W_{\theta}(\infty)}\big] = \bE_{0,i}\left[ e^{-e^{-\theta x} W_{\theta}(\infty)} \right]$. Then $\Phi_{\theta}$ is a travelling wave solution with speed $\frac{\lambda(\theta)}{\theta}$.
\end{lemma}
\begin{proof}
First, it follows from the spatial homogeneity of the branching MAP that
\begin{equation}
    \left( W_{\theta}(\infty), \bP_{x,i} \right) \overset{d.}{=} \left( e^{-\theta x} W_{\theta}(\infty), \bP_{0,i} \right).
\end{equation}
This yields that $\bE_{x,i} \big[e^{-W_{\theta}(\infty)}\big] = \bE_{0,i}\left[ e^{-e^{-\theta x} W_{\theta}(\infty)} \right]$ and $\Phi_{\theta}(x,i)$ is well defined.
Moreover,  when $\theta \lambda'(\theta) < \lambda(\theta)$ and $\sum_{k\ge 1} (k \log k) \mu_j(k) < \infty$ for all $j \in \mathcal{I}$, it follows from Theorem \ref{Biggins_theorem} that $W_{\theta}(\infty)$ is non-degenerate and $\bP_{0,i}(W_{\theta}(\infty) = 0) =\bq_i$. Thus we have
\begin{equation}
    \lim_{x\rightarrow -\infty} \Phi_{\theta}(x,i) = \lim_{x\rightarrow -\infty} \bE_{0,i} \left[ e^{-e^{-\theta x} W_{\theta}(\infty)} \right] = \bq_i \mbox{ and } \lim_{x\rightarrow -\infty} \Phi_{\theta}(x,i) = 1.
\end{equation}
Since $\Phi_{\theta}(x,i) =  \bE_{0,i}\left[ e^{-e^{-\theta x} W_{\theta}(\infty)} \right]$ and $W_{\theta}(\infty)$ is non-negative, we get that $x\mapsto \Phi_{\theta}(x,i)$ is non-decreasing in $x$. Thus, $\Phi_{\theta} \in \mathcal{T}_1$.

By the decomposition at time $s\ge 0$, we deduce that, $\bP_{x,i}$-a.s.,
\begin{equation}\label{W_distribution}
    W_{\theta}(\infty) = e^{-\lambda(\theta)s} \sum_{u\in \cN_s} W_{\theta}^{(u)}(\infty),
\end{equation}
where $W_{\theta}^{(u)}(\infty)$ is the limit of the additive martingale for the branching Markov  additive process starting from $(X_u(s), J_u(s))$;  given $\cF_s$, $\{W_{\theta}^{(u)}(\infty): u\in \cN_s\} $ are conditionally independent.

For $t\ge 0$, $x\in \RR$ and $i\in \cI$, with $\Phi_{\theta}$ in the statement, we define
\begin{align}
    u(t,x,i) :&= \Phi_{\theta}\Big(x-\frac{\lambda(\theta)}{\theta}t, i\Big)  =\bE_{0,i} \left[ \exp\left\{ -e^{\lambda(\theta)t-\theta x} W_{\theta}(\infty) \right\} \right]\\
    &=  \bE_{x,i} \left[ \exp\left\{ -e^{\lambda(\theta)t} W_{\theta}(\infty) \right\} \right].
\end{align}
Let $0\le s\le t$, then by \eqref{W_distribution} and the branching property, we get that
\begin{align}
      u(t,x,i) &=\bE_{x,i} \left[\exp\bigg\{ -e^{\lambda(\theta)t} e^{-\lambda(\theta)s} \sum_{u\in \cN_s} W_{\theta}^{(u)}(\infty) \bigg\}\right]\\
     &= \bE_{x,i}\bigg[ \prod_{u\in \cN_s} \bE_{X_{u}(s),J_u(s)} \left[\exp\left\{ -e^{\lambda(\theta)(t-s)} W_{\theta}^{(u)}(\infty) \right\}\right]\bigg]\\
     &= \bE_{x,i} \bigg[\prod_{u\in \cN_s}  u(t-s,X_u(s),J_u(s))\bigg].
\end{align}
In particular, setting $s=t$, we have
\begin{equation}
    u(t,x,i) = \bE_{x,i} \bigg[\prod_{u\in \cN_t}  u(0,X_u(t),J_u(t))\bigg] = \bE_{x,i} \bigg[\prod_{u\in \cN_t}  \Phi_{\theta}(X_u(t),J_u(t))\bigg].
\end{equation}
It follows from Proposition \ref{prop:FKPP} that $u(t,x,i)$ satisfies the FKPP equation. 
Recall that $u(t,x,i) := \Phi_{\theta}\Big(x-\frac{\lambda(\theta)}{\theta}t,i\Big)$, then $\Phi_{\theta}$ is a travelling wave with speed $\frac{\lambda(\theta)}{\theta}$ by definition. 
\end{proof}

Recall that $Z_{\theta^*}(\infty)$ is the limit of the derivative martingale with the critical parameter. 
\begin{lemma}\label{lemma_travelling_exist_critical}
     Suppose $\ind{\beta_jm_j>0}\sum_{k\ge 1}k(\log k)^2\mu_j(k)<\infty$ for all $j\in \cI$. Then the function $\Phi_{\theta^*}(x,i) = \bE_{x,i} \big[e^{-Z_{\theta^*}(\infty)}\big] = \bE_{0,i}\left[ e^{-e^{-\theta x} Z_{\theta^*}(\infty)} \right]$ is a travelling wave solution with speed $\lambda'(\theta^*) = \frac{\lambda(\theta^*)}{\theta^*}$.
\end{lemma}
\begin{proof}
Using the argument similar to the proof of Lemma \ref{lemma_travelling_exist_super}, we have that $\Phi_{\theta^*}$ is well defined and $\Phi_{\theta^*}\in\mathcal{T}_1$.
Recall that 
\begin{equation}
    Z_{\theta^*}(t) = e^{-\lambda(\theta^*)t} \sum_{u\in \cN_t} e^{-\theta^* X_u(t)} \left[V_{J_u(t)}(\theta^*)(X_u(t)+\lambda'(\theta^*)t) - V'_{J_u(t)}(\theta^*)  \right]
\end{equation}
and $Z_{\theta^*}(\infty) = \lim_{t\rightarrow\infty} Z_{\theta^*}(t)$ almost surely under $\bP_{x,i}$. 
Therefore, we know that under $\bP_{x,i}$,
\begin{equation}
    Z_{\theta^*}(\infty) \overset{d.}{=} e^{-\lambda(\theta)s} \sum_{u\in \cN_s} \left(Z_{\theta^*}^{(u)}(\infty)+ \lambda'(\theta^*)s W_{\theta^*}^{(u)}(\infty)\right),
\end{equation}
where $W_{\theta^*}^{(u)}(\infty)$ and $Z_{\theta^*}^{(u)}(\infty)$ are the limits of the additive martingale and derivative martingale for the Markov branching additive process starting from $(X_u(s), J_u(s))$, respectively, and given $\cF_s$, $\{(W_{\theta^*}^{(u)}(\infty),Z_{\theta^*}^{(u)}(\infty): u\in \cN_s\} $ are independent. Since
$\lim_{t\rightarrow\infty} W_{\theta^*}(t) = 0$ almost surely, we have
\begin{equation}\label{Z_distribution}
    Z_{\theta^*}(\infty) \overset{d.}{=} e^{-\lambda(\theta)s} \sum_{u\in \cN_s} Z_{\theta^*}^{(u)}(\infty).
\end{equation}
The remaining arguments are very similar to the proof of Theorem \ref{lemma_travelling_exist_super} and we omit the details. 
\end{proof}

\begin{proof}[Proof Theorem \ref{thm:TW}: the existence part]
By Lemmas \ref{lemma_travelling_exist_super} and \ref{lemma_travelling_exist_critical}, it remains to show the non-existence of travelling waves with speed $\rho<\frac{\lambda(\theta^*)}{\theta^*}$.
The proof is an extension of classical arguments (see e.g.\ \cite{Har-FKPP}). By Corollary \ref{thm:speed of Lt} the velocity of 
the leftmost particle is $-\lambda'(\theta^*)$. As $\rho < \frac{\lambda(\theta^*)}{\theta^*}= \lambda'(\theta^*)$, we have $\mathbf{P}_{x,i}$-a.s.
\begin{equation}
     \lim_{t\to \infty} \left(\min_{u\in \mathcal{N}_t} X_u(t) + \rho t \right) = -\infty, \quad \mbox{on } \mathscr{S}. 
\end{equation}
Let $\Phi$ be a solution of the martingale problem with parameter $\rho$, then for every $t\ge 0$, it holds that  
\begin{align}      
    \Phi(x,i) &= \mathbf{E}_{x,i} \bigg[\prod_{u \in  \mathcal{N}_t} \Phi (X_u(t) +\rho t,J_u(t)) \bigg]\\
    &= \mathbf{E}_{x,i} \bigg[ \left(\ind{\mathscr{S}^c} + \ind{\mathscr{S}} \right) \prod_{u \in  \mathcal{N}_t} \Phi (X_u(t) +\rho t,J_u(t)) \bigg]\\
  	&\le \mathbf{q}_i +   \mathbf{E}_{0,i}  \Big[ \ind{\mathscr{S}} \Phi \big(x + \min_{u\in \mathcal{N}_t} X_u(t) +\rho t, J_u(t)\big) \Big].  
\end{align}
Since $\Phi\in\TOne$, we have $\mathbf{P}_{x,i}$-a.s. on $\mathscr{S}$,
\begin{equation}
    \limsup_{t\to \infty} \Phi \big(x + \min_{u\in \mathcal{N}_t} X_u(t) +\rho t, J_u(t)) 
    \le \lim_{y\to -\infty} \max_{j\in \cI}\Phi(y,j) = \max_{j\in\cI}\bq_j. 
\end{equation}
By Fatou's lemma, we deduce that $\Phi(x,i) \leq \bq_i + (1-\bq_i) \max_{j\in\cI}\bq_j$ for all $x\in \RR$ and $i\in \cI$. 
Choose $i_0\in\cI$ such that $\bq_{i_0} = \max_{j\in\cI} \bq_j$. Thus, $\Phi(x,i_0) \leq \bq_{i_0} + (1-\bq_{i_0}) \bq_{i_0} = 1 - (1-\bq_{i_0})^2 < 1$ for all $x\in\RR$,  which  contradicts to the fact that $\Phi\in \TOne$. 
\end{proof}

\subsection{Proof of uniqueness of travelling waves}\label{sec:TW-u}
We now prove the uniqueness of travelling waves, for which we further assume that the branching MAP is spectrally negative; this assumption means that, for all $i\in\cI$, the L\'evy process $\chi_i$ has no positive jumps (its L\'evy measure satisfies $\Lambda_i(0,\infty)=0$) and $\mathbb{P}(U_{ij}\leq 0) = 1$ for all $i,j\in\cI$. 
In this proof, we follow the general ideas of \cite{KYPRIANOU04, YangRen11}.

\begin{proof}[Proof Theorem \ref{thm:TW}: the uniqueness part]
We treat the supercritical ($\theta \in (0,\bar{\theta})$) and  critical ($\theta=\theta^*$) regimes separately. 

\paragraph{Supercritical regime} Let $\theta \in (0,\bar{\theta})$ with $\rho_{\theta} = \frac{\lambda(\theta)}{\theta} > \frac{\lambda(\theta^*)}{\theta^*}.$
    Consider a branching MAP under law $\mathbf{P}_{0,i}$ and define the space-time barrier 
    \begin{equation}
        \Gamma^{(x,\rho_{\theta})} := \{ (y,t)\in \mathbb{R}\times \mathbb{R}_+: y+\rho_{\theta} t = x  \} \mbox{ for } x\geq 0.
    \end{equation}
    When a particle crosses this barrier, it is stopped immediately. According to \cite[Theorem 46.2]{Sato99}, a L\'evy process without positive jumps that crosses an upper barrier from below must do so continuously. Consequently, a particle governed by such a process is stopped upon hitting the barrier.  Let $C(x,\rho_{\theta})$ denotes the random collection of particles stopped at the barrier, then is a \emph{stopping line}, as it satisfies the fundamental property that, if $u\in C(x,\rho_{\theta})$, then $v\notin C(x,\rho_{\theta})$ for all $v\prec u$; see \cite{ Chauvin91,Jagers89}. 
    Then we have the following properties.
    \begin{itemize}
        \item For any $u\in C(x,\rho_{\theta})$ and $v\prec u$, we have $v\notin C(x,\rho_{\theta})$. By \eqref{eq:min_-infty}, we have that  $\lim_{t\rightarrow\infty} (\min_{u\in \cN_t} X_u(t) + \rho_{\theta} t) = \infty$ $\bP_{x_0,i}$-a.s. on $\mathscr{S}$, then all lines of descendants from the ancestor hit $\Gamma^{(x,\rho_{\theta})}$ for all $x>x_0$.
        \item $\lim_{x\rightarrow\infty} \inf \left\{|u|: u\in C(x,\rho_{\theta}) \right\} = \infty$, where $|u|$ is the generation of the particle $u$. This follows from the fact that the number of offspring in the $n$-th generation is finite almost surely, and their life lengths are finite almost surely. Therefore, $\max\left\{ X_u(s): |u|= n, s\leq d_u \right\}$ must be finite almost surely. Then, we get that $\inf \left\{|u|: u\in C(x,\rho_{\theta}) \right\}$ tends to infinity as $x\rightarrow\infty$.
        \item For $x<y$ and any $u\in C(y,\rho_{\theta})$, there exists a unique $v\in C(x,\rho_{\theta})$ such that $v\preceq u$.
        This follows from the first point and the fact that when the particles hit the barrier. Let $\cF_{\Gamma^{(x,\rho_{\theta})}}$ be the natural filtration generated by ancestral types and spatial paths receding from particles at the moment they hit $\Gamma^{(x,\rho_{\theta})}$ (see \cite{Chauvin91,Jagers89} for precise definition). Therefore, for $x<y$, $\cF_{\Gamma^{(x,\rho_{\theta})}} \subset\cF_{\Gamma^{(y,\rho_{\theta})}}$.
        \item Let $C_j(x,\rho_{\theta})$ denotes the random collection of type $j$ particles stopped at the barrier $\Gamma^{(x,\rho_{\theta})}$. We claim that  $(\#C_j(x,\rho_{\theta}): j\in \cI)$ forms a supercritical continuous-time branching process where $x\ge 0$ plays the role of time. First, we use $\bP_{(x,t),i}$ to denote the law of branching MAP with the initial particle starting from space-time position $(x,t)$ of type $i$. Then it follows from the spatial homogeneity of MAP that for $(x_1,t_1), (x_2,t_2)\in \Gamma^{(x,\rho_{\theta})}$, we have $(\#C_j(y,\rho_{\theta}):j\in\cI; \bP_{(x_1,t_1),i}) \overset{d.}{=} (\#C_j(y,\rho_{\theta}):j\in\cI; \bP_{(x_2,t_2),i})$ with $y> x$ and any $i\in\cI$.  Moreover, the law of $(\#C_j(y,\rho_{\theta}):j\in\cI; \bP_{(x_1,t_1),i})$ only depends on the ``time" $y-x$ between the stopping lines and the type $i$ of the initial particle. 
        
        We use $\sigma_u^x$ to denote the stopping time when the particle $u$ hit the barrier $\Gamma^{(x,\rho_{\theta})}$. 
        By the strong Markov branching property on stopping lines (see \cite{Chauvin91,Neveu87} for branching Brownian motions and \cite{Kyprianou99} for branching L\'{e}vy processes), we know that given $\cF_{\Gamma^{(x,\rho_{\theta})}}$, the processes $(\#C_j(y,\rho_{\theta}):j\in\cI, y> x;\bP_{(X_u(\sigma_u^x),\sigma_u^x),J_u(\sigma_u^x)})$ for $u\in \cN_t$ are independent. Therefore, $(\#C_j(x,\rho_{\theta}): j\in \cI; \bP_{0,i})$ satisfies the definition on \cite[Page 200]{Athreya1968} and forms a continuous-time branching process indexed by $x\ge 0$. 
        
        Moreover, since all lines of descendants from the ancestor will hit the barrier $\Gamma^{(x,\rho_{\theta})}$, we know that the  process along the stopping line will survive on the event $\mathscr{S}$. Thus, the continuous-time branching process $(\#C_j(x,\rho_{\theta}): j\in \cI, x\ge 0)$ is  supercritical.
    \end{itemize}
    
    Define $m_{ij}^{\theta}(x) := \bE_{0,i} [\#C_j(x,\rho_{\theta})]$, $i,j\in\cI$.  Since the matrix $Q$ of the branching MAP is irreducible, we have that the matrix $M^{\theta}(x) := (m_{ij}^{\theta}(x))_{i,j\in\cI}$ is irreducible. By the property of a continuous-time branching process, there exists a positive matrix $A^{\theta}$ such that $M^{\theta}(x) = e^{xA^{\theta}}$. By the Perron-Frobenius theorem, $A^{\theta}$ has a positive PF eigenvalue $\eta^{\theta}$ with corresponding right and left eigenvectors $\vec{h}^{\theta} := (h^{\theta}_i:i\in \cI)$ and $\vec{\pi}^{\theta} := (\pi^{\theta}_i:i\in\cI)$ such that $\langle \vec{\pi}^{\theta}, \vec{h}^{\theta} \rangle = \langle \vec{\pi}^{\theta}, \boldsymbol{1} \rangle = 1$. Therefore,
    \begin{equation}\label{eq:PF_continuous_time_GWP}
        \sum_{j\in\cI} m_{ij}^{\theta}(x) e^{-\eta^{\theta}x} h^{\theta}_j = h^{\theta}_i, \mbox{ for all } i\in\cI.
    \end{equation}
    Furthermore, by the Kesten-Stigum theorem (see, for example, \cite[Theorem 2.1]{Georgii03}), we know that  
    \begin{equation}\label{eq:Kesten-Stigum_GWP_super2}
        \lim_{x\rightarrow\infty} \#C_i(x,\rho_{\theta}) e^{-\eta^{\theta} x} = \pi^{\theta}_i \widetilde{W}^{\theta}, \quad \mbox{ $\bP_{0,i}$-a.s.}
    \end{equation}
    for some non-negative random variable $\widetilde{W}^{\theta}$.
    
    We next define for $x\ge 0$
    \begin{equation}\label{martingale_W_stopping}
        \widetilde{W}_x(\rho_\theta) := \sum_{u\in C(x,\rho_{\theta})} e^{-\theta X_u(\sigma_u^x) - \lambda(\theta)\sigma_u^x} V_{J_u(\sigma_u^x)}(\theta) = \sum_{i\in \cI} \#C_i(x,\rho_{\theta}) V_i(\theta) e^{-\theta x}. 
    \end{equation}
    Recall by Theorem~\ref{Biggins_theorem} that $W_{\theta}(\infty) = \lim_{t\to \infty} W_{\theta}(t)$ holds  $\bP_{0,i}$-a.s.\ and in $\mathcal{L}^1(\bP_{0,i})$. 
    By the strong Markov property at the stopping line, we have that
    \begin{align}
        &\bE_{0,i}\left[ W_{\theta}(\infty) | \cF_{\Gamma^{(x,\rho_{\theta})}} \right] = \lim_{t\rightarrow\infty} \bE_{0,i}\left[ W_{\theta}(t) | \cF_{\Gamma^{(x,\rho_{\theta})}} \right]\\
        &= \lim_{t\rightarrow\infty}\bigg(\sum_{u\in C(x,\rho_{\theta}),\sigma_u^x \le t} e^{-\theta X_u(\sigma_u^x) - \lambda(\theta)\sigma_u^x} V_{J_u(\sigma_u^x)}(\theta)  +   \sum_{u\in \cN_t, \sigma_u^x > t} e^{-\theta X_u(t) -\lambda(\theta) t} V_{J_u(t)}(\theta)\bigg)\\
        &= \widetilde{W}_x(\rho_\theta), \label{eq:mart-sl}
    \end{align}
    where we used the fact $\{u\in \cN_t: \sigma_u^x>t\} \rightarrow \emptyset$ as $t\rightarrow\infty$.
    Then $(\widetilde{W}_x(\rho_\theta), x\geq 0)$ is a $\bP_{0,i}$-martingale with respect to $\{\cF_{\Gamma^{(x,\rho_{\theta})}}: x\geq 0\}$ and
    \begin{equation}\label{stopping_W_theta}
        W_{\theta}(\infty) = \lim_{x\rightarrow \infty} \widetilde{W}_x(\rho_\theta) = \lim_{x\rightarrow\infty} \sum_{i\in \cI} \#C_i(x,\rho_{\theta}) V_i(\theta) e^{-\theta x}, \quad \mbox{ $\bP_{0,i}$-a.s.\ and in $\mathcal{L}^1(\bP_{0,i})$.}
    \end{equation}
    Note that $m_{ij}^{\theta}(x) = \bE_{0,i} [\#C_j(x,\rho_{\theta})]$. Taking expectation on the both sides of \eqref{martingale_W_stopping} and then letting $x\rightarrow\infty$, we get that
    \begin{equation}\label{eq:limit_continuous_time_GWP}
        \lim_{x\rightarrow\infty} \sum_{j\in \cI} m_{ij}^{\theta}(x) e^{-\theta x} V_j(\theta) = V_i(\theta).
    \end{equation} 
    Therefore, combining \eqref{eq:limit_continuous_time_GWP} with \eqref{eq:PF_continuous_time_GWP}, we have $\eta^{\theta} = \theta$ and $h_i^{\theta} = cV_i(\theta)$.
    Then, \eqref{eq:Kesten-Stigum_GWP_super2} will be
    \begin{equation}\label{eq:Kesten-Stigum_GWP_super}
        \lim_{x\rightarrow\infty} \#C_i(x,\rho_{\theta}) e^{-\theta x} = \pi^{\theta}_i \widetilde{W}^{\theta}. 
    \end{equation}
    By \eqref{stopping_W_theta} again, 
    \begin{equation}\label{eq:relation_W_theta}
        W_{\theta}(\infty) = \sum_{i\in\cI} \frac{h^{\theta}_i}{c} \pi_i^{\theta}\widetilde{W}^{\theta} = \widetilde{W}^{\theta}/c.
    \end{equation}
    
    On the other hand, let $\Phi_{\theta}$ be a travelling wave with speed $\rho_{\theta}$. For $z\in \RR$, define 
    \begin{align}\label{stopping_product_martingale}
        \widetilde{M}_x(z,\rho_{\theta}) &:= \prod_{u\in C(x,\rho_{\theta})} \Phi_{\theta}(z+X_u(\sigma_u^x)+\rho_{\theta} \sigma_u^x, J_u(\sigma_u^x))\\
        &= \exp \bigg\{ \sum_{j\in\cI} \# C_j(x,\rho_{\theta}) \log \Phi_{\theta}(z+x,j) \bigg\}.
    \end{align}
    Then, we claim that $(\widetilde{M}_x(z,\rho_{\theta}), x\geq 0)$    
    is a $\bP_{0,i}$-martingale with respect to $\{\cF_{\Gamma^{(x,\rho_{\theta})}}: x\geq 0\}$. 
    To see this, we define 
    \begin{equation}
        M_t(z,\rho_{\theta}) := \prod_{u\in \cN_t} \Phi_{\theta}(z+X_u(t)+\rho_{\theta} t, J_u(t)), \qquad t\ge 0. 
    \end{equation}
    By Proposition~\ref{prop_martingale_problem}, $(M_t(z,\rho_{\theta}),t\geq 0)$ is a non-negative bounded martingale with respect to $\{ \cF_t:t\geq 0\}$. Therefore, $\bP_{0,i}$-a.s. 
    \begin{equation}
        M_{\infty}(z,\rho_{\theta}) := \lim_{t\rightarrow\infty} M_t(z,\rho_{\theta})
    \end{equation}
    exists and is non-degenerated. Similarly as \eqref{eq:mart-sl}, we deduce by the bounded convergence theorem and the strong Markov property that 
    \begin{align}
        &\bE_{0,i}\left[ M_{\infty}(z,\rho_{\theta}) | \cF_{\Gamma^{(x,\rho_{\theta})}} \right]\\
        &= \lim_{t\rightarrow\infty} \bE_{0,i}\left[ M_{t}(z,\rho_{\theta}) | \cF_{\Gamma^{(x,\rho_{\theta})}} \right]\\
        &= \lim_{t\rightarrow\infty}\bigg(\prod_{u\in C(x,\rho_{\theta}), \sigma_u^x \leq t} \!\!\! \Phi_{\theta}(z\!+\!X_u(\sigma_u^x)\!+\!\rho_{\theta} \sigma_u^x, J_u(\sigma_u^x))  
         \prod_{u\in \cN_t, \sigma_u^x > t} \Phi_{\theta}(z\!+\!X_u(t)\!+\!\rho_{\theta} t, J_u(t))\bigg)\\
        &= \prod_{u\in C(x,\rho_{\theta})} \Phi_{\theta}(z+X_u(\sigma_u)+\rho_{\theta} \sigma_u^x, J_u(\sigma_u^x)) = \widetilde{M}_x(z,\rho_{\theta}),
    \end{align}
    where we also used the fact that $\lim_{t\rightarrow\infty} (\min_{u\in \cN_t} X_u(t) + \rho_{\theta} t) = \infty$, $\bP_{0,i}$-a.s.\ on $\mathscr{S}$ (see \eqref{eq:min_-infty}),
    and that $\lim_{t\rightarrow \infty}\prod_{u\in \cN_t, \sigma_u^x > t} \Phi_{\theta}(z+X_u(t)+\rho_{\theta} t, J_u(t)) = 1$  on $\mathscr{S}^c$.
    Thus, $(\widetilde{M}_x(z,\rho_{\theta}), x\geq 0)$ is a $\bP_{0,i}$-martingale and converges to $M_{\infty}(z,\rho_{\theta})$ in $\mathcal{L}^1(\bP_{0,i})$ and $\bP_{0,i}$-a.s.\ as $x\to\infty$. Therefore, we have
    \begin{equation}\label{stopping_lim_M}
        \lim_{x\rightarrow \infty} -\sum_{j\in\cI} \# C_i(x,\rho_{\theta}) \log \Phi_{\theta}(z+x,j) = -\lim_{x\rightarrow \infty} \log \widetilde{M}_x(z,\rho_{\theta})
        =-\log M_{\infty}(z,\rho_{\theta}). 
    \end{equation}
       By \eqref{stopping_lim_M} and \eqref{eq:Kesten-Stigum_GWP_super}, we have
    $\alpha := \lim_{x\rightarrow\infty} -\sum_{j\in\cI} \pi^{\theta}_j e^{\theta x} \log\Phi_{\theta}(x,j)$ exists.
    Taking expectation in \eqref{stopping_product_martingale}, it follows from the bounded convergence theorem and equations \eqref{eq:Kesten-Stigum_GWP_super} \eqref{eq:relation_W_theta} that
    \begin{align}
        \Phi_{\theta}(z,i) &= \bE_{0,i} \left[\lim_{x\rightarrow \infty}\widetilde{M}_x(z,\rho_{\theta})\right] = \bE_{0,i} \bigg[\lim_{x\rightarrow\infty} \exp \bigg\{ \sum_{j\in\cI} \# C_j(x,\rho_{\theta}) \log \Phi_{\theta}(z+x,j) \bigg\} \bigg]\\
        &= \bE_{0,i} \bigg[\lim_{x\rightarrow\infty} \exp \bigg\{ \sum_{j\in\cI} \pi_j^{\theta} \widetilde{W}^{\theta} e^{\theta x} \log \Phi_{\theta}(z+x,j) \bigg\} \bigg]\\
        &= \bE_{0,i} \bigg[\lim_{x\rightarrow\infty} \exp \bigg\{ \sum_{j\in\cI} \pi_j^{\theta} cW_{\theta}(\infty) e^{\theta x} \log \Phi_{\theta}(z+x,j) \bigg\} \bigg]\\
        &= \bE_{0,i} \left[ \exp\left\{ -\alpha c W_{\theta}(\infty) e^{-\theta z} \right\} \right].
    \end{align}

\paragraph{Critical regime $\rho_{\theta^*} = \lambda'(\theta^*)=\frac{\lambda(\theta^*)}{\theta^*}$.} 
 let $\Phi_{\theta^*}$ be a travelling wave with speed $\rho_{\theta^*}$. 
     Recall that $C_i(x,\rho_{\theta^*})$ denotes the random collection of type $i$ particles stopped at the barrier $\Gamma^{(x,\rho_{\theta^*})}$. Similarly as in the supercritical case, we have that 
    \begin{align}\label{martingale_M_x_critical}
        \widetilde{M}_x(z,\rho_{\theta^*}) = \exp\bigg\{ \sum_{j\in\cI} \#C_j(x,\rho_{\theta^*}) \log \Phi_{\theta^*} (z+x,j) \bigg\}, \qquad x\ge 0
    \end{align}
    is a $\bP_{0,i}$-martingale which converges to $\Phi_{\theta^*}(z,i)$ a.s.\ and in $\mathcal{L}^1(\bP_{0,i})$. 
    
    For $b>0$, let us also add a killing barrier at $\Gamma^{(-b,\rho_{\theta^*})}$ for this branching MAP, which means the truncation as in \eqref{eqn:defTruncatedDerivativeMartingale}. Define $\widetilde{C}_{i}(x,\rho_{\theta^*})$ to be the collection of type $i$ particles that are stopped at the barrier $\Gamma^{(x,\rho_{\theta^*})}$ for the truncated branching MAP and let $\widetilde{C}(x,\rho_{\theta^*}) := \bigcup_{i\in\cI} \widetilde{C}_{i}(x,\rho_{\theta^*})$. 
    Let $\gamma^{(-b,\theta^*)}$ be the event that the branching MAP survives and remains entirely to the right of $\Gamma^{(-b,\rho_{\theta^*})}$, such that the truncation does not take effect on the event $\gamma^{(-b,\theta^*)}$. By \eqref{eq:min_-infty} applied to $\theta^*$, we know that $\bP_{0,i}(\gamma^{(-b,\theta^*)}\mid \mathscr{S}) \rightarrow 1$ as $b\rightarrow\infty$. On the event  $\gamma^{(-b,\theta^*)}$, we have $C_{i}(x,\rho_{\theta^*}) = \widetilde{C}_{i}(x,\rho_{\theta^*})$ and
    \begin{equation}\label{lim_Product_martingale_exists}
        \lim_{x\rightarrow\infty} -\sum_{j\in\cI} \#\widetilde{C}_{j}(x,\rho_{\theta^*}) \log\Phi_{\theta^*}(z+x,j)
    \end{equation}
    exists and is non-negative. Define
    \begin{align}
        Z^{(b)}_x(\rho_{\theta^*})
        &:= \sum_{u\in \widetilde{C}(x,\rho_{\theta^*})} R_{J_u(\tau_u^{(x)})} (X_u(\tau_u^{(x)})+\rho_{\theta^*}\tau_u^{(x)} +b) e^{-\theta^*X_u(\tau_u^{(x)}) -\lambda(\theta^*)\tau_u^{(x)}} V_{J_u(\tau_u^{(x)})}(\theta^*)\\
        &= \sum_{j\in\cI} R_j(x+b) \#\widetilde{C}_{j}(x,\rho_{\theta^*}) e^{-\theta^*x} V_j(\theta^*).
    \end{align}
    Let $\widetilde{\cF}_{\Gamma^{(x,\rho_{\theta^*})}}$ be the natural filtration generated by ancestral type and spatial paths receding from particles at the moment they hit $\Gamma^{(x,\rho_{\theta^*})}$ before meeting $\Gamma^{(-b,\rho_{\theta^*})}$. With similar arguments as in the proof of \eqref{eq:mart-sl}, we deduce by the strong Markov property that, $(Z^{(b)}_x(\rho_{\theta^*}), x\geq 0)$ is a $\bP_{0,i}$-martingale with respect to $\{\widetilde{\cF}_{\Gamma^{(x,\rho_{\theta^*})}}: x\geq 0 \}$ and
    \begin{equation}\label{eq:derivative_cont_time_GWP}
        \lim_{x\rightarrow\infty} \sum_{j\in\cI} R_j(x+b) \#\widetilde{C}_{j}(x,\rho_{\theta^*}) e^{-\theta^*x} V_j(\theta^*) = Z_{\theta^*}^{(b)}(\infty).
    \end{equation}
    The arguments of \eqref{martingale_W_stopping} and \eqref{stopping_W_theta} still work for $\theta = \theta^*$. Therefore, we have
    \begin{equation}\label{eq:additive_cont_time_GWP}
        \lim_{x\rightarrow\infty} \sum_{j\in\cI}\#\widetilde{C}_{j}(x,\rho_{\theta^*}) e^{-\theta^*x} V_j(\theta^*) = 0, \quad \mbox{$\bP_{0,i}$-a.s.}
    \end{equation}
    By \eqref{eq:derivative_cont_time_GWP}, \eqref{eq:additive_cont_time_GWP} and Lemma \ref{lem:renewal theorem}, we get that
    \begin{equation}\label{equation_lim_Cx_theta}
        \lim_{x\rightarrow\infty} \sum_{j\in\cI} c_{ren} x \#\widetilde{C}_{j}(x,\rho_{\theta^*}) e^{-\theta^*x} V_j(\theta^*) = Z_{\theta^*}^{(b)}(\infty), \quad \mbox{$\bP_{0,i}$-a.s.}
    \end{equation}

    Similarly to the arguments for supercritical speed regime,
    we know that  $(\#C_i(x,\rho_{\theta^*}): i\in \cI)_{x\ge 0}$ forms a supercritical continuous-time branching process where $x$ plays the role of time. Again, by the Kesten-Stigum theorem (\cite[Theorem 2.1]{Georgii03}), there is a non-negative vector $\pi^{\theta^*} = (\pi^{\theta^*}_i: i\in\cI)$ with $\langle \pi^{\theta^*}, 1 \rangle = 1$, such that for all $i\in\cI$, $\bP_{0,i}(\lim_{x\rightarrow\infty} \#C_j(x,\rho_{\theta^*})/\#C(x,\rho_{\theta^*}) = \pi^{\theta^*}_j\mid \mathscr{S}) = 1$. Therefore, 
    \begin{equation}\label{eqn:tildeC}
    \lim_{x\rightarrow\infty} \#\widetilde{C}_{j}(x,\rho_{\theta^*})/\#\widetilde{C}(x,\rho_{\theta^*}) = \pi^{\theta^*}_j, \quad \mbox{$\bP_{0,i}$-a.s.\ on $\gamma^{(-b,\rho_{\theta^*})}$.} 
    \end{equation} 
    Applying \eqref{eqn:tildeC} to \eqref{equation_lim_Cx_theta}, we deduce that,
        for all $j\in\cI$,
    \begin{equation}
         \lim_{x\rightarrow\infty} c_{ren} x \#\widetilde{C}(x,\rho_{\theta^*}) e^{-\theta^*x} \langle \pi^{\theta^*}, V(\theta^*) \rangle = Z_{\theta^*}^{(b)}(\infty), \quad \mbox{$\bP_{0,i}$-a.s.\ on $\gamma^{(-b,\rho_{\theta^*})}$.}
    \end{equation}
    Using \eqref{eqn:tildeC} again, we have, for $j\in\cI$,  
    \begin{equation}
         \lim_{x\rightarrow\infty} c_{ren} x \#\widetilde{C}_j(x,\rho_{\theta^*}) e^{-\theta^*x}  = \widetilde{\pi}_j Z_{\theta^*}^{(b)}(\infty), \quad \mbox{$\bP_{0,i}$-a.s.\ on $\gamma^{(-b,\rho_{\theta^*})}$.}
    \end{equation}
    where $\widetilde{\pi}_j = \pi^{\theta^*}_j/ \langle \pi^{\theta^*}, V(\theta^*) \rangle$. Combining this with \eqref{lim_Product_martingale_exists}, we have
    \begin{equation}
        \beta:= \lim_{x\rightarrow\infty} -\sum_{j\in\cI} x^{-1} e^{\theta^* x} \widetilde{\pi}_j \log \Phi_j(x,j)
    \end{equation}
    exists and is positive. 
    It follows that, $\bP_{0,i}$-a.s.\ on $\gamma^{(-b,\rho_{\theta^*})}$, 
       \begin{align}
        & \lim_{x\rightarrow\infty} \sum_{j\in\cI} \#\widetilde{C}_{j}(x,\rho_{\theta^*}) \log\Phi_{\theta^*}(z+x,j) \\
        &= \lim_{x\rightarrow\infty} \sum_{j\in\cI} c_{ren}^{-1} x^{-1} e^{\theta^*x} \widetilde{\pi}_j Z_{\theta^*}^{(b)}(\infty)  \log\Phi_{\theta^*}(z+x,j)\\
        &= c_{ren}^{-1} Z_{\theta^*}^{(b)}(\infty) e^{-\theta^* z} \lim_{x\rightarrow\infty} \sum_{j\in\cI}  \frac{x+z}{x}\frac{1}{x+z}e^{\theta^*(x+z)} \widetilde{\pi}_j   \log\Phi_{\theta^*}(z+x,j)\\
        &= -c_{ren}^{-1} Z_{\theta^*}^{(b)}(\infty) e^{-\theta^* z} \beta .
    \end{align}
    Recalling that $\bP_{0,i}(\gamma^{(-b,\theta^*)}\mid \mathscr{S}) \rightarrow 1$ as $b\rightarrow\infty$ and $\widetilde{M}_x(z,\rho_{\theta^*})$ given by \eqref{martingale_M_x_critical} is an $\mathcal{L}^1$-martingale,  we deduce by the the bounded convergence theorem that
    \begin{align}
        \Phi_{\theta^*}&(z,i) = \bE_{0,i} \left[\lim_{x\rightarrow \infty}\widetilde{M}_x(z,\rho_{\theta^*})\right] = \bE_{0,i} \bigg[ \exp \bigg\{ \lim_{x\rightarrow\infty} \sum_{j\in\cI} \# C_j(x,\rho_{\theta^*}) \log \Phi_{\theta^*}(z+x,j) \bigg\} \bigg]\\ 
        &=  \lim_{b\rightarrow\infty}  \bE_{0,i} \bigg[ \exp\bigg\{ \lim_{x\rightarrow\infty} \sum_{j\in\cI} \#C_{j}(x,\rho_{\theta^*}) \log\Phi_{\theta^*}(z+x,j) \bigg\} \ind{\gamma^{(-b,\rho_{\theta^*})}} +\ind{\mathscr{S}^c} \bigg]\\
        &= \lim_{b\rightarrow\infty} \bE_{0,i} \bigg[ \exp\bigg\{ \lim_{x\rightarrow\infty} \sum_{j\in\cI} \#\widetilde{C}_{j}(x,\rho_{\theta^*}) \log\Phi_{\theta^*}(z+x,j) \bigg\} \ind{\gamma^{(-b,\rho_{\theta^*})}} +\ind{\mathscr{S}^c} \bigg]\\
        &= \lim_{b\rightarrow\infty} \bE_{0,i} \left[ \exp\left\{ -c_{ren}^{-1} Z_{\theta^*}^{(b)}(\infty) e^{-\theta^* z} \beta \right\} \ind{\gamma^{(-b,\rho_{\theta^*})}}+\ind{\mathscr{S}^c} \right]\\
        &= \bE_{0,i} \left[ \exp\left\{ -\beta  Z_{\theta^*}(\infty) e^{-\theta^* z}\right\}  \right],
    \end{align}
    where we used the fact that both $\lim_{x\rightarrow\infty} \sum_{j\in\cI} \#C_{j}(x,\rho_{\theta^*}) \log\Phi_{\theta^*}(z+x,j)$ and $Z_{\theta^*}(\infty)$ are zero $\bP_{0,i}$ on $\mathscr{S}^c$.
    This completes the proof.
\end{proof}

\section{Proof of the spine decomposition theorem}\label{section:proof_spine}
\subsection{The spine decomposition with respect to the additive martingale}\label{section:subsec_proof_spine}
    We use the same notations in Section \ref{subsection_spine_decom} and give the proofs of the results in Section \ref{subsection_spine_decom}. First, we assume that each particle has at least one child and prove Theorems \ref{thm_spine_decomposition} under this assumption. 
    Then, we will prove these results allowing the possibility of no offspring when a particle dies.
    
    Intuitively, We can construct a probability measure $\bP^*_{(x,i)}$ on $\widetilde{\cF}_t$ by
    \begin{align}
        \dd \widetilde{\bP}_{x,i}(\tau, M, \xi) \big{|}_{\widetilde{\cF}_t} &= \dd\PP_{x,i}\left((X_{\xi},J_{\xi})_{t} \right) \dd L^{\beta(J_{\xi})} (\mathbf{n}_t) \prod_{v\prec \xi_t} \mu_{J_{\xi}(d_v)}(A_v)\\
        &\qquad \times\prod_{v\prec \xi_t} \bigg[ \frac{1}{A_v} \prod_{j:vj\in O_v} \dd \bP_{X_{\xi}(d_v), J_{\xi}(d_v)} ((\tau,M)_{t-d_v}^{v,j}) \bigg], \label{def_tildeP}
    \end{align}
    where 
    \begin{itemize}
        \item $\PP_{x,i}$ is the law of the Markov additive process $(X_{\xi}(t),J_{\xi}(t))$ with MAP triplet $((\phi_i)_{i\in\cI}, Q, G)$ starting from $(x,i)$, which gives the motion of the spine, and $(X_{\xi},J_{\xi})_t$ is short for $((X_{\xi}(s),J_{\xi}(s)), 0\leq s\leq t)$;
        \item  Recall that $\mathbf{n} = (n_t:t\geq 0)$ is the counting process of fission times along the spine, i.e.\ $n_s = |\xi_s|$ is the generation of $\xi_s$. We write $L^{\beta(J_{\xi})}$ for the law of a Poisson (Cox) process with rate $\beta(\Theta_t)\dd t$ and $\mathbf{n}_t$ is short for $(n_s:0\leq s\leq t)$.
        \item  $\mu_{J_{\xi}(d_v)}(A_v)$ is the probability that a particle with type $J_{\xi}(d_v)$ has an offspring of size $A_v$;
        \item  $\frac{1}{A_v}$ represents that we choose the spine uniformly and $O_v$ is the set of $v$'s children except the one in the spine;
        \item $(\tau,M)_{t-s}^{v,j}$ stands for the marked subtree rooted at $vj$ shifted by time $d_v$, and the subscript $t-s$ indicates that this time-shifted subtree evolves until time $t-s$. 
    \end{itemize} 

    We have defined in Lemma~\ref{lemma_martingale_zeta} that 
    \begin{equation}
            \zeta_t := \sum_{u\in\cN_t} \left(\prod_{v\prec u} A_v\right) e^{-\theta X_u(t) - \lambda(\theta) t} V_{J_u(t)}(\theta) \ind{\xi_t = u}.
    \end{equation}
    To prove that $(\zeta_t, t\geq 0, \widetilde{\bP}_{x,i})$ is a martingale, we proceeds by decomposing $\zeta_t$ into the product of three parts $\zeta_t^{(1)}$, $\zeta_t^{(2)}$ and $\zeta_t^{(3)}$, which will be defined sequentially as the argument develops. We also need the following definition from \cite{Liu11}.
    \begin{definition}\label{defn:mart-cond}
        Suppose that $(\Omega,\mathcal{H},P)$ is a probability space, $\{\mathcal{H}_t, t\geq 0\}$ is a filtration on $(\Omega,\mathcal{H})$ and $\mathcal{K}$ is a sub-$\sigma$-field of $\mathcal{H}$. A real-valued process $\{U_t, t\geq 0 \}$ on $(\Omega,\mathcal{H},P)$ is called a $P(\cdot|\mathcal{K})$-martingale with respect to $\{\mathcal{H}_t, t\geq 0\}$ if: 
        \begin{enumerate}
            \item[(i)]  It is adapted to $\{\mathcal{H}_t \vee \mathcal{K}, t\geq 0\}$;
            \item[(ii)]  For any $t\geq 0$, $E|U_t|<\infty$;
            \item[(iii)]  For any $t>s$,
        \begin{equation}
            E(U_t|\mathcal{H}_s\vee\mathcal{K}) = U_s, \quad \mbox{a.s.}
        \end{equation}
        \end{enumerate}
        We also say that $\{U_t, t\geq 0\}$ is a martingale with respect to $\{\mathcal{H}_t, t\geq 0\}$ given $\mathcal{K}$. 
    \end{definition}
    
    First, by \cite[Theorem 5.4]{Hardy09}, we have the following lemma.
    \begin{lemma}\label{lemma_martingale_zeta1}
        Suppose that, given the path of the type process $J_{\xi}$, $\mathbf{n} = (n_t: t\geq 0)$ is a Cox process on $\RR_+$ with intensity $\beta(J_{\xi}(t)) \dd t$ along the path of $J_{\xi}(t)$. Then, in the sense of Definition~\ref{defn:mart-cond}, 
        \begin{equation}
            \zeta_t^{(1)} := \prod_{v < \xi_t} m(J_{\xi}(d_v)) \cdot \exp\left\{ -\int_0^t((m-1)\beta)(J_{\xi}(s)) \dd s \right\},
        \end{equation}
        is an $L^{\beta(J_{\xi})}$-martingale with respect to the natural filtration $\{\mathcal{L}_t, t\geq 0\}$ of $\mathbf{n}$ given $\widetilde{\cG}$, where $((m-1)\beta)(i):=(m_i-1)\beta_i$, and $\widetilde{\cG}$ is defined in \eqref{eqn:Sigma_fields_G} as the $\sigma$-field generated by the positions and types of the spine.         
    
    Define a probability measure $L^{(m\beta)(J_{\xi})}$ by
        \begin{equation}
            \frac{\dd L^{(m\beta)(J_{\xi})}}{\dd L^{\beta(J_{\xi})}} \bigg{|}_{\mathcal{L}_t} := \prod_{v < \xi_t} m(J_{\xi}(d_v)) \cdot \exp\left\{ -\int_0^t((m-1)\beta)(J_{\xi}(s)) \dd s \right\}.
        \end{equation}
    Then $L^{(m\beta)(J_{\xi})}$ is the law of a  Cox process with intensity $(m\beta)(J_{\xi}(t))\dd t$. 
    \end{lemma}

    Recall that $\Xi_{\theta}(t)$ is defined by \eqref{def_Xi_theta}. Similarly, for any $u\in \cN_t$, we define
    \begin{equation}
        \Xi_{\theta}^{(u)}(t) := e^{ -\theta X_{u}(t) - \lambda(\theta) t + \int_0^t (\beta(m-1))(J_{u}(s)) \dd s } V_{J_{u}(t)}(\theta).
    \end{equation}
    
    \begin{lemma}\label{lemma_martingale_zeta2}
        Define
    \begin{equation}
        \zeta_t^{(2)} := \Xi_{\theta}^{(\xi_t)}(t) = e^{ -\theta X_{\xi}(t) - \lambda(\theta) t + \int_0^t (\beta(m-1))(J_{\xi}(s)) \dd s } V_{J_{\xi}(t)}(\theta), \qquad t\ge 0. 
    \end{equation}
    Then $(\zeta_t^{(2)}, t\geq 0)$ is a $\widetilde{\bP}_{x,i}$-martingale with respect to $(\widetilde{\mathcal{G}}_t, t\geq 0)$. 
    \end{lemma}
    \begin{proof}
    Recall that $W_{\theta} (t) = \sum_{u\in \cN_t} e^{-\left(\theta X_u(t)+\lambda(\theta) t\right)} V_{\JJ_u(t)}(\theta)$. By the many-to-one formula (for example, see \cite{Harris-Roberts17}), we have 
    \begin{align}\label{eq:W=Xi}
        \bE_{x,i} [W_{\theta}(t)] =  \EE_{x,i} \left[ e^{-\theta \chi_t-\lambda(\theta)t + \int_0^t (\beta(m-1))(\Theta_s) \dd s} V_{\Theta_t}(\theta)   \right] = \EE_{x,i} [\Xi_{\theta}(t)].
    \end{align}
    Note that $\bE_{x,i} [W_{\theta}(t)] = e^{-\theta x}V_i(\theta)$, hence we have $\EE_{x,i} [\Xi_{\theta}(t) ]= \Xi_{\theta}(0)$. Combining this with the Markov property of a MAP, we deduce that  $(\Xi_{\theta}(t), t\geq 0)$ is a $\PP_{x,i}$-martingale.  
    Since we read from \eqref{def_tildeP} that the law of $(X_{\xi}, J_{\xi})$ under $\widetilde{\bP}_{x,i}$ is $\PP_{x,i}$, the desired result follows.  
    \end{proof}   
    
    The next lemma follows from \cite[Theorem 5.5]{Hardy09}.  
    \begin{lemma}\label{lemma_martingale_zeta3}
        The process
        \begin{equation}
            \zeta_t^{(3)} := \prod_{v\prec  \xi_t} \frac{A_v}{m(J_{\xi}(d_v))}, \qquad t\ge 0,
        \end{equation}
        is a $\widetilde{\bP}_{x,i}(\cdot\mid\widehat{\cG})$-martingale with respect to $\{\widetilde{\cF}_t, t\geq 0\}$.
    \end{lemma}

    Summarizing, we check straightforwardly the identity 
    \begin{equation}
        \zeta_t =  \sum_{u\in\cN_t} \left(\prod_{v\prec u} A_v\right) e^{-\theta X_u(t) - \lambda(\theta) t} V_{J_u(t)}(\theta) \ind{\xi_t = u} = \zeta_t^{(1)} \zeta_t^{(2)} \zeta_t^{(3)}, \qquad t\ge 0.
    \end{equation}
    We are now ready to prove that $(\zeta_t, t\ge 0)$ is a martingale. 
    \begin{proof}[Proof of Lemma~\ref{lemma_martingale_zeta}]
        The proof is similar to that in  \cite[Lemma 2.7] {RS20}.
        $(\zeta_t^{(1)}, t\geq 0)$ is a $\widetilde{\bP}_{x,i}(\cdot|\widetilde{\cG})$-martingale with respect to $\{\widetilde{\cF}_t, t\geq 0\}$, and $(\zeta_t^{(3)}, t\geq 0)$ is a $\widetilde{\bP}_{x,i}(\cdot|\widehat{\cG})$-martingale with respect to $\{\widetilde{\cF}_t, t\geq 0\}$. Note that $\widetilde{\cG} \subset \widehat{\cG}$, and $\zeta_t^{(1)} \in \widehat{\cG}$, $\zeta_t^{(3)} \in \widetilde{\cF}_t$. By \cite[Lemma 2.3]{Liu11}, we have $(\zeta_t^{(1)} \zeta_t^{(3)}, t\geq 0)$ is a $\widetilde{\bP}_{x,i}(\cdot|\widetilde{\cG})$-martingale with respect to $\{\widetilde{\cF}_t, t\geq 0\}$. Note that $\zeta_t^{(2)} \in \widetilde{\cG}$, $\zeta_t^{(1)}\zeta_t^{(3)} \in \widetilde{\cF}_t$. Using \cite[Lemma 2.3]{Liu11} again, we get that $(\zeta_t, t\geq 0)$ is a $\widetilde{\bP}_{x,i}$-martingale with respect to $\{\widetilde{\cF}_t, t\geq 0\}$.
    \end{proof}

    \begin{lemma}
        $W_{\theta}(t)$ is the projection of $\zeta_t$ onto $\cF_t$, that is,
        \begin{equation}
            W_{\theta}(t) = \widetilde{\bP}_{x,i}(\zeta_t \mid \cF_t).
        \end{equation}
    \end{lemma}
    \begin{proof}
        Note that 
        \begin{equation}
            \zeta_t = \sum_{u\in\cN_t} \prod_{v\prec u} A_v e^{-\theta X_u(t) - \lambda(\theta) t} V_{J_u(t)}(\theta) \ind{\xi_t = u}.
        \end{equation}
        Therefore,
        \begin{align}
            \widetilde{\bP}_{x,i}(\zeta_t\mid \cF_t) &= \sum_{u\in\cN_t} \prod_{v\prec u} A_v e^{-\theta X_u(t) - \lambda(\theta) t} V_{J_u(t)}(\theta)  \widetilde{\bP}_{x,i}(  \ind{\xi_t = u} \mid \cF_t)\\
            &= \sum_{u\in\cN_t}  e^{-\theta X_u(t) - \lambda(\theta) t} V_{J_u(t)}(\theta),
        \end{align}
        where we used $\widetilde{\bP}_{x,i}(  \ind{\xi_t = u} \mid \cF_t) = \prod_{v\prec u} \frac{1}{A_v}$. This completes the proof.
    \end{proof}

    \begin{proof}[Proof of Theorem \ref{thm_spine_decomposition}]
    Recall that the probability measure $\widetilde{\bP}^{\theta}_{x,i}$ is defined by
     \begin{equation}
        \frac{\dd \widetilde{\bP}^{\theta}_{x,i}}{\dd \widetilde{\bP}_{x,i}} \bigg{|}_{\widetilde{\cF}_t} = \frac{\zeta_t}{\zeta_0}.
    \end{equation}   
    Then by \eqref{def_tildeP} we have 
    \begin{align}
        \dd \widetilde{\bP}^{\theta}_{x,i}(\tau, M, \xi)|_{\widetilde{\cF}_t} 
        =&\, \frac{\zeta_t^{(2)}}{\zeta_0} \dd\PP_{x,i}\left((X_{\xi}, J_{\xi})_t \right) \zeta_t^{(1)} \dd L^{\beta(J_{\xi})} (\mathbf{n}_t) \zeta_t^{(3)} \prod_{v\prec \xi_t} \mu_{J_{\xi}(d_v)}(A_v) \\
        &\,\prod_{v\prec \xi_t} \Bigg[ \frac{1}{A_v} \prod_{j:vj\in O_v} \dd \bP_{X_{\xi}(d_v), J_{\xi}(d_v)} ((\tau,M)_{t-d_v}^{v,j}) \Bigg]\\
        =&\, \dd\PP_{x,i}^{\theta}\left((X_{\xi}, J_{\xi})_t \right) \dd L^{(m\beta)(J_{\xi})} (\mathbf{n}_t)  \prod_{v\prec \xi_t} \frac{\mu_{J_{\xi}(d_v)}(A_v) A_v}{m(J_{\xi}(d_v))}\\
        &\,\prod_{v\prec \xi_t} \Bigg[ \frac{1}{A_v} \prod_{j:vj\in O_v} \dd \bP_{X_{\xi}(d_v), J_{\xi}(d_v)}((\tau,M)_{t-d_v}^{v,j}) \Bigg].  \label{eqn:decomposition of zeta}
    \end{align}
    Here we also use the  probability measure $\PP_{x,i}^{\theta}$ given by Lemma~\ref{lemma_spine_MAP}, whose proof is postponed to Section~\ref{sec:lem2.10};  in particular, under $\PP_{x,i}^{\theta}$ a MAP has  characteristics given by \eqref{MAP_new_characher}. 
    Then we read from \eqref{eqn:decomposition of zeta} the description of the particle system stated in the theorem.  
    \end{proof}

    Now we prove Theorem \ref{thm_spine_decomposition}, allowing the possibility of no offspring when a particle dies. Our proof follows the construction of the spine decomposition for branching Markov processes given in \cite{RS20}.
    
    We now require a slight modification to the definition of a marked tree with a distinguished spine. 
    Let $\dagger$ be a fictitious node not in $\tau$. Following the construction in \cite[Page 6]{RS20}, a spine $\xi$ on a marked tree $(\tau,M)$ is a subset of $\tau\cup\{\dagger\}$ such that
    \begin{itemize}
        \item $\varnothing\in\xi$ and $|\xi\cap(N_t\cup\{\dagger\})|=1$ for all $t\geq 0$.
        \item If $u\in\xi$ and $v\prec u$, then $v\in\xi$.
        \item If $u\in\xi$ and $A_u>0$, then there exists a unique $j=1,\cdots,A_u$ with $uj\in\xi$. If $u\in\xi$ and $A_u=0$, then $\xi\cap \cN_t$ is empty for all $t\geq d_u$. In this case, we will write $u=\dagger-1$.
    \end{itemize}
    Then we call $d_{\dagger-1}$ the ``lifetime" of the spine. Let $\xi_t := u$ be the unique element $u\in\xi\cap (N_t\cup\{\dagger\})$. Define $D_t:= \{u\in\tau: d_u\leq t, A_u = 0 \}$ be the set of particles that died, before or at time $t$, with no offspring. For a particle $u\in D_t$, two distinct cases arise, each requiring a separate treatment: either $m(J_u(d_u)) = 0$, or $m(J_u(d_u)) > 0$ but $A_u = 0$. Moreover, if $\xi_t = \dagger$, then there is a unique $u\in D_t$ such that $\dagger-1=u$. Similar to \cite[Equation (2.2)]{RS20}, we have
    \begin{align}
        \dd &\widetilde{\bP}_{x,i}(\tau, M, \xi) \big{|}_{\widetilde{\cF}_t}
        = \ind{\xi_t\in\tau} \dd\PP_{x,i}\left((X_{\xi}, J_{\xi})_t \right) \dd L^{\beta(J_{\xi})} (\mathbf{n}_t) \prod_{v\prec \xi_t} \mu_{J_{\xi}(d_v)}(A_v)\\
        &\quad\times\prod_{v\prec \xi_t} \Bigg[ \frac{1}{A_v} \prod_{j:vj\in O_v} \dd \bP_{X_{\xi}(d_v), J_{\xi}(d_v)}((\tau,M)_{t-d_v}^{v,j}) \Bigg] + \ind{\xi_t=\dagger} \dd\PP_{x,i}\left((X_{\xi}, J_{\xi})_t \right) \dd L^{\beta(J_{\xi})} (\mathbf{n}_t) \\
        &\quad\times \prod_{v\prec \dagger-1} \mu_{J_{\xi}(d_v)}(A_v) \prod_{v\prec \dagger-1} \Bigg[ \frac{1}{A_v} \prod_{j:vj\in O_v} \dd \bP_{X_{\xi}(d_v), J_{\xi}(d_v)}((\tau,M)_{t-d_v}^{v,j}) \Bigg].
    \end{align}
    In this case, we also define
    \begin{equation}
        \zeta_t := \sum_{u\in\cN_t} \bigg(\prod_{v\prec u} A_v\bigg) e^{-\theta X_u(t) - \lambda(\theta) t} V_{J_u(t)}(\theta) \ind{\xi_t = u}.
    \end{equation}
    Define
    \begin{align}
        \zeta_t^{(1)} :&= \prod_{v \prec \xi_t} m(J_{\xi}(d_v)) \cdot \exp\left\{ -\int_0^{t\wedge d_{\dagger-1}}((m-1)\beta)(J_{\xi}(s)) \dd s \right\}, \\
        \zeta_t^{(2)} :&= e^{ -\theta X_{\xi}(t\wedge d_{\dagger-1}) - \lambda(\theta) (t\wedge d_{\dagger-1}) + \int_0^{t\wedge d_{\dagger-1}} (\beta(m-1))(J_{\xi}(s)) \dd s } V_{\Theta_{t\wedge d_{\dagger-1}}}(\theta),\\
        \zeta_t^{(3)} :&= \prod_{v\prec  \xi_t} \frac{A_v}{m(J_{\xi}(d_v))}.
    \end{align}
   When $m(J_{\xi}(d_v)) = 0$, it holds that $A_v = 0$ a.s.\ and we use the convention $\frac{0}{0}=1$ such that $\frac{A_v}{m(J_{\xi}(d_v))} \ind{m(J_{\xi}(d_v))= 0} = \ind{m(J_{\xi}(d_v))= 0}$ a.s. In other words, we have
    \begin{align}
         \zeta_t^{(3)} &= \prod_{v\prec  \xi_t} \bigg(\frac{A_v}{m(J_{\xi}(d_v))}\ind{m(J_{\xi}(d_v))\ne 0} + \ind{m(J_{\xi}(d_v))= 0}\bigg) \\
         &  = \ind{\xi_t\in \cN_t} \prod_{v\prec  \xi_t} \frac{A_v}{m(J_{\xi}(d_v))}+ \ind{\xi_t = \dagger}\ind{m(J_{\xi}(d_{\dagger-1}))= 0} \prod_{v\prec  \dagger-1} \frac{A_v}{m(J_{\xi}(d_v))},
    \end{align}
    where the last equality follows from the fact that $A_{\dagger-1} = 0$ when $\xi_t = \dagger$ and $m(J_{\xi}(d_{\dagger-1})) > 0$.
    Therefore, if $\xi_t = \dagger$, then either $m(J_{\xi}(d_{\dagger-1})) = 0$ such that $\zeta^{(1)}_t =0$, or $m(J_{\xi}(d_{\dagger-1})) > 0$ such that $\zeta^{(3)}_t =0$. This yields that $\zeta_t= 0$ on $\{\xi_t = \dagger\}$.
    Then we have the identity
    \begin{equation}
        \zeta_t =  \zeta_t^{(1)} \zeta_t^{(2)} \zeta_t^{(3)}, \qquad t\ge 0. 
    \end{equation} 
    From the discussion about $\zeta_t^{(3)}$ above we deduce Lemma~\ref{lemma_martingale_zeta3} in this case. 
    According to \cite[Example 5.5.5]{Bremaud},  Lemma \ref{lemma_martingale_zeta1} still holds even when $m_j= 0$ for some $j\in\cI$. 
    By Lemma~\ref{lemma_spine_MAP}, it is immediate that $\{\zeta_t^{(2)},t\geq 0\}$ is a $\bP_{x,i}$-martingale with respect to $\widetilde{\cF}_t$.    
   We conclude that Lemma~\ref{lemma_martingale_zeta} holds. The definition of $\zeta_t$ yields that $\widetilde{\bP}_{x,i}^{\theta}(\xi_t\in \cN_t) = 1$ for any $t\geq 0$, which implies that $\widetilde{\bP}_{x,i}^{\theta}(\xi_t\in \cN_t, \forall t\geq 0) = 1$. Therefore,  
    \begin{align}
         \dd \widetilde{\bP}^{\theta}_{x,i}(\tau, M, \xi)|_{\widetilde{\cF}_t} 
        =&\, \dd\PP_{x,i}^{\theta}\left((X_{\xi}, J_{\xi})_t\right) \dd L^{(m\beta)(J_{\xi})} (\mathbf{n}_t)  \prod_{v\prec \xi_t} \frac{\mu_{J_{\xi}(d_v)}(A_v) A_v}{m(J_{\xi}(d_v))}\\
        &\,\prod_{v\prec \xi_t} \Bigg[ \frac{1}{A_v} \prod_{j:vj\in O_v} \dd \bP_{(X_{\xi}(d_v), J_{\xi}(d_v))}((\tau,M)_{t-d_v}^{v,j}) \Bigg].
    \end{align}
    Theorems \ref{thm_spine_decomposition} still holds.

    \subsection{Proof of Lemma \ref{lemma_spine_MAP}}\label{sec:lem2.10}
    \begin{proof}[Proof of Lemma \ref{lemma_spine_MAP}]
        We use the method of extended generator in \cite{PalRol2002}.
        Following \cite{PalRol2002} and the notations in Proposition \ref{prop:rep of MAP}, we first decompose $\chi_t = \chi_t^{(1)}+\chi^{(2)}_t$, where $\{\chi_t^{(1)}, t\ge 0\}$ and $\{\chi^{(2)}_t, t\ge 0\}$ are two independent processes, \[\chi_t^{(1)} := \sum_{n\ge 1}U^{n}_{\Theta(T_n^{-}), \Theta(T_n)}\ind{T_n\le t}\]
        is a pure jump continuous-time Markov process, and $\chi_t^{(2)}$ behaves in law as a L\'evy process with Laplace exponent $\phi_i$, when $\Theta(t)=i$. Thus, the process $(\chi_t^{(1)}, \Theta(t))$ has extended generator 
        \begin{align}
            \mathbf{A}^{(1)}f(x,i) &=\sum_{k\neq i}q_{ik}\int_{-\infty}^{\infty}(f(x+y,k)-f(x,i))\PP(U_{ik}\in \dd y)\\
            &= \sum_{k\in \cI} q_{ik}\int_{-\infty}^{\infty}f(x+y,k)\PP(U_{ik}\in \dd y),
        \end{align}
        with domain $\mathcal{D}(\mathbf{A}^{(1)})$ consisting of absolutely continuous functions for which the above integrals are finite. Notice that we have $q_{ii} = -q_i=-\sum_{k\neq i}q_{ik}$.
        For $i\in\mathcal{I}$, the L\'evy process with Laplace exponent $\phi_i$ has extended generator
        \[\mathbf{A}^{i}g(x) = \ta_ig'(x)+\frac{\sigma^2_i}{2}g''(x)+\int_{-\infty}^{\infty}\left(g(x+y)-g(x)-\ind{|y|<1} yg'(x)\right)\Lambda_i(\dd y),\]
        with domain $\mathcal{C}^2(\mathbb{R})\subset\mathcal{D}(\mathbf{A}^{i})$.
        
        Then the process $\left(\chi^{(2)}_t, \chi^{(1)}_t, \Theta_t \right)$ has extended generator $\mathbf{A}$, such that 
        $\forall g\in \mathcal{C}^2(\mathbb{R})\subset\mathcal{D}(\mathbf{A}^{i})$ and $\forall f\in \mathcal{D}(\mathbf{A}^{(1)})$, 
        \[\mathbf{A}(gf)(x,y,i) = g(x)(\mathbf{A}^{(1)}f)(y,i)+ f(y,i)(\mathbf{A}^{i}g)(x).\]

        Take $\widetilde{g}(x) := e^{-\theta x}, \widetilde{f}(y,i) := e^{-\theta y}V_i(\theta)$ and $\widetilde{h}(x,y,i) := \widetilde{g}(x)\widetilde{f}(y,i) = e^{-\theta x}e^{-\theta y}V_i(\theta)$, we have 
        \begin{align}
             (\mathbf{A}\widetilde{h})(x,y,i) &= e^{-\theta x}\sum_{k\in \mathcal{I}}q_{ik}\int_{-\infty}^{\infty}e^{-\theta(y+s)}V_k(\theta)\PP(U_{ik}\in \dd s)+ e^{-\theta y}V_i(\theta)\phi_i(\theta)e^{-\theta x}\\
             &=e^{-\theta(x+y)}\left([Q\circ G(\theta)\vec{V}(\theta)]_i+\phi_i(\theta)V_i(\theta)\right)\\
             &=e^{-\theta(x+y)}\big(\lambda(\theta)V_i(\theta)-\beta_i(m_i-1)V_i(\theta)\big)\\
             &=\left(\lambda(\theta)-\beta_i(m_i-1)\right)\widetilde{h}(x,y,i).
        \end{align}
        For $t\ge 0$, define
        \begin{align*}
         E^{\widetilde{h}}(t)&:=\frac{\widetilde{h}(\chi^{(2)}(t),\chi^{(1)}(t),\Theta(t))}{\widetilde{h}(\chi^{(2)}(0),\chi^{(1)}(0),\Theta(0))}\exp{\left(-\int_{0}^{t}\frac{(\mathbf{A}\widetilde{h}(\chi^{(2)}(s),\chi^{(1)}(s),\Theta(s))}{\widetilde{h}(\chi^{(2)}(s),\chi^{(1)}(s),\Theta(s))}\dd s\right)}  \\
        & = \frac{e^{-\theta \chi_t}V_{\Theta_t}(\theta)}{e^{-\theta x}V_i(\theta)}\exp{\left(-\int_0^t(\lambda(\theta)-\beta(m-1)(\Theta_s))\dd s\right)}=\frac{\Xi_{\theta}(t)}{\Xi_{\theta}(0)} .
        \end{align*}
        By \cite[Lemma 3.1]{PalRol2002}, $(E^{\widetilde{h}}(t), t\ge 0)$ is a $\mathbb{P}_{x,i}$-local martingale. Since $\mathbb{E}_{x,i}[E^{\widetilde{h}}(t)] = \mathbb{E}_{x,i}[\frac{\Xi_{\theta}(t)}{\Xi_{\theta}(0)}]=1$ by \eqref{eq:W=Xi}, then $E^{\widetilde{h}}$ is a true martingale.
        According to \cite[Lemma~4.1 and Theorem~4.2]{PalRol2002}, define the probability change
        \begin{equation}
        \frac{\dd \PP_{x,i}^{\theta}}{\dd \PP_{x,i}} \bigg{|}_{\mathcal{F}_t^{(\chi,\Theta)}} := \frac{\Xi_{\theta}(t)}{\Xi_{\theta}(0)},
    \end{equation}
        then under $\mathbb{P}^{\theta}_{x,i}$, $(\chi^{(2)},\chi^{(1)},\Theta)$ has extended generator
        \[\widetilde{\mathbf{A}}F := \frac{1}{\widetilde{h}}[\mathbf{A}(F\widetilde{h})-F\mathbf{A}\widetilde{h}],\]
        where $F\in \mathcal{D}(\widetilde{\mathbf{A}}) = \mathcal{D}(\mathbf{A})$.
        Take $F(x,y,i) = g(x)f(y,i)$ for any $g\in \mathcal{D}(\mathbf{A}^{i}), f\in \mathcal{D}(\mathbf{A}^{(1)})$. Recall that  $\widetilde{h}(x,y,i) =\widetilde{g}(x)\widetilde{f}(y,i)$,
        $\widetilde{g}(x)=e^{-\theta x}$ and
        $\widetilde{f}(y,i) = e^{-\theta y}V_i(\theta)$.
        Then 
        \begin{align}
            \widetilde{\mathbf{A}}F(x,y,i) &= \frac{1}{\widetilde{h}}\mathbf{A}(F\widetilde{h})(x,y,i)- F(x,y,i)\cdot(\lambda(\theta)-\beta_i(m_i-1))\\
            &= \frac{1}{\widetilde{f}\cdot\widetilde{g}} \left(g\widetilde{g}(\mathbf{A}^{(1)}(f\widetilde{f}))+ f\widetilde{f}(\mathbf{A}^{i}(g\widetilde{g}))\right) - fg(\lambda(\theta)-\beta_i(m_i-1)) \\
            &= g(x)\cdot I_1(y,i) + f(y,i)\cdot I_2(x,i)- f(y,i)g(x)(\lambda(\theta)-\beta_i(m_i-1)),
        \end{align}
        where
        \begin{equation}
            I_1(y,i) = \sum_{k\in\mathcal{I}}\frac{q_{ik}V_k(\theta)}{V_i(\theta)}\int_{-\infty}^{\infty}f(y+s,k)e^{-\theta s}\PP(U_{ik}\in \dd s)
        \end{equation}
        and 
        \begin{align}
            I_2(x,i) &= \ta_i(g'(x)-\theta g(x))+\frac{\sigma_i^2}{2}(g''(x)-2\theta g'(x)+\theta^2 g(x))\\
            &\quad+\int_{-\infty}^{\infty}\left(g(x+s)e^{-\theta s}-g(x)-\ind{|s|<1}s(g'(x)-\theta g(x)\right)\Lambda_i(\dd s).
        \end{align}
        Moreover, with notations in \eqref{MAP_new_characher}, we define $\widetilde{\mathbf{A}}_i$ and $\widetilde{\mathbf{A}}^{(1)}$ as follows: for any $g\in \mathcal{D}(\widetilde{\mathbf{A}}^{i})=\mathcal{D}(\mathbf{A}^{i}) $ and $ f\in \mathcal{D}(\widetilde{\mathbf{A}}^{(1)}) = \mathcal{D}(\mathbf{A}^{(1)})$,
        \begin{align*}          
        \widetilde{\mathbf{A}}^{i}g(x) &:= \widetilde{\ta}_ig'(x)+\frac{\widetilde{\sigma}^2_i}{2}g''(x)+\int_{-\infty}^{\infty}\left(g(x+y)-g(x)-\ind{|y|<1} yg'(x)\right)\widetilde{\Lambda}_i(\dd y),\\
        \widetilde{\mathbf{A}}^{(1)}f(x,i) &:=\sum_{k\neq i}\widetilde{q}_{ik}\int_{-\infty}^{\infty}(f(x+y,k)-f(x,i))\PP(\widetilde{U}_{ik}\in \dd y)\\
        &= \sum_{k\in \cI} \widetilde{q}_{ik}\int_{-\infty}^{\infty}f(x+y,k)\PP(\widetilde{U}_{ik}\in \dd y).
        \end{align*}
        Then by straightforward computation, we have $I_1(y,i) = \widetilde{\mathbf{A}}^{(1)}f(y,i)+(q_{ii}-\widetilde{q}_{ii})f(y,i)$, and $I_2(x,i) = \widetilde{\mathbf{A}}_i g(x)$.
        Therefore,
        \begin{align}
            &\hspace{1.5em}\widetilde{\mathbf{A}}F(x,y,i) \\
            &=g(x)(\widetilde{\mathbf{A}}^{(1)}f)(y,i)+f(y,i)(\widetilde{\mathbf{A}}^{i}g)(x)+f(y,i)g(x)((q_{ii}-\widetilde{q}_{ii})+\phi_i(\theta)-\lambda(\theta)+\beta_i(m_i-1))\\
            &=g(x)(\widetilde{\mathbf{A}}^{(1)}f)(y,i)+f(y,i)(\widetilde{\mathbf{A}}^{i}g)(x), 
        \end{align}
        where the last equality is deduced from the definition of PF eigenvector $\vec{V}(\theta)$ of $\mathcal{M}(\theta)$:
        \[
        \Big((q_{ii}-\widetilde{q}_{ii})+\phi_i(\theta)-\lambda(\theta)+\beta_i(m_i-1)\Big) V_i(\theta)= [\mathcal{M}(\theta)\vec{V}(\theta)]_i-\lambda(\theta)V_i(\theta)=0.
        \] 
        This implies that, under $\PP_{x,i}^{\theta}$, $(\chi,\Theta)$ is a MAP with characteristics given in the lemma.
    \end{proof}

\subsection{The spine decomposition with respect to the truncated derivative martingale}\label{sec:spine-2}
    We now prove Proposition~\ref{prop:spine-decomp-2}, the spine decomposition used in Section~\ref{subsec: L1_conv_Zb}. The proof is analogous to that for the additive martingale, but with the spine's movement now governed by a MAP conditioned to stay positive (non-negative). For simplicity, we assume that each particle has at least one child; the extension to allow extinction can be treated in a similar way as what we have done in the additive martingale case. 
    Here we fix $\theta = \theta^*$ and let $\widehat{X}_{\xi}(t) = \theta^* X_{\xi}(t)+\lambda(\theta^*)t$, $\widehat{\chi}_t = \theta^*\chi_t+\lambda(\theta^*)t$. 

    Recall the change of measure defined in \eqref{eqn:branching-Conditioned}. We have already explained that the spine $(\widehat{X}_{\xi}, J_{\xi})$ has the law of a MAP conditioned to stay positive given by \eqref{eq:MAP-cond}.

    We still define $\zeta^{(1)}$ as in Lemma~\ref{lemma_martingale_zeta1} and $\zeta^{(3)}$ as in Lemma~\ref{lemma_martingale_zeta3}.  Analogously to $\zeta^{(2)}$ in Lemma~\ref{lemma_martingale_zeta2}, we define 
    \begin{equation}
        \zeta_t^{(2)\uparrow} := R_{J_{\xi}(t)}(\widehat{X}_{\xi}(t)+b)\ind{\inf_{0\le s \le t}\widehat{X}_{\xi}(s) \ge -b}\zeta_t^{(2)}.
    \end{equation}
    where $\left(R_j(x), j\in \mathcal{I}, x\in\RR_+\right)$ is defined in \eqref{eq:R-mart}, as the renewal function of $(\widehat{X}_{\xi}, J_{\xi})$ under $\widetilde{\bP}_{x,i}^{\theta^*}$. 
    By Lemma~\ref{lemma_martingale_zeta2} and the definition of $\PP^{\uparrow}_{x,i}$ in \eqref{eq:MAP-cond}, the process $(\widehat{X}_\xi,J_\xi)$ under $\widetilde{\bP}^{\uparrow}_{x,i}$ has the same law as $(\chi, \Theta)$ under  $\PP^\uparrow_{x,i}$, i.e.\ it is a MAP conditioned to stay above $-b$. By Lemma~\ref{lemma_martingale_zeta2} again, $(\zeta_t^{(2)}, t\geq 0)$ is a $\widetilde{\bP}_{x,i}$-martingale with respect to $(\widetilde{\mathcal{G}}_t, t\geq 0)$, therefore $(\zeta_t^{(2)\uparrow}, t\geq 0)$ is a $\widetilde{\bP}_{x,i}^{\uparrow}$-martingale with respect to $(\widetilde{\mathcal{G}}_t, t\geq 0)$.

    Then we set  
    \begin{equation}
        \begin{aligned}
            \zeta_t^{\uparrow} &:= \zeta^{(1)}_t\zeta_t^{(2)\uparrow}\zeta^{(3)}_t\\
            &= \sum_{u\in\cN_t} \bigg(\prod_{v\prec u} A_v\bigg) R_{J_u(t)}(\widehat{X}_{\xi}(t)+b)\ind{\inf_{0\le s \le t}(\widehat{X}_{\xi}(s))\ge -b}e^{-\widehat{X}
            u(t)} V_{J_u(t)}(\theta^*) \ind{\xi_t = u}.
        \end{aligned}
    \end{equation}
    Then $(\zeta_t^{\uparrow}, t\geq 0, \widetilde{\bP}_{x,i})$ is a martingale with respect to $\{\widetilde{\cF}_t, t\geq 0\}$.
    So we can define
    \begin{equation}
        \left.\frac{\dd \widetilde{\bP}^{\uparrow}_{x,i}}{\dd \widetilde{\bP}_{x,i}}\right|_{\widetilde{\mathcal{F}}_t} := \frac{\zeta_t^{\uparrow}}{\zeta_0^{\uparrow}}.  
    \end{equation}
    Similarly as \eqref{eqn:decomposition of zeta}, we can decompose 
    \begin{align*}
        \dd \widetilde{\bP}^{\uparrow}_{x,i}(\tau, M, \xi)|_{\widetilde{\cF}_t}
        =&\, \zeta_t^{(2)\uparrow} \dd\PP_{x,i}\left((X_{\xi}, J_{\xi})_t \right) \zeta_t^{(1)} \dd L^{\beta(J_{\xi})} (\mathbf{n}_t) \zeta_t^{(3)} \prod_{v\prec \xi_t} \mu_{J_{\xi}(d_v)}(A_v) \\
        &\, \prod_{v\prec \xi_t} \bigg[ \frac{1}{A_v} \prod_{j:vj\in O_v} \dd \bP_{X_{\xi}(d_v), J_{\xi}(d_v)} ((\tau,M)_{t-d_v}^{v,j}) \bigg]\\
        =&\, \dd\PP_{x,i}^{\uparrow}\left((X_{\xi}, J_{\xi})_t \right) \dd L^{(m\beta)(J_{\xi})} (\mathbf{n}_t)  \prod_{v\prec \xi_t} \frac{\mu_{J_{\xi}(d_v)}(A_v) A_v}{m(J_{\xi}(d_v))}\\
        &\,\prod_{v\prec \xi_t} \bigg[ \frac{1}{A_v} \prod_{j:vj\in O_v} \dd \bP_{X_{\xi}(d_v), J_{\xi}(d_v)}((\tau,M)_{t-d_v}^{v,j}) \bigg].
    \end{align*}
    This matches the description the branching MAP under $\widetilde{\bP}^\uparrow_{x,i}$ in Proposition~\ref{prop:spine-decomp-2}. 

\bibliographystyle{abbrv}
\bibliography{MBAP}

\section*{Acknowledgement}
    This work is partially supported by 
   the National Key R\&D Program of China (grant 2022YFA1006500) and 
    National Natural Science Foundation of China (Grant Nos.\ 12288201 and 12301169). 
\end{document}